\newif\ifconfver
\confverfalse      

\newif\ifcutshort      
\cutshorttrue

\newif\ifcutshortlvltwo  
\cutshortlvltwofalse

\ifconfver
\documentclass[10pt,twocolumn,twoside]{IEEEtran}
\else
\documentclass[11pt, draftclsnofoot, onecolumn]{IEEEtran}
\fi

\usepackage{subfigure}
\usepackage{blkarray}
\usepackage{bigdelim}
\usepackage[ruled]{algorithm2e}
\usepackage{enumerate}

\usepackage{pifont}
\usepackage[dvipsnames]{xcolor}
\usepackage{cite}
\usepackage{float}
\usepackage{threeparttable}

\usepackage{amsmath,amsfonts,bm,amssymb,epsfig,psfrag,ifthen,color}
\usepackage{booktabs}
\usepackage{multirow}
\usepackage{array}
\usepackage{calc}
\usepackage{graphicx}
\usepackage{psfrag}

\usepackage{stfloats}
\usepackage{url}
\usepackage[caption=false]{subfig}
\usepackage{longtable}
\usepackage{amsmath,amsfonts,amssymb,amsbsy,bm,paralist,theorem,ifthen,color}
\usepackage{mathtools}

\usepackage{pifont}       
\usepackage{bbding}       
\usepackage{fontawesome} 
\usepackage{cleveref}

\ifconfver

\else

\fi

\def\VectorFont{\bf}
\newcommand{\vx}{{\VectorFont x}}
\newcommand{\vy}{{\VectorFont y}}
\newcommand{\vc}{{\VectorFont c}}

\newcommand{\vz}{{\VectorFont z}}
\newcommand{\vw}{{\VectorFont w}}
\newcommand{\vlambda}{{\boldsymbol{\lambda}}}

\usepackage{amsmath,lipsum}
\usepackage{cuted}
\usepackage{empheq}

\newtheorem{corollary}{Corollary}

\newtheorem{definition}{Definition}%
\newtheorem{assumption}{Assumption}
\newtheorem{remark}{Remark}%
\newtheorem{theorem}{Theorem}
\newtheorem{lemma}{Lemma}
\usepackage{caption}

\newtheorem{proof}{Proof}

\ifconfver

\else

\fi

\begin{document}

\bibliographystyle{IEEEtran}

\title{ {A Unified Inexact Stochastic ADMM for Composite Nonconvex and Nonsmooth Optimization}}

\ifconfver \else {\linespread{1.1} \rm \fi

\author{\vspace{0.8cm}Yuxuan~Zeng,	~Jianchao~Bai,
~Shengjia~Wang,
~and Zhiguo~Wang\\
\thanks{
Yuxuan Zeng, Zhiguo Wang, and Shengjia Wang are with College of Mathematics, Sichuan University, Chengdu 610064, China (e-mail: 2020222010085@stu.scu.edu.cn,~wangzhiguo@scu.edu.cn,~wsjia@stu.scu.edu.cn).}
\thanks{Jianchao Bai  (corresponding author) is with Research \& Development Institute of Northwestern Polytechnical University in Shenzhen, Shenzhen 518057, China;  School of Mathematics and Statistics,  Northwestern Polytechnical
University, Xi'an  710129,      China (e-mail: jianchaobai@nwpu.edu.cn).}
}

\maketitle

\vspace{-1.5cm}
\begin{center}
\today
\end{center}\vspace{0.5cm}

\begin{abstract}
In this paper, we propose a unified framework of inexact stochastic Alternating Direction Method of Multipliers (ADMM) for solving nonconvex problems subject to linear constraints, whose objective comprises an average of finite-sum smooth functions and a nonsmooth but possibly nonconvex function. The new framework is highly versatile. Firstly, it not only covers several existing algorithms such as SADMM, SVRG-ADMM, and SPIDER-ADMM but also guides us to design a novel accelerated hybrid stochastic ADMM algorithm, which utilizes a new hybrid estimator to trade-off variance and bias. Second, it enables us to exploit a more flexible dual stepsize in the convergence analysis. Under some mild conditions, our unified framework preserves $\mathcal{O}(1/T)$ sublinear convergence. Additionally, we establish the linear convergence under error bound conditions. Finally, numerical experiments
demonstrate the efficacy of the new algorithm for some nonsmooth and nonconvex problems.
\\\
\\\\
\end{abstract}


\ifconfver \else \IEEEpeerreviewmaketitle} \fi

\vspace{-0.5cm}
\section{Introduction}
\label{intro}
In the realm of machine learning (ML), exploiting the structural information of problems is crucial to enabling optimization at extreme scale. A prevalent example of such structure includes summation structures along with linear constraints,
thus giving rise to ADMM methods. 
ADMM \cite{boyd2011distributed,nishihara2015general} has gained widespread application and research attention attributed to its adeptness in fully exploiting this kind of structure and employing splitting techniques.
This work focuses on a class of nonconvex optimization problems as follows:
\begin{equation}\label{pro 1.1}
\min _{\mathbf{x}, \mathbf{y}} F(\mathbf{x}, \mathbf{y}):=f(\mathbf{x})+g(\mathbf{y}),  \quad \text { s.t. } A \mathbf{x}+B \mathbf{y}=\mathbf{b},
\end{equation}
where $f = \frac{1}{N}\sum_{1}^{N}f_{i} : \mathbb{R}^{n_x} \rightarrow \mathbb{R}$ is  nonconvex and smooth,
$N$ is the number of components which can be very large,
$g: \mathbb{R}^{n_y} \rightarrow \mathbb{R}$ is a locally Lipschitz continuous, possibly \textit{nonconvex and nonsmooth} function, frequently treated as a regularizer to prevent overfitting. Matrices 
$A \in \mathbb{R}^{m \times n_x}$ and $B \in \mathbb{R}^{m \times n_y}$
serve as linear operators encoding the model structure.
Problem (\ref{pro 1.1}) covers a broad range of applications
in statistical learning, compressive sensing, and machine learning \cite{goodfellow2016deep,zhu2020linear}, especially in neural networks, such as the graph-guided fused lasso \cite{kim2009multivariate} and the nonconvex problem with SCAD penalty \cite{gong2013general,xu2022inertial}. 

\subsection{Motivation and Related Work}
For large-scale optimization problems, deterministic ADMM is impractical due to an evaluation of the full gradient across all samples.
Consequently, stochastic versions of ADMM \cite{wang2012online,ouyang2013stochastic,suzuki2014stochastic} have been developed. 
Nevertheless, algorithms using stochastic gradients, given their signiﬁcant variance, adopt decaying step-sizes to ensure convergence, resulting in a slower convergence rate. 
Thanks to variance reduction (VR) techniques, some fast stochastic ADMM methods \cite{zheng2016fast,liu2020accelerated,BHZ22} with constant step-sizes have been developed, integrating VR and/or momentum acceleration \cite{Nesterov2004,AllenZhu2017KatyushaTF} techniques. These efficient methods have demonstrated improved convergence rates in both theoretical analyses and practical applications.

Previous discussions have primarily focused on convex problems, however, many applications in machine learning cannot be captured by convex models. Consequently, (stochastic) algorithms \cite{mai2020convergence,wang2021distributed,zhang2020global,guo2023preconditioned,cai2023cyclic} for addressing structured nonconvex optimization problems have been explored extensively. Recent works have also extended ADMM and its variations to such problems.
Guo's work \cite{GHWt18} has established the linear convergence of the nonconvex deterministic ADMM under the extra Kurdyka-Lojasiewicz (KL) condition. 
Substantial progress has been achieved in VR-type stochastic ADMMs.

\begin{table*}[t]
\caption{A comparison of stochastic ADMM for the nonconvex optimization problems.}
\label{table alg}
\vskip 0.10in
\begin{center}
\begin{small}
\begin{sc}
\begin{tabular}{lcccr}
\toprule
Algorithms & Nonsmooth term & dual step-size &Convergence rate& \\
\midrule
SVRG-ADMM \cite{zheng2016stochastic} & convex & $s=1$ & $\mathcal{O}(1/T)$\\
SAGA-ADMM \cite{huang2016stochastic} & convex & $s=1$ & $\mathcal{O}(1/T)$ \\
SPIDER-ADMM \cite{huang2019faster} & convex & $s=1$ & $\mathcal{O}(1/T)$ \\
SARAH-ADMM \cite{BLZhang21} & nonconvex & $s\in(0,1]$ & linear rate \\
AH-SADMM (ours) & nonconvex & $s\in(0,2)$ & $\mathcal{O}(1/T)$ / linear rate \\
\bottomrule
\end{tabular}
\end{sc}
\end{small}
\end{center}
\vskip -0.1in
\end{table*}

\begin{itemize}
\item The ﬁrst class of VR-type stochastic ADMM is based on SAG-estimator, including SAGA-ADMM variant. Theoretical analyses on these algorithms have been studied in \cite{huang2016stochastic}, demonstrating $\mathcal{O}(1/T)$ convergence rate. 
However, SAGA-type ADMM \cite{huang2016stochastic} requires substantial storage space of  $\mathcal{O}(Nn_{\mathbf{x}})$, posing a significant storage burden for large $N$.

\item The second one is SVRG-ADMM
\cite{huang2016stochastic,zheng2016stochastic}, accompanied by its accelerated variant ASVRG-ADMM \cite{10055828}. ASVRG-ADMM has also established its almost surely (a.s.) R-linear convergence under the KL condition.

\item The third class relies on SARAH \cite{nguyen2017sarah}, resulting in SPIDER-ADMM \cite{huang2019faster} and SARAH-ADMM \cite{BLZhang21}. 
\end{itemize}

Numerous researchers have delved into the theoretical analysis of existing algorithms.
For example, \cite{BLZhang21} has demonstrated the global linear convergence of VR-type stochastic ADMMs, including SAG, SAGA, SVRG, and SARAH. They highlighted that the recursive SARAH-ADMM method achieved relatively superior results in computed tomography (CT) reconstruction problems.
Moreover, recursive methods employing biased gradient estimators have captured increased attention in both theoretical and practical research. Their appeal lies in benefits such as low gradient storage, oracle complexity bounds, and efficient empirical performance. 
A comprehensive comparison of complexities associated with various stochastic ADMM was conducted in
\cite{huang2019faster}. 
This study revealed that the incremental ﬁrst-order oracle (IFO) complexity of the recursive SPIDER-ADMM is lower than that of SVRG-ADMM and SAGA-ADMM, and SPIDER-ADMM exhibits faster numerical convergence. 
Recently, \cite{pham2020proxsarah} introduced the ProxSARAH algorithm, integrating proximal operators with the SARAH gradient estimator to tackle composite nonconvex optimization problems.
This novel approach achieves the best-known complexity bound.

Algorithms of SVRG and SARAH types are both double-loop, with the batch size determined by the sample size $N$. 
Larger-scale problems may require a relatively larger batch size.
Additionally, the selection of snapshot points in these algorithms also matters.
To develop easily implemented methods,
\cite{tran2022hybrid} combined SARAH with an unbiased estimator, proposing the Proximal Hybrid SARAH-SGD (ProxHSGD) algorithm, which exhibits desirable advantages.
While hybrid estimators of ProxHSGD are biased, their properties of variance reduction and the ability to trade-off between variance and bias can be employed to develop new stochastic algorithms with better oracle complexity.
Moreover, this algorithm establishes convergence without setting checkpoints to compute full gradients as in SVRG and SARAH. 
Instead, it ensures convergence with both single sample and mini-batch, thereby reducing the oracle complexity bound and facilitating practical implementation.

Additional nonsmoothness alongside nonconvexity presents more theoretical challenges, preventing the use of general (sub)gradient-based methods.
Prior works \cite{GHWt18,BLZhang21,10055828} have theoretically analyzed ADMM for nonconvex and nonsmooth problems. Recently,  \cite{wang2019global} provided convergence guarantees for both exact and inexact \textit{deterministic} ADMM under specific update properties satisfied by the iteration sequence. These properties resemble a unified framework, helping the design of efficient algorithms based on distinct characteristics of subproblems.
\cite{BZZ2023} established the convergence of \textit{deterministic} ADMM with larger dual step-size under a unified framework. This work inspires the methods proposed in this paper.

\subsection{Contributions}
Developing efficient inexact ADMM with convergence for nonconvex problems has been a challenging task, especially when the objective is neither smooth nor convex. Recently, SGD with hybrid estimator \cite{tran2022hybrid} and nonconvex accelerated deterministic ADMM \cite{BZZ2023} with larger dual step-size have made some advances. Thus, it is intriguing to ask the following questions:
\begin{itemize}
\item \textit{Is it possible to analyze a class of stochastic inexact ADMM algorithms with larger dual step-size for nonconvex and nonsmooth problems under a unified framework?}

\item \textit{Is it possible to apply the hybrid technique to nonconvex stochastic ADMM, designing novel algorithms with linear convergence?}
\end{itemize}

The work in this paper aims to address these questions. To that end, we make the following contributions:

\textbf{Unified Inexact Stochastic ADMM (UI-SADMM).}\quad
To address a broad class of nonsmooth, nonconvex, and constrained problems, we develop a new and unified analysis framework for inexact stochastic ADMM and apply it to develop the innovative accelerated hybrid stochastic ADMM. This framework covers several well-known algorithms, such as SADMM, SVRG-ADMM, and SPIDER-ADMM. In addition, our analysis does not require convexity of the nonsmooth regularizer as in some existing works.

\textbf{Novel Accelerated Hybrid Stochastic ADMM.}\quad
We first adopt a ‘hybrid’ strategy, combining a biased stochastic VR-based gradient estimator with an unbiased one to form a new estimator for solving the $\mathbf{x}$-subproblem. This novel estimator can trade-off the variance and bias of the underlying estimators. Next, we integrate this hybrid estimator into accelerated (stochastic) ADMM \cite{BHZ22,BZZ2023}, developing a novel accelerated hybrid stochastic ADMM (AH-SADMM). This algorithm achieves convergence without check-points or $N\times n_{x}$-table to store gradient components.


\textbf{Linear Convergence Guarantee for Nonconvex and Nonsmooth Problems.}\quad 
When handling exact or linearized subproblems, the dual stepsize $s$, as in \cite{yang2017alternating,yashtini2022convergence}, is in the range $(0, \frac{1+\sqrt{5}}{2})$, larger than the general stepsize $s=1$. In this paper, we further extend the range of dual stepsize to $s\in(0,2)$ even with an inexact subproblem solution. 
Using a speciﬁcally constructed potential function,
we establish the best-known sublinear convergence rate of $\mathcal{O}(1/T)$, and the a.s. linear convergence for our UI-SADMM, under error bound conditions.

\subsection{Notations}
We use the symbol $\left\|\ .  \right\|$ and $\sigma_{min}\left(\cdot\right)$ to denote the Euclidean norm of a vector (or the spectral norm of a matrix) and the smallest eigenvalue of a matrix, respectively. 
Denote the $\sigma$-field generated by the random variables of the first $k$ iterations of Algorithm \ref{alg 1} as $\mathcal{F}_k$, and the expectation conditioned on $\mathcal{F}_k$ as $\mathbb{E}_k$. It's clear that  the iterate $\left(\mathbf{x}^k, \mathbf{y}^k, \mathbf{\lambda}^k \right)$
is $\mathcal{F}_k$-measurable since $\mathbf{x}^k$, $\mathbf{y}^k$ and $\mathbf{\lambda}^k$ are dependent on the random gradient information of the first $k$ iterations.
A set-valued mapping $F: \mathbb{R}^{n} \rightarrow \mathbb{R}^m$ is said to be outer semicontinuous at a point $\bar{\mathbf{x}}$, if there exist $\mathbf{x}^k \rightarrow \bar{\mathbf{x}}$ and $v^k \rightarrow v$ satisfying  $v^k \in F\left(\mathbf{x}^k\right)$, such that for any $v \in \mathbb{R}^m$, $v \in F(\bar{\mathbf{x}})$ holds. We use con $\mathcal{C}$ to denote a convex hull of  a given closed set $\mathcal{C}$, and the distance of $\mathbf{x}$ to the set $\mathcal{C}$ is denoted by dist($\mathbf{x},\mathcal{C}$)$:=\inf _{\mathbf{y} \in \mathcal{C}}\{\|\mathbf{x}-\mathbf{y}\|\}$.

\section{Prliminaries}
In the section, we recall some definitions and basic assumptions regarding the problem (\ref{pro 1.1}). 
We begin by some core definitions which are the backbone of theoretical analysis.

\begin{definition}[Clarke subgradient]\label{def clarke}
For a function $g: \mathbb{R}^n \rightarrow \mathbb{R}$ which is locally Lipschitz continuous on an open set $\mathcal{S} \subset \mathbb{R}^n$, let $\mathcal{C}$ be a subset of $\mathcal{S}$ such that $g$ is differentiable over the set $\mathcal{C}$. Then from Theorem 8.49 and Theorem 9.61 in \cite{rockafellar2009variational}, the Clarke subgradient set of function $g$ at a point $\bar{x} \in \mathcal{S}$ can be expressed as
$$
\partial g(\bar{x}):=\operatorname{con}\{v: \exists x \rightarrow \bar{x} \text { with } x \in \mathcal{C}, \nabla g(x) \rightarrow v\},
$$
which is nonempty, convex and compact for  $\forall \bar{x} \in \mathcal{S}$. In addition, $\partial g$ is outer semicontinuous and locally bounded on $\mathcal{S}$, we refer to \cite{rockafellar2009variational} for more details.
We further denote the set of critical points of $g$ by
$$
\operatorname{crit} g:=\left\{x \in \mathbb{R}^n: \operatorname{dist}(0, \partial g(x))=0\right\}.
$$
\end{definition}

\begin{definition}
Given accuracy $\epsilon\in(0,1)$,
the point $\left(\mathbf{x}^*, \mathbf{y}^*, \mathbf{\lambda}^*\right)$ is said to be an $\epsilon$-stationary point of the problem (\ref{pro 1.1}) if 
\begin{equation}
\begin{array}{r}
\mathbb{E}\left\|\nabla f\left(\mathbf{x}^*\right)-A^T \lambda^*\right\|^2 \leq \epsilon, \\
\mathbb{E}\left[\operatorname{dist}\left(B^T \lambda^*, \partial g\left(\mathbf{y}^*\right)\right)\right]^2 \leq \epsilon, \\
\mathbb{E}\left\|A \mathbf{x}^*+B \mathbf{y}^*-\mathbf{b}\right\|^2 \leq \epsilon.
\end{array}
\end{equation}
\end{definition}

Next, we give some fundamental assumptions. 

\begin{assumption}\label{assum 1}
\begin{itemize}
\item[(a)] The gradient mapping  $\nabla f$ is Lipschitz continuous, i.e. there exits a constant modulus $L>0$ such that
\begin{equation}
\left\|\nabla f\left(\mathbf{z}_1\right)-\nabla f\left(\mathbf{z}_2\right)\right\| \leq L\left\|\mathbf{z}_1-\mathbf{z}_2\right\|, \forall \mathbf{z}_1, \mathbf{z}_2 \in \mathbb{R}^{n_x}.
\end{equation}

\item[(b)] There exist $\sigma>0$ and batch size $M$ such that	
\begin{equation}
\mathbb{E}\left[\left\|\nabla f\left({\mathbf{x}},\xi_{M}\right) -
\nabla f\left({\mathbf{x}}\right)\right\|^2\right] \leq \frac{\sigma^2}{M},
\end{equation}
where the stochastic gradient estimator $\nabla f\left({\mathbf{x}},\xi_{M}\right)=\frac{1}{M} \sum_{i=1}^{M} \nabla f\left({\mathbf{x}},\xi_{i}\right),$ $\left\{\xi_{i}\right\}$
denotes a set of i.i.d. random variables which satisfy $\mathbb{E}[\nabla f\left({\mathbf{x}},\xi_{i}\right)]=\nabla f\left({\mathbf{x}}\right).$ 

\item[(c)] $($ Range $(B) \cup \mathbf{b}) \subseteq$ Range $(A)$, where Range($\cdot$) returns the image of any given matrix.
\end{itemize}
\end{assumption}

Assumption \ref{assum 1} (a) indicates the smoothness of function $f$. The bounded variance Assumption \ref{assum 1} (b) is standard for theoretical analysis. These assumptions are required for all existing 
stochastic gradient-based and VR-gradient-based methods.
Directly from the Assumption \ref{assum 1} (c), it's easy to derive  $\lambda^{k+1}-\lambda^k=-s \beta \mathbf{r}^{k+1} \in \operatorname{Range}(A)$, implying
$$
\left\|\lambda^{k+1}-\lambda^k\right\| \leq \sigma_A^{-\frac{1}{2}}\left\|A^{\top}\left(\lambda^{k+1}-\lambda^k\right)\right\|,
$$
where $\sigma_A$ is the smallest positive eigenvalue of $A^{\top} A$ (or equivalently the smallest positive eigenvalue of $A A^{\top}$ ). 
Especially, if $A$ is nonsingular or has a full column or full row rank, Assumption \ref{assum 1} (c) can be ensured.

\section{Inexact  Stochastic  ADMM and Convergence Analysis}\label{inexact sto admm}
In this section, we establish an  analysis of inexact stochastic ADMM within a unified framework for the problem (\ref{pro 1.1}). Initially,
we introduce the augmented Lagrangian (AL) $
\mathcal{L}_\beta(\mathbf{x}, \mathbf{y}, \boldsymbol{\lambda})$ corresponding to the problem (\ref{pro 1.1}),  formulated with a  penalty parameter $\beta>0$  and a dual multiplier $\boldsymbol{\lambda}$ as
\begin{equation}\label{AL function}
\begin{aligned}	
\mathcal{L}_\beta(\mathbf{x}, \mathbf{y}, \boldsymbol{\lambda})=
&f(\mathbf{x})+g(\mathbf{y})-\boldsymbol{\lambda}^{\top}(A \mathbf{x}+B \mathbf{y}-\mathbf{b})+\frac{\beta}{2}\|A \mathbf{x}+B \mathbf{y}-\mathbf{b}\|^2.
\end{aligned}
\end{equation}
Algorithms based on AL have been extensively studied. Among them, the classical ADMM algorithm optimizes in the following alternative order:
\begin{equation}
\left\{\begin{array}{l}
\mathbf{y}^{k+1} \in \arg \min _{\mathbf{y}} \mathcal{L}_\beta\left(\mathbf{x}^k, \mathbf{y}, \boldsymbol{\lambda}^k\right) \\
\mathbf{x}^{k+1} \in \arg \min _{\mathbf{x}} \mathcal{L}_\beta\left(\mathbf{x}, \mathbf{y}^{k+1}, \boldsymbol{\lambda}^k\right) \\
\boldsymbol{\lambda}^{k+1}=\boldsymbol{\lambda}^k-s \beta\left(A \mathbf{x}^{k+1}+B \mathbf{y}^{k+1}-\mathbf{b}\right),
\end{array}\right.
\end{equation}
where $s \in\left(0, 2 \right)$ denotes the stepsize of dual variable $\boldsymbol{\lambda}$.

\subsection{The Update Rule of $\mathbf{y}$}
We first proceed with the update for the variable $\mathbf{y}$ in step 2 of Algorithm \ref{alg 1}.
The proximal operator is used for updating $\mathbf{y}$,
and an appropriate matrix  $\mathcal{D}^{k}_y \succeq \mathbf{0}$ could be adaptively chosen, considering the structural characteristics of nonsmooth function $g$. 
The optimality condition for variable $\mathbf{y}$ generated at the $k$-th iteration implies:
\begin{align}\label{y update criteria}
&\frac{\beta}{2}\left\|\mathbf{y}^{k+1}-{\mathbf{y}}^{k}\right\|_{\mathcal{D}^{k}_y}^2+\mathcal{L}_\beta\left(\mathbf{x}^k, \mathbf{y}^{k+1}, \boldsymbol{\lambda}^k\right)
\leq \mathcal{L}_\beta\left(\mathbf{x}^k,{\mathbf{y}}^{k}, \boldsymbol{\lambda}^k\right),\nonumber\\
0 \in &\partial g\left(\mathbf{y}^{k+1}\right)+
\beta\left(A \mathbf{x}^{k}+B\mathbf{y}^{k+1}-\mathbf{b}-\frac{1}{\beta} {\boldsymbol{\lambda}}^k\right) + \beta \mathcal{D}^{k}_y \left(\mathbf{y}^{k+1} - \mathbf{y}^k \right),
\end{align}
where $\partial g\left(\cdot\right)$ represents the Clarke subgradient of $g(\cdot)$, as defined in Definition \ref{def clarke}.
Combining (\ref{y update criteria}) with $\xi_y^{k+1} \in \partial_{\mathbf{y}} \mathcal{L}_\beta\left(\mathbf{x}^k,{\mathbf{y}}^{k+1}, \boldsymbol{\lambda}^k\right)$ to yield
\begin{equation}\label{y-update grad}
\left\|\xi_y^{k+1}\right\| \leq c_y \beta\left\|\mathbf{y}^{k+1}-{\mathbf{y}}^{k}\right\|,
\end{equation}
where $c_y$ is a positive constant satisfying $ \left\|\  \mathcal{D}^{k}_y  \right\| \leq c_y$.

\subsection{The Inexact Update of $\mathbf{x}$}
In practice, the nonconvex function $f$ often lacks specific structure, and the substantial sample size $N$ poses a challenge in accurately solving the subproblem for updating the variable $\mathbf{x}$ using full-gradient methods.
To address this challenge, we introduce a versatile inexact stochastic ADMM Algorithm. The inexact update of $\mathbf{x}$ in the step 3 of Algorithm \ref{alg 1}  with a mini-batch size of $M$ is expressed as follows:
\begin{align}\label{x-update}
&\mathbf{x}^{k+1} \approx \arg \min _{\mathbf{x} \in \mathbb{R}^{n_x}} 
\left\langle\nabla f\left({\mathbf{x}}^k,\xi_{M}\right), \mathbf{x}-{\mathbf{x}}^k\right\rangle +\frac{\beta}{2}\left\|\mathbf{x}-{\mathbf{x}}^{k}\right\|_{\mathcal{D}_x^k}^2+
\frac{\beta}{2}\|A \mathbf{x}+B \mathbf{y}^{k+1}-\mathbf{b} - \frac{\mathbf{\lambda}^k}{\beta}\|^2,
\end{align}
where  $\mathcal{D}_x^k$ is a symmetric bounded positive definite matrix, and $\nabla f\left({\mathbf{x}}^k,\xi_{M}\right)$ denotes the stochastic gradient satisfying Assumption \ref{assum 1} (b).

To enhance the algorithm's generality, we have provided only the inexact criteria for updating $\mathbf{x}$ as in (\ref{x update criteria}),
where $\xi_x^{k+1}=\nabla_x \mathcal{L}_\beta\left(\mathbf{x}^{k+1}, \mathbf{y}^{k+1}, \boldsymbol{\lambda}^k\right)$, and $\hat{c}_x>0$, $c_x>0$.
Many inexact stochastic ADMMs adhere to the specified criterion. 
We discuss and summarize algorithms that satisfy this condition, including SADMM, SVRG-ADMM, and the recursive SPIDER-ADMM in Remark \ref{inexact cri} in Appendix \ref{app:remark 1}.  It's noticed that  
SPIDER, akin to SAGA but avoiding the storage gradient issue present in SAGA, surpasses its performance. Therefore, we exclude SAGA-type algorithms due to the limited spaces.

\subsection{The Update Rule of $\boldsymbol{\lambda}$}
The update rule for dual multiplier $\boldsymbol{\lambda}$ in step 4 of Algorithm \ref{alg 1} is as follows:
\begin{equation}
\boldsymbol{\lambda}^{k+1}=\boldsymbol{\lambda}^k-s \beta\left(A \mathbf{x}^{k+1}+B \mathbf{y}^{k+1}-\mathbf{b}\right).
\end{equation}
Research has expanded the range of dual stepsize to $(0, \frac{1+\sqrt{5}}{2})$, resulting in improved recovery capabilities \cite{yang2017alternating}. In this study, we further broaden the dual stepsize range to (0, 2). This enhanced flexibility in the dual step size not only improves numerical performance but also mitigates algorithm sensitivity while ensuring convergence.

Before establishing the following theoretical analysis, we introduce the following notations:  
$\mathbf{d}_x^k=\mathbf{x}^{k+1}-\mathbf{x}^k$,
$\mathbf{d}_y^k=\mathbf{y}^{k+1}-\mathbf{y}^k$, $ \mathbf{d}_\lambda^k=\boldsymbol{\lambda}^{k+1}-\boldsymbol{\lambda}^k$, and 
$\mathbf{r}^{k}=A {\mathbf{x}}^{k}+B \mathbf{y}^{k}-\mathbf{b}$.

\begin{empheq}[box=\fbox]{align}
\mathbb{E}_{k}\left[\mathcal{L}_\beta\left(\mathbf{x}^{k+1}
, \mathbf{y}^{k+1}, \boldsymbol{\lambda}^k\right)\right] 
&\leq
\mathcal{L}_\beta\left({\mathbf{x}}^k, \mathbf{y}^{k+1}, \boldsymbol{\lambda}^k\right)- 
\frac{\beta}{2} \mathbb{E}_{k}\left[ \left\| \mathbf{x}^{k+1}-{\mathbf{x}}^k\right\|_{\mathcal{D}^k_x}^2 \right] 
+ \frac{ (\hat{c}_x \beta)^2}{2} \frac{\sigma^2}{M},\nonumber\\
\mathbb{E}_{k}\left[\left\|\xi_x^{k+1}\right\|^2 \right]
&\leq (c_x \beta)^{2} \mathbb{E}_{k}\Big(\frac{\sigma^2}{M} + \left\|{\mathbf{x}}^k-\mathbf{x}^{k+1}\right\|^2+\left\|\mathbf{y}^{k+1}-\mathbf{y}^k\right\|^2\Big),\label{x update criteria}
\end{empheq}

\section{Convergence Analysis}
In the previous Section \ref{inexact sto admm}, we presented the framework of the generalized UI-SADMM, including its specific update rules. Subsequently, we will establish global convergence under proper conditions.

\begin{algorithm}[t]
\caption{Unified Inexact Stochastic ADMM}
\label{alg 1}
\KwIn{ $\beta >0, s \in (0,2),$
and initial values of $\mathbf{w}^0=\left(\mathbf{x}^0, \mathbf{y}^0, \boldsymbol{\lambda}^0\right);$}
\For{$k=0,1,2,...$}{
 1: Choose proper bounded matrices $\mathcal{D}_y^k \succeq \mathbf{0}$ and $\mathcal{D}_x^k \succeq \mathbf{0};$\\
 2: $
\begin{aligned}
\mathbf{y}^{k+1} = \text{argmin} _{\mathbf{y} } &\mathcal{L}_\beta\left(\mathbf{x}^k, \mathbf{y}, \boldsymbol{\lambda}^k\right)+\frac{\beta}{2}\left\|\mathbf{y}-{\mathbf{y}}^{k}\right\|_{\mathcal{D}_y^k}^2;
\end{aligned}$\\
3: Solve the following problem
$\begin{aligned}
\quad \mathbf{x}^{k+1} \approx  \arg \min _{\mathbf{x}} 
\left\langle\nabla f\left({\mathbf{x}}^k,\xi_{M}\right), \mathbf{x}-{\mathbf{x}}^k\right\rangle +
\frac{\beta}{2}\|A \mathbf{x}+B \mathbf{y}^{k+1}-\mathbf{b} - \frac{\mathbf{\boldsymbol{\lambda}}^k}{\beta}\|^2
+\frac{\beta}{2}\left\|\mathbf{x}-{\mathbf{x}}^{k}\right\|_{\mathcal{D}_x^k}^2
\end{aligned}$ inexactly such that the inexact criteria are satisfied;\\
 4: $\boldsymbol{\lambda}^{k+1}=\boldsymbol{\lambda}^k-s \beta\left(A \mathbf{x}^{k+1}+B \mathbf{y}^{k+1}-\mathbf{b}\right)$;}
5: {\bfseries Output:} Iterates $\mathbf{x}$ and $\mathbf{y}$ chosen uniformly random from $\left\{ (\mathbf{x}^k,\mathbf{y}^k)\right\}$
\end{algorithm}

\subsection{Global Convergence and Sublinear Convergence}


\begin{lemma}\label{x y w grad 0}
Let $\left\{\mathbf{w}^k:= (\mathbf{x}^k,\mathbf{y}^k,\mathbf{\lambda}^k)\right\}$ be the iterate satisfying the conditions 
(\ref{y update criteria}) and (\ref{x update criteria}).
Suppose  Assumption \ref{assum 1} (a), (b), (c) hold and AL function is bounded below, we can choose the parameters in Algorithm \ref{alg 1} such that
\begin{equation}
\left\{\begin{array}{l}
\frac{1+\tau}{s\beta\sigma_{A}}\psi_1(s)\left(2 L_{f}^{2}+4c_x^2\beta^2\right)
+\widehat{A}-\frac{\beta}{2} \sigma_{\min }\left(D_x^k\right) \leq-w \\
2 \widehat{A}-\frac{\beta}{2} \sigma_{\min }\left(D_y^k\right) \leq-w ,
\end{array}\right.
\end{equation}
where $w>0$,  
$\hat{A}= 4\frac{1+\tau}{s\beta\sigma_{A}}\psi_1(s)c_x^2\beta^2,$  $\sigma_{min}( \mathcal{D}^{k}_x )$ and
$\sigma_{min}( \mathcal{D}^{k}_y )$ are the smallest positive eigenvalue of $\mathcal{D}^{k}_x$ and $\mathcal{D}^{k}_y$, respectively.
We further denote 
\begin{equation}\label{def of potential}
\begin{aligned}
\mathcal{P}^k =& \mathcal{P}\left({\mathbf{x}}^{k}, \mathbf{y}^{k}, \boldsymbol{\lambda}^{k}\right) := \mathcal{L}_\beta\left(\mathbf{x}^k, \mathbf{y}^{k}, \boldsymbol{\lambda}^k\right)+\hat{A}\left\|\ \mathbf{d}_x^{k-1}  \right\|^2  +  \hat{A}\left\|\ \mathbf{d}_y^{k-1}  \right\|^2  + \frac{1+\tau}{s\beta\sigma_{A}}\psi_2(s) \left\|\  A\mathbf{d}_\lambda^{k-1} \right\|^2,
\end{aligned}
\end{equation}
where constant $\tau \in (0,1)$, $\psi_1(s)$ and $\psi_2(s)$ are defined in \cref{ATlambda}. Then the following statements hold:
\begin{align}
&\min_{k\in\left\{0,...,K\right\}} \Biggl\{ \mathbb{E}\left[\left\|\mathbf{d}_y^k\right\|^2\right] +\mathbb{E}\left[\left\|\mathbf{d}_x^k\right\|^2\right]+\mathbb{E}\left[\left\|\mathbf{d}_\lambda^k\right\|^2\right]\Biggr\}
\leq \frac{\Delta_{0,k}}{\mu(K+1)} + 
\frac{\mu_{\sigma,M}}{\mu}\frac{\sigma^2}{M}
,
\end{align}
where $\Delta_{0,k}=\mathbb{E}\mathcal{P}\left({\mathbf{x}}^{0}, \mathbf{y}^{0}, \boldsymbol{\lambda}^{0}\right) - \mathbb{E}\mathcal{P}\left({\mathbf{x}}^{k+1}, \mathbf{y}^{k+1}, \boldsymbol{\lambda}^{k+1}\right)$, $\mu = \min\left\{w, \frac{\tau}{s\beta}
\right\}$, $\mu_{\sigma,M} := 
\frac{ \widehat{c}_x^2 \beta^2}{2 }+\frac{8(1+\tau) \psi_1(s)c_x^2 \beta^2}{s \beta \sigma_A}$.
\end{lemma}

\begin{theorem}\label{thm:grad}
Supposing the conditions and Assumptions in \cref{x y w grad 0} hold, we have  	
\begin{equation}
\min_{1\leq k \leq K} \mathbb{E}\left[
\operatorname{dist}\left(\mathbf{0}, \partial \mathcal{L}_\beta\left(\mathbf{w}^k\right)\right)^2\right]= \mathcal{O}(1/T).
\end{equation}	
\end{theorem} 

Proofs of  \cref{x y w grad 0} and \cref{thm:grad} are deferred to \cref{app:thm1}.

\subsection{Linear Convergence Rate.}
This subsection is dedicated to establishing the local linear convergence of the iterative sequence $\left\{\mathbf{w}^k\right\}$ and the sequence of potential functions $\left\{\mathbb{E}\left\{ \mathcal{P}^k \right\}\right\}$ under specific assumptions. Denoting by $\Omega^*$ the set of all stationary points of the problem, i.e.,
$$
\begin{aligned}
&\Omega^*=\big\{\left(\mathbf{x}^*, \mathbf{y}^*, \lambda^*\right): A^{\top} \lambda^*=\nabla f\left(\mathbf{x}^*\right), \\
&B^{\top} \lambda^* \in \partial g\left(\mathbf{y}^*\right), A \mathbf{x}^*+B \mathbf{y}^*=\mathbf{b}\big\}.
\end{aligned}$$

\begin{assumption}\label{assump error bound}
(a)(\textbf{Error bound condition} \cite{luo1993error}) For any $\xi \geq \inf _{\mathbf{w}} \mathcal{L}_\beta(\mathbf{w})$, there exist $\epsilon>0$ and $\tau>0$ such that the inequality 
\begin{equation}
\operatorname{dist}\left(\mathbf{w}, \Omega^*\right) \leq \tau \operatorname{dist}\left(\mathbf{0}, \partial \mathcal{L}_\beta(\mathbf{w})\right)
\end{equation}
holds, whenever $\operatorname{dist}\left(\mathbf{0}, \partial \mathcal{L}_\beta(\mathbf{w})\right) \leq \epsilon$ and $\mathcal{L}_\beta(\mathbf{w}) \leq \xi$.

(b) The set $\Omega^*$ is nonempty and there exists a positive constant  $\omega^*$ such that $\left\|\mathbf{w}_1-\mathbf{w}_2\right\| \geq \omega^*$, whenever $\mathbf{w}_1$, $\mathbf{w}_2 \in \Omega^*$ and $F\left(\mathbf{x}_1, \mathbf{y}_1\right) \neq F\left(\mathbf{x}_2, \mathbf{y}_2\right)$.

(c) Function $g$ exhibits local weak convexity near
$$
\Omega_y^*:=\left\{\mathbf{y}: \text { there exist } \mathbf{x} \text { and } \boldsymbol{\lambda} \text { such that }(\mathbf{x}, \mathbf{y}, \boldsymbol{\lambda}) \in \Omega^*\right\},
$$
which implies the existence of $\varepsilon, \sigma, \delta>0$, for $\forall\mathbf{y}_1, \mathbf{y}_2$ with $\operatorname{dist}\left(\mathbf{y}_1, \Omega_y^*\right) \leq \epsilon, \operatorname{dist}\left(\mathbf{y}_2, \Omega_y^*\right) \leq$ $\epsilon$ and $\left\|\mathbf{y}_1-\mathbf{y}_2\right\| \leq \delta$ and for $\forall\nu \in \partial g\left(\mathbf{y}_2\right)$, the following holds:
\begin{equation}
g\left(\mathbf{y}_1\right) \geq g\left(\mathbf{y}_2\right)+\left\langle\nu, \mathbf{y}_1-\mathbf{y}_2\right\rangle-\sigma\left\|\mathbf{y}_1-\mathbf{y}_2\right\|^2 .
\end{equation}
\end{assumption}

See the remark of \cref{assump error bound} in \cref{app:remark of error bound}.
Now, we show that the linear convergence of the sequence $\left\{\mathbb{E}\left\{ \mathcal{P}^k \right\}\right\}$.

\begin{theorem}\label{thm:linear conv}
Suppose that conditions in \cref{x y w grad 0} hold.  
Let  $\left\{\mathbf{w}^k:= (\mathbf{x}^k,\mathbf{y}^k,\mathbf{\lambda}^k)\right\}$ be the iterates generated by Algorithm \ref{alg 1}. Then the following statements hold.
\begin{itemize}
\item[(i)]  $\lim _{k \rightarrow \infty} \operatorname{dist}\left(\mathbf{w}^k, \Omega^*\right)=0$ a.s.
\item[(ii)] There exist constants $\tilde{C}\in \left(0,1\right)$, $\breve{C}>0$ such that 
\begin{equation}\label{linear conv 1}
\mathbb{E} \mathcal{P}^k-F^* \leq \left(\tilde{C}
\right)^k \left(\mathbb{E} \mathcal{P}^0-F^*
\right)+ \breve{C}\frac{\sigma^2}{M}\quad \text{a.s.}
\end{equation}
\end{itemize}
\end{theorem}

If $\frac{\sigma^2}{M}=0$ (noiseless case), (\ref{linear conv 1}) reduces to
\begin{equation} \label{linear conv 2}
\mathbb{E} \mathcal{P}^k-F^* \leq \left(\tilde{C}
\right)^k \left(\mathbb{E} \mathcal{P}^0-F^*
\right)\quad \text{a.s.},
\end{equation}
which indicates that $\mathbb{E} \mathcal{P}^k$ a.s. converges to $F^*$ at a linear rate as the number of iteration $k$ goes to infinity. 


\section{Hybrid Stochastic Estimators for Inexact Update}

\begin{algorithm}[tb]
\caption{Accelerated Hybrid Stochastic Algorithm}
\label{alg 5.1}
\KwIn{ $\Theta>\Lambda, \tau=1-\sqrt{\frac{\Theta-\mu}{\Theta+\mu}},$ and initial values of $\widehat{\mathbf{x}}_0=\breve{\mathbf{x}}_1=\mathbf{x}_1=\mathbf{x}^k$ and $v_0:=\frac{1}{M} \sum_{\hat{\xi}_i \in \mathcal{M}} \nabla f\left(\widehat{\mathbf{x}}_0 ; \hat{\xi}_i\right)$;}
\For{$t=0,1,2,...,m$}{
1: $\beta_t=\max \left\{\bar{\beta}_t, \tau\right\}, \textrm{ where } \bar{\beta}_t=2 /(t+1);$\\
2: $\widehat{\mathbf{x}}_t=\beta_t \breve{\mathbf{x}}_t+\left(1-\beta_t\right) \mathbf{x}_t;$\\
3: $\text { Generate a sample pair }\left(\xi_t, \zeta_t\right) \text{ independently};$\\
4: Compute $ \tilde{v}_t $ as defined in (\ref{def of tilde v});\\
5: $\breve{\mathbf{x}}_{t+1}=\arg \min \left\{\left\langle \tilde{v}_t , \mathbf{x}\right\rangle+\frac{\gamma_t}{2}\left\|\mathbf{x}-\breve{\mathbf{x}}_t\right\|^2+\phi(\mathbf{x})\right\}$;\\
6: $\mathbf{x}_{t+1}=\beta_t \breve{\mathbf{x}}_{t+1}+\left(1-\beta_t\right) \mathbf{x}_t$;\\}
7: {\bfseries Output:} $\left(\widehat{\mathbf{x}}_{\bar{m}},\breve{\mathbf{x}}_{\bar{m}},\mathbf{x}_{\bar{m}} \right) $ chosen uniformly random from $\left\{\left(\widehat{\mathbf{x}}_i,\breve{\mathbf{x}}_i,\mathbf{x}_i \right)
\right\}^{m+1}_{i=1}.$
\end{algorithm}

In this section, we introduce an acceleration method that employs a novel hybrid stochastic gradient estimator to approximate a solution for the $\mathbf{x}$ subproblem (step 3 in \cref{alg 1}).
Applying this method to address step 5 of UI-SADMM results in the development of the novel Accelerated Hybrid SADMM (AH-SADMM).
This innovative hybrid stochastic gradient estimator combines the SARAH estimator and any unbiased one, aiming to achieve a balance between variance and bias.

\subsection{Inexact Update with Hybrid Stochastic Estimators}

Before delving into the details of the hybrid stochastic gradient and the updating step for the $\mathbf{x}$-subproblem, we define the hybrid stochastic estimators for the gradient of a smooth function $f$.
\begin{definition}
With two independent random variables $\xi_t$ and $\zeta_t$ satisfying 
$\mathbb{E}_{\xi_t}\left[ \nabla f\left(\widehat{\mathbf{x}}_t ; \xi_t\right) \right]= \nabla f\left(\widehat{\mathbf{x}}_t \right)$, 
$\mathbb{E}_{\zeta_t}\left[ \nabla f\left(\widehat{\mathbf{x}}_t ; \zeta_t\right) \right]= \nabla f\left(\widehat{\mathbf{x}}_t \right)$, $v_t$  is denoted as a hybrid stochastic estimator of gradient $\nabla f\left(\widehat{\mathbf{x}}_t \right)$ given the following quantity:
\begin{align}\label{hy sto grad}
v_t:=& \alpha_{t-1} v_{t-1}+\alpha_{t-1}\left(\nabla f\left(\widehat{\mathbf{x}}_t ; \xi_t\right)-\nabla f\left(\widehat{\mathbf{x}}_{t-1} ; \xi_t\right)\right) +\left(1-\alpha_{t-1}\right) \nabla f\left(\widehat{\mathbf{x}}_t ; \zeta_t\right).
\end{align}

\end{definition}
This definition reveals certain special cases. Specifically, when $\alpha_t=0$, the gradient estimator coincides with that of SGD and SVRG. 
Conversely, when $\alpha_t=1$, the gradient estimator transforms into the recursive SARAH estimator, yielding more updated information based on $\widehat{\mathbf{x}}_{t-1}$ and $v_{t-1}$ compared to SVRG, which relies on older snapshot point $\tilde{\mathbf{x}}$. 
These cases are summarized in the \cref{table 2}, using notations $\hat{\nabla} f\left(\widehat{\mathbf{x}}_t ; \zeta_t\right)=\nabla f\left( \tilde{\mathbf{x}} \right) + \nabla f\left(\widehat{\mathbf{x}}_t ; \zeta_t\right) - \nabla f\left(\tilde{\mathbf{x}} ; \zeta_t\right)$,  $\Delta_{t,t-1}^f=\nabla f\left(\widehat{\mathbf{x}}_t ; \xi_t\right)-\nabla f\left(\widehat{\mathbf{x}}_{t-1} ; \xi_t\right)$.
In this paper, we concentrate on the case $\alpha_t \in \left(\underline{\alpha},1\right), \underline{\alpha}>0$, which can be treated as a hybrid recursive stochastic estimator.

To analyze the specific
update of $\mathbf{x}$-subproblem, we begin by giving the necessary  definitions:
\begin{align} \Phi^k(\mathbf{x})  &:=h^k(\mathbf{x})+\phi^k(\mathbf{x}), \label{x-sub inner update} \\
h^k(\mathbf{x})
&=f(\mathbf{x})+\frac{\beta}{2}\left\|\mathbf{x}-\mathbf{x}^k\right\|_{\mathcal{D}_x^k},\nonumber\\
\phi^k(\mathbf{x}) &=\mathbf{x}^{\top} \mathbf{p}^k+\frac{\beta}{2} \| \mathbf{x}-\mathbf{x}^k \|^2_{A^{\top} A},\nonumber\\
\mathbf{p}^k &=-A^{\top}\left[\boldsymbol{\lambda}^k-\beta\left(A \mathbf{x}^k+B \mathbf{y}^{k+1}-\mathbf{b}\right)\right].\nonumber
\end{align}

\begin{table}[t]
\caption{Extreme cases of $\alpha_t$}
\label{table 2}
\vskip 0.10in
\begin{center}
\begin{small}
\begin{sc}
\begin{tabular}{lccr}
\toprule
Parameter $\alpha_t$ & Special choose & Estimator \\
\midrule
$\alpha_t=0$ & $v_t = \nabla f\left(\widehat{\mathbf{x}}_t ; \zeta_t\right)$ & SGD \\
$\alpha_t=0$  & $v_t = \hat{\nabla} {f}\left(\widehat{\mathbf{x}}_t ; \zeta_t\right)$ & SVRG \\
$\alpha_t=1$ & $v_t - {v}_{t-1}= \Delta_{t,t-1}^f$  & SARAH \\
\bottomrule
\end{tabular}
\end{sc}
\end{small}
\end{center}
\vskip -0.1in
\end{table}

With the introduction of a hybrid stochastic estimator for $\nabla f\left(\widehat{\mathbf{x}}_t \right)$ and the function $h^k$, we proceed to define the hybrid stochastic estimator for $\nabla h^k\left(\widehat{\mathbf{x}}_t \right)$ as follows:
\begin{align}
&\tilde{v}_t := \nabla h^k\left(\widehat{\mathbf{x}}_t ;\xi_{t}; \zeta_t\right)
:=  v_t + \beta \mathcal{D}_x^k (\widehat{\mathbf{x}}_{t}-\mathbf{x}^k),\nonumber\\
& \nabla h^k\left(\mathbf{x} ; \xi_t\right) = \nabla f\left(\mathbf{x} ; \xi_t\right) + \beta \mathcal{D}_x^k (\mathbf{x}-\mathbf{x}^k),\nonumber \\
&  \nabla h^k\left(\mathbf{x} ; \zeta_t\right) = \nabla f\left(\mathbf{x} ; \zeta_t\right) + \beta \mathcal{D}_x^k (\mathbf{x}-\mathbf{x}^k),\nonumber \\
&\tilde{v}_t:=\alpha_{t-1} \tilde{v}_{t-1}+\alpha_{t-1}\left(\nabla h^k\left(\widehat{\mathbf{x}}_t ; \xi_t\right)-\nabla h^k\left(\widehat{\mathbf{x}}_{t-1} ; \xi_t\right)\right)  +\left(1-\alpha_{t-1}\right) \nabla h^k\left(\widehat{\mathbf{x}}_t ; \zeta_t\right)\nonumber\\
&\quad = v_t + \beta \mathcal{D}_x^k (\widehat{\mathbf{x}}_t -\mathbf{x}^k).\label{def of tilde v}
\end{align}
Let
$u_t=\nabla h^k\left(\widehat{\mathbf{x}}_t ; \zeta_t\right)$. The stochastic gradient $u_t$ satisfies
$\mathbb{E}_{\zeta_t}\left[u_t\right]=\nabla h^k\left(\widehat{\mathbf{x}}_t\right)$ and 
$\mathbb{E}_{\zeta_t}\left[\left\|u_t-\nabla h^k\left(\widehat{\mathbf{x}}_t\right)\right\|^2\right]\leq  \sigma^2$. It is worth noting that $h^k$ satisfies the following Assumption:
\begin{assumption}\label{assum of func h}
There exist $\mu>0$  and $\Lambda>0$ such that
\begin{align}\label{h lip}
-\frac{\mu}{2}\left\|\mathbf{z}_1-\mathbf{z}_2\right\|^2 &\leq h^k\left(\mathbf{z}_2\right)-h^k\left(\mathbf{z}_1\right)-\left\langle\nabla h^k\left(\mathbf{z}_1\right), \mathbf{z}_2-\mathbf{z}_1\right\rangle \leq \frac{\Lambda}{2}\left\|\mathbf{z}_1-\mathbf{z}_2\right\|^2.
\end{align}
\end{assumption}

Since our focus is on solving the $\mathbf{x}$-subproblem within a fixed outer iteration count $k$,
we use concise notation, denoting $\Phi^k, h^k, \phi^k$ and $\Lambda^k$ as $\Phi, h, \phi$ and $\Lambda$, respectively. Our algorithm extends the accelerated gradient method proposed in \cite{BHZ22} for nonconvex problems, integrating a hybrid stochastic gradient.

\begin{remark}

\textbf{(Variance-bias trade-off \cite{tran2022hybrid})}
When $\alpha \in (0, 1)$, the bias of $\tilde{v}_t$ can be formulated as 
\begin{equation}\label{trade-off}
\begin{aligned}
&\operatorname{Bias}\left[\tilde{v}_t \mid \mathcal{F}_t\right]=  \left\|\mathbb{E}_{\left(\xi_t, \zeta_t\right)}\left[ \tilde{v}_t-\nabla h\left(\widehat{\mathbf{x}}_t\right)
\mid \mathcal{F}_t\right]\right\| \\
&=\alpha \left\| \tilde{v}_{t-1}-\nabla h\left(\widehat{\mathbf{x}}_{t-1}\right)
\right\| <\left\|\tilde{v}_{t-1}-\nabla h\left(\widehat{\mathbf{x}}_{t-1}\right)\right\| .
\end{aligned}
\end{equation}
Notably, the bias of $\tilde{v}_t$ is  smaller than that of $\tilde{v}_{t-1}$. The SARAH estimator, defined as
$v_{t}^{\mathrm{sarah}} := v_{t-1}^{\mathrm{sarah}}+\nabla h\left(\widehat{\mathbf{x}}_t ; \xi_t\right)
-\nabla h\left(\widehat{\mathbf{x}}_{t-1} ; \xi_{t}\right)$ with $\operatorname{Bias}\left[v_t^{\text {sarah }} \mid \mathcal{F}_t\right]=\left\|v_{t-1}^{\text {sarah }}-\nabla h\left(\widehat{\mathbf{x}}_{t-1}\right)\right\|$, suggests that the bias of hybrid estimator is smaller than that of SARAH estimator. 
The parameter $\alpha$ plays a crucial role in regulating the trade-off between bias and variance.
Further discussions can be found in \cref{app:hy sto}.
\end{remark}

\subsection{Accelerated Hybrid Update for Inexact Subproblem}
AH-SADMM integrates acceleration steps (Steps 2 and 6 in \cref{alg 5.1}) from ASADMM\cite{BHZ22} and introduces a novel hybrid gradient estimate $\tilde{v}_t$ for updating x (Step 5). This hybrid gradient estimate extends the recursive SARAH gradient estimator, enabling the simultaneous utilization of both recursive and the latest gradient information. 
The theoretical proof will be provided to demonstrate that the inexact solution obtained by \cref{alg 5.1} for the x subproblem adheres to the inexact criterion \eqref{x update criteria}.

\begin{figure}[b] 
\vskip -0.1in
\centering
\subfigure{\includegraphics[width=0.434\textwidth]{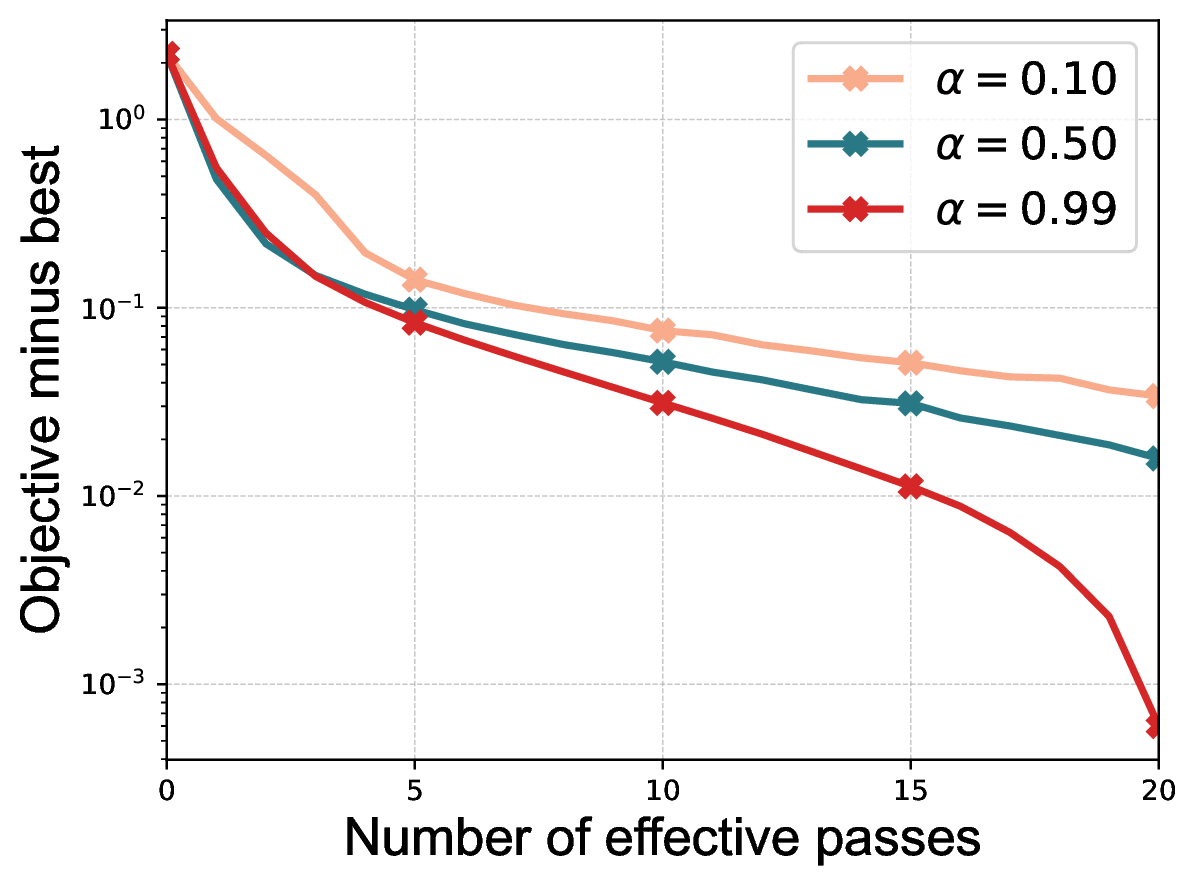}} 
\subfigure{\includegraphics[width=0.430\textwidth]{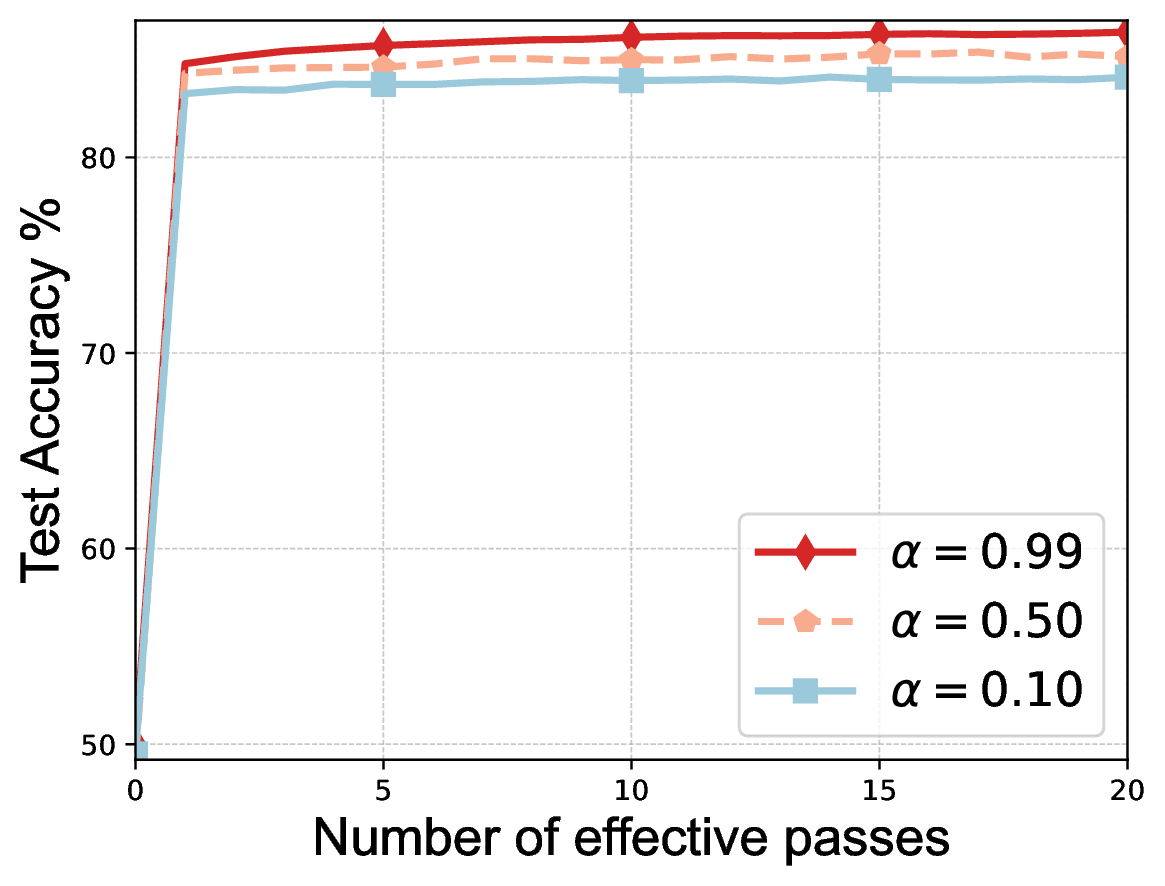}}
\caption{Comparison of different hybrid parameter $\alpha$ for solving the nonconvex problem (\ref{equ log with SCAD}).}
\label{fig:alpha}
\vskip -0.1in
\end{figure}

\begin{theorem}\label{thm 4.3}
Suppose that Assumption \ref{assum of func h} and conditions in \cref{x y w grad 0} hold. Let $\left\{\left(\widehat{\mathbf{x}}_t,\breve{\mathbf{x}}_t,\mathbf{x}_t \right) \right\}$ be generated by Algorithm \ref{alg 5.1}, we obtain
\begin{align}
&\frac{1}{m+1}\left[ \sum_{t=0}^{m}\frac{\beta_t}{2}
\mathbb{E}\left[ \left\|\breve{\mathbf{x}}_{t+1}-\mathbf{x}_t\right\|_{\mathcal{M}}^2\right]
- \sum_{t=0}^{m} \Gamma_t
\mathbb{E}\left[ \left\|\mathbf{s}_t\right\|^2\right]
\right]\nonumber\\
\leq 	&\frac{\mathbf{H}_{0}-\mathbf{H}_{m+1}}{m+1}+ \frac{\sigma^{2}}{2{r}(1+{\alpha})}(c_1+\frac{1}{c_1})\frac{1}{\sqrt{M(m+1)}},
\end{align}
where $\mathcal{M}=\beta A^{\top} A$, $\Gamma_t=\frac{\gamma_t \beta_t - \Lambda \beta_t^2}{2}$, $\mathbf{s}_t=\breve{\mathbf{x}}_{t+1}-\breve{\mathbf{x}}_t$, $r>0$ $c_1>0$, and the functions $\left\{\mathbf{H}_{i}\right\}$ are defined in \eqref{def of Hi}.
\end{theorem}

The above theorem establishes the sublinear convergence of the iterative subsequences. Based on this theorem, we will specify the settings for parameters $\alpha$, $\gamma_{t}$, $g(\alpha)$, and $M$ leading to the following corollary.

\begin{figure}[b] 
\vskip -0.1in
\centering
\subfigure{\includegraphics[width=0.440\textwidth]{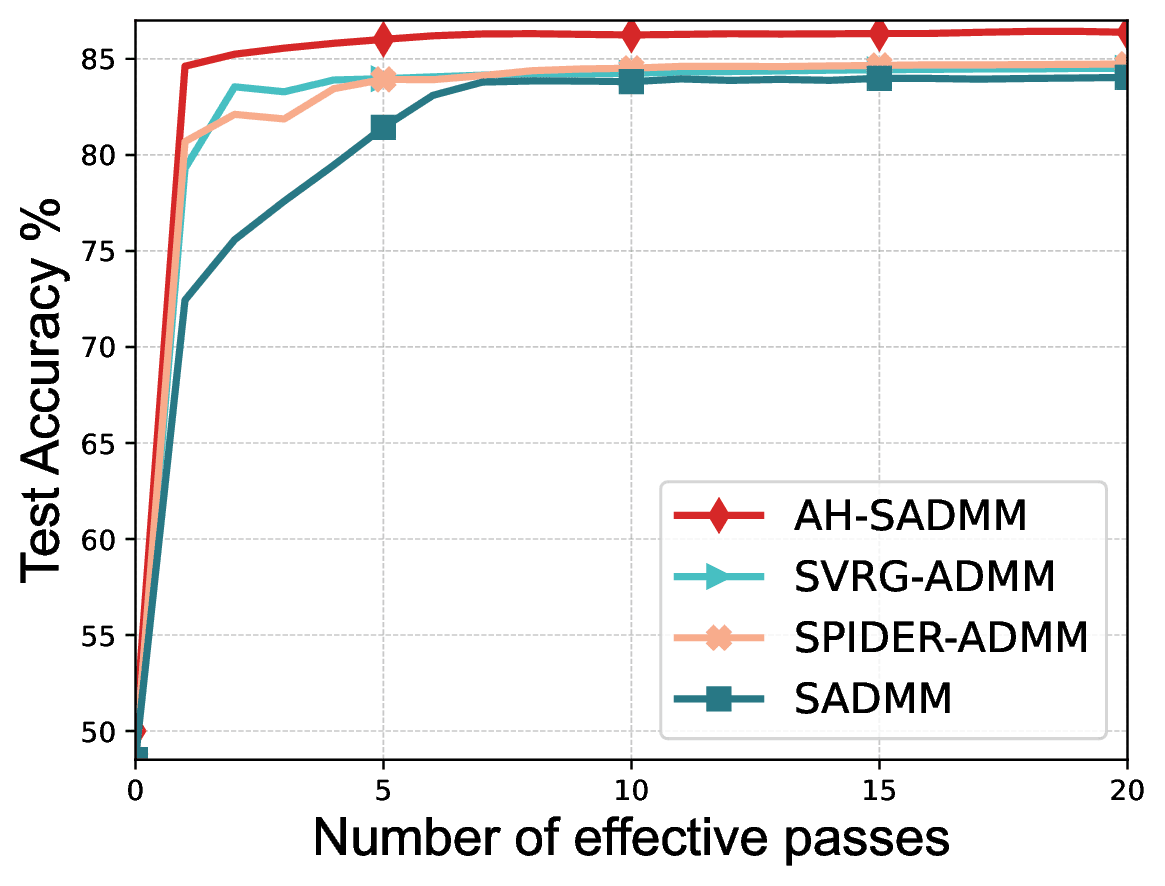}} 
\subfigure{\includegraphics[width=0.434\textwidth]{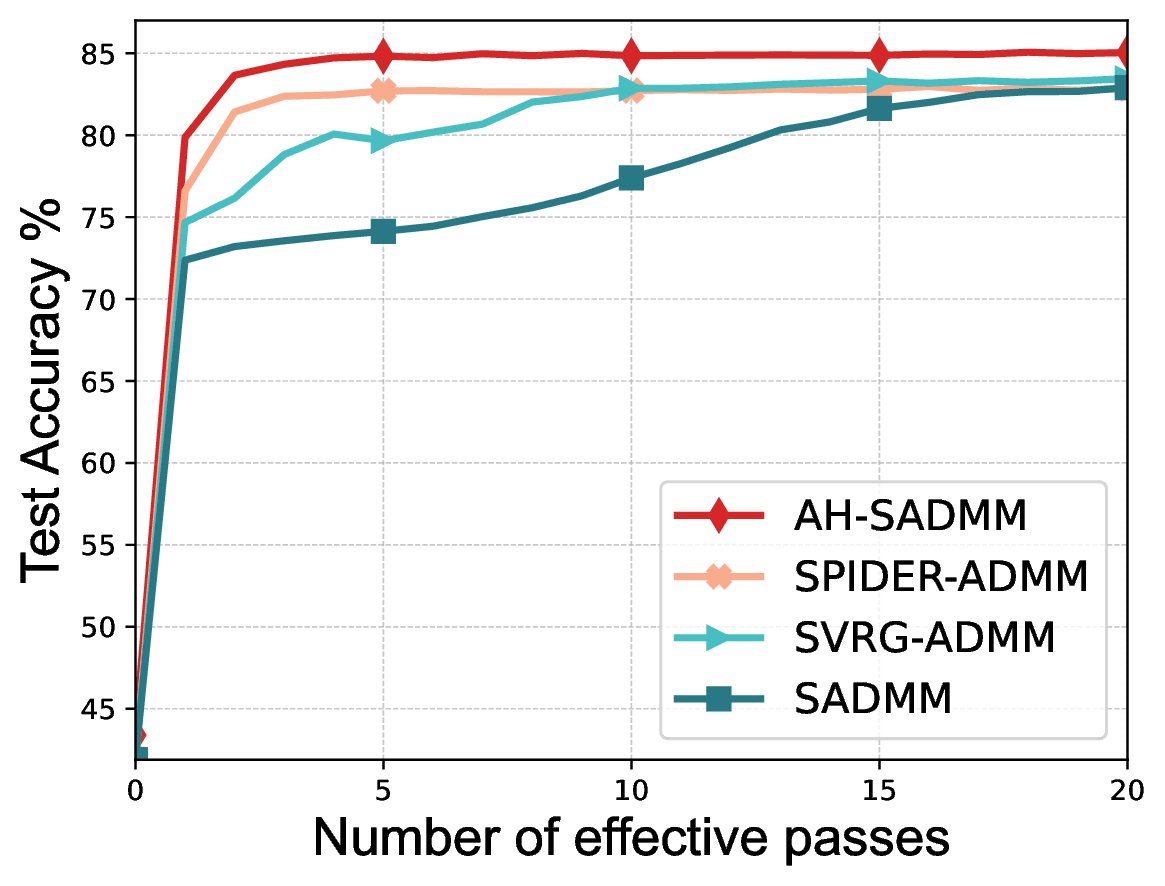}}
\caption{Test accuracy of \eqref{equ log with SCAD} on ijcnn1 (left) and a9a (right).}
\label{fig:ijcn1 acc}
\vskip -0.1in
\end{figure}

\begin{corollary}\label{cor 5.5}
Suppose that Assumption \ref{assum of func h} and conditions in \cref{thm 4.3} hold. Let $\left\{\left(\widehat{\mathbf{x}}_t,\breve{\mathbf{x}}_t,\mathbf{x}_t \right) \right\}$ be generated by Algorithm \ref{alg 5.1} using the following constant hybrid parameter $\alpha$ and step-size $\gamma_t$:
\begin{equation}\label{alpha setting}
\begin{cases}
&\alpha=1-\frac{c_1}{\sqrt{M(m+1)}},\\
&\gamma_{t} = \beta_{t}(\frac{\mu + 2g(\alpha)}{2\tau -\tau^2} -\mu) \frac{t+1}{t},\\
&g(\alpha)=\alpha \Lambda \sqrt{\frac{2l_3 }{  (1-\alpha^2)}}\\
& M=\sigma^{k_1}(m+1)^{k_2};\quad k_1>0, k_2 \in (0,1). 
\end{cases}
\end{equation}
Then, we have
\begin{align}
&\lim _{m \rightarrow \infty}	\mathbb{E} \left\|\ \nabla \Phi\left(\widehat{\mathbf{x}}_{\bar{m}}\right) \right\|^2
=0.\nonumber
\end{align}
\end{corollary}
\begin{remark}\label{remark of alpha} 
When the inner iteration $m$ is large, $\alpha$ is very close to 1, showing that the biased term dominates the unbiased one in the hybrid estimator $\tilde{v}_t$. 
This observation is consistent with \cref{fig:alpha}, indicating that AH-SADMM achieves superior performance as the hybrid parameter 
$\alpha$ approaches 1. 
Moreover, since $\nabla \Phi(\mathbf{x})=\nabla_x \mathcal{L}_\beta\left(\mathbf{x}, \mathbf{y}^{k+1}, \lambda^k\right)+\beta \mathcal{D}_x^k\left(\mathbf{x}-\mathbf{x}^k\right)$ and $\lim _{t \rightarrow \infty} \mathbb{E}
\left[\nabla \Phi\left({\mathbf{x}}_{\bar{m}}\right)\right]=0,$ the inexact criterion (\ref{x update criteria}) is satisfied by setting $\widehat{\mathbf{x}}^k=\widehat{\mathbf{x}}_{\bar{m}}$ for all $t$ sufficiently large.
\end{remark}
The proofs for \cref{thm 4.3} and \cref{cor 5.5} are provided in \cref{app: pf of inner grad} and \cref{app:cor 5.5}.

\section{Numerical Experiments}
In this section, we provide two nonconvex examples to illustrate the performance of the proposed AH-SADMM and compare it with several state-of-the-art methods, including SADMM, Accelerated SADMM (ASADMM), SVRG-ADMM, and SPIDER-ADMM.  
Given comprehensive comparisons between SPIDER and SAGA (storage gradient-type algorithms) in \cite{huang2019faster, BLZhang21,tran2022hybrid}, with SPIDER outperforming SAGA, we leave further comparisons of SAGA-type algorithms.
Specifically, the variant of AH-SADMM without acceleration ($\beta_t =1$), relying solely on hybrid gradients, is denoted as H-SADMM. The choice of $\alpha_t$ follows \eqref{alpha setting}.

\subsection{Nonconvex Binary Classification problem}
We consider the binary classification problem with a SCAD penalty term:
\begin{equation}\label{equ log with SCAD}
\min_{\mathbf{x}}\frac{1}{n}\sum^{n}_{i=1} f_{i}(\mathbf{x})+\lambda_1
\sum_{j=1}^{n_y} p_{\kappa}\left(\left|(A\mathbf{x})_j\right|\right),
\end{equation}
where the nonconvex sigmoid loss function $f_{i}$ is as defined by $f_{i}(\mathbf{x})= \frac{1}{1+\exp \left(b_{i} a_{i}^{T} \mathbf{x}\right)}$, with the set of training samples
$\left\{\left(a_i, b_i\right)\right\}_{i=1}^{N}$. And the nonconvex SCAD penalty $p_{\kappa}\left(\cdot\right)$ is defined in \eqref{def of SCAD}. 
The given matrix $A$ decodes the sparsity pattern of the graph, obtained by sparse inverse covariance estimation {\cite{FHT08}}. To address \eqref{equ log with SCAD}, we introduce an additional primal variable $\mathbf{y}$ with the constraint $A\mathbf{x}=\mathbf{y}.$
Setting $\lambda_1=10^{-5}$, 
more details of experiments can be found in \cref{add for exps}.

\begin{figure}[h] 
\vskip -0.1in
\centering
\subfigure[ijcnn1]{\includegraphics[width=0.232\textwidth]{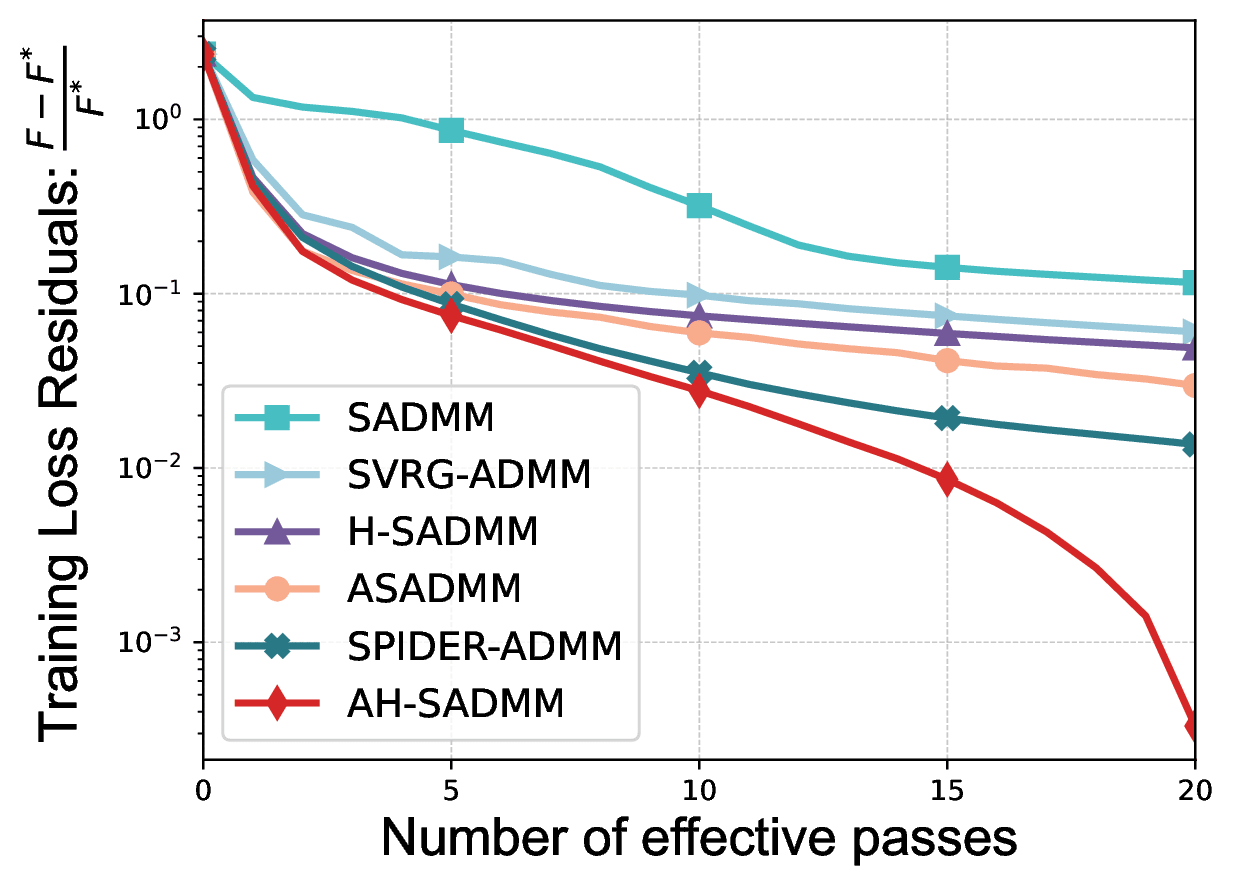}} 
\subfigure[w8a]{\includegraphics[width=0.232\textwidth]{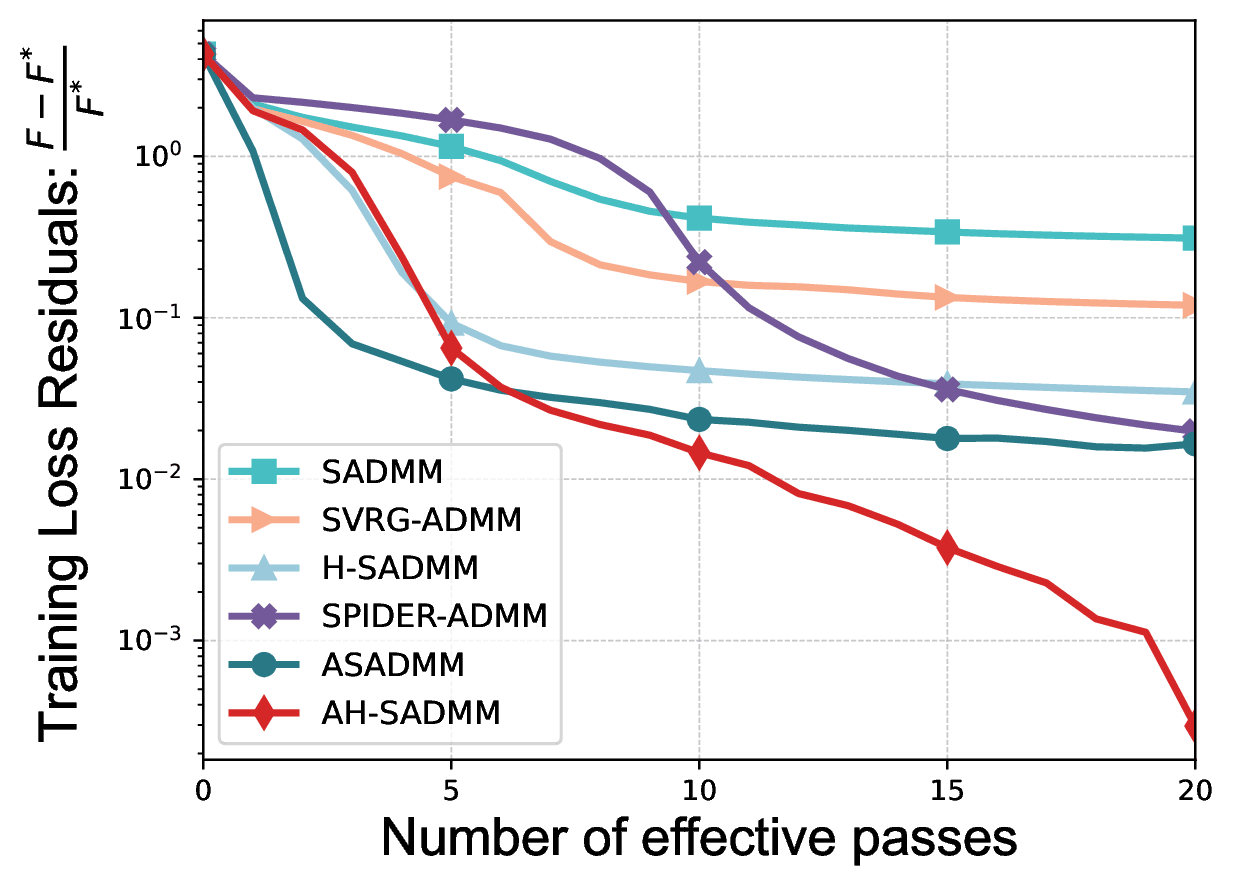}}
\subfigure[covtype.binary]{\includegraphics[width=0.242\textwidth]{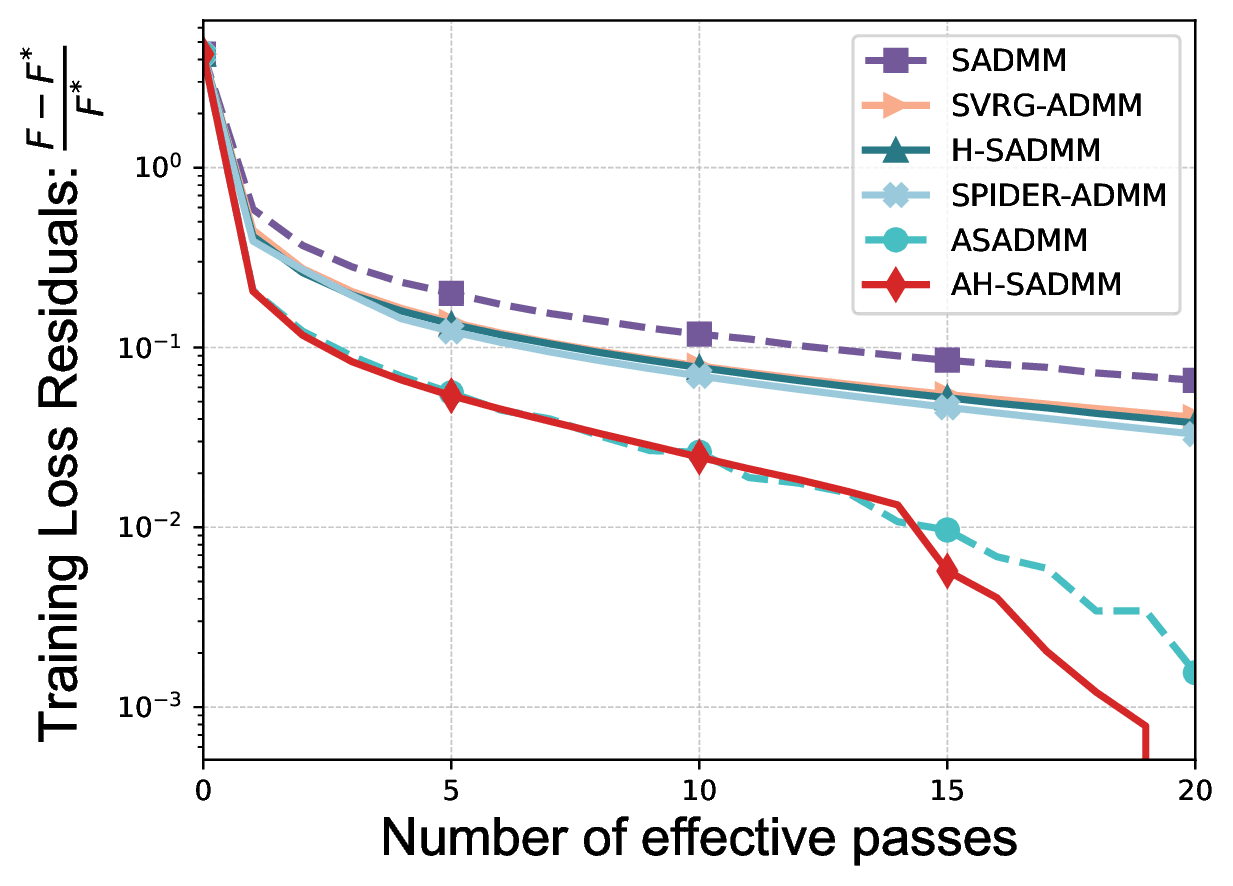}}
\subfigure[a9a]{\includegraphics[width=0.234\textwidth]{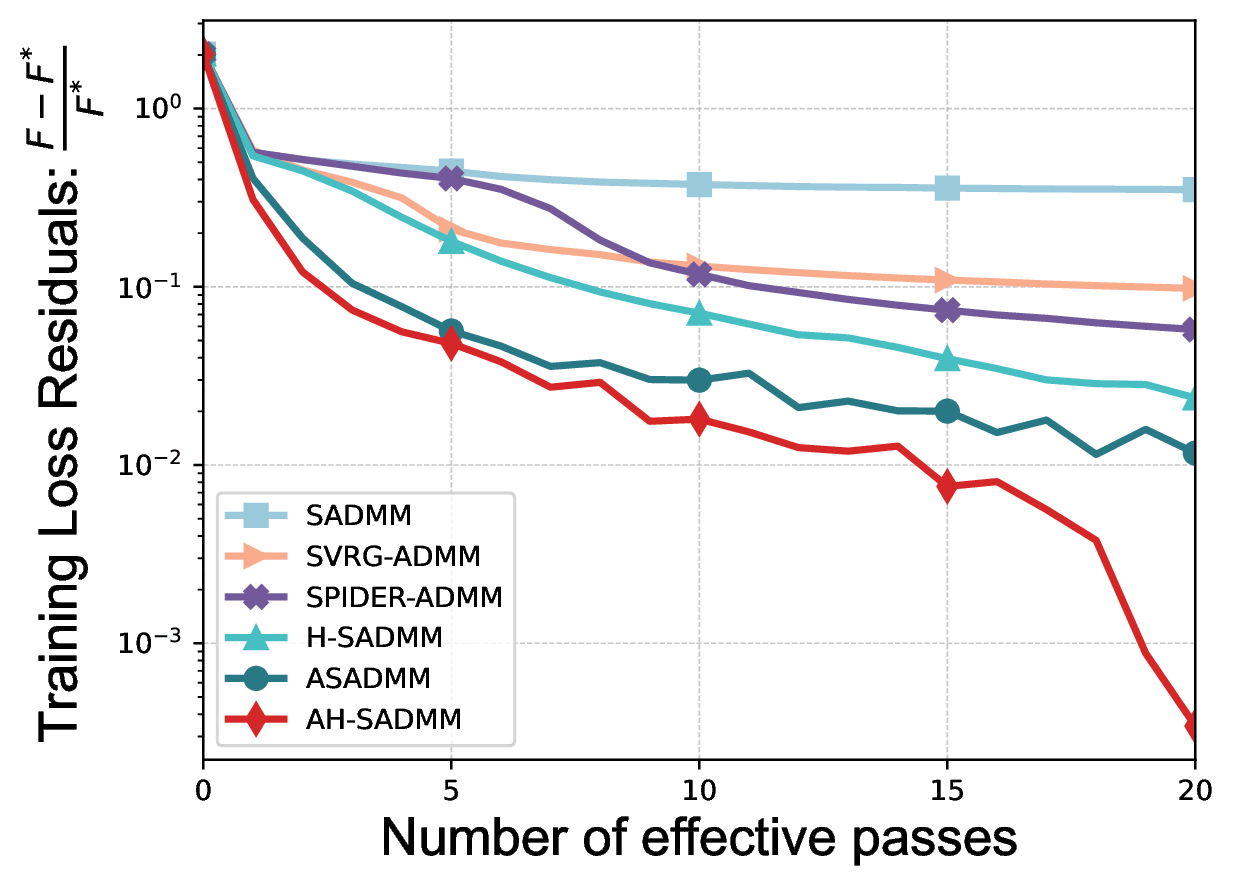}}
\caption{The training loss of \eqref{equ log with SCAD} on some real datasets.}
\label{fig:ijcn1}
\vskip -0.1in
\end{figure}

We conducted experiments on several publicly available datasets from LIBSVM. \cref{fig:ijcn1 acc} and \cref{fig:ijcn1} show that algorithms with VR techniques exhibit faster convergence with better generalization than SADMM.
This observation is consistent with \cref{thm:linear conv} ( the vanishing variance resulting in linear convergence). The results in \cref{fig:ijcn1} also indicate that algorithms employing momentum acceleration techniques, such as ASADMM and AH-SADMM, contribute to a more rapid convergence.
Moreover, our AH-SADMM, employing a hybrid gradient estimator and acceleration techniques, outperforms other algorithms in terms of descent speed. 

\subsection{Weight Pruning for Neural Network}

We now validate the effectiveness of AH-SADMM by training the LeNet-5 \cite{lecun1998gradient} neural network with weight pruning. Specifically, we address the image classification task by training the following problem:
\begin{equation}\label{pro: LeNet}
\begin{array}{ll}
\underset{W, b}{\operatorname{minimize}} & f(W, b)+\rho\|Z\|_1 \\
\text { subject to } & W=Z
\end{array}
\end{equation}
using MNIST and CIFAR-10 datasets, where $f(W, b)$ denotes the LeNet-5 model, $\rho$ is the penalty parameter. 
The MNIST and CIFAR-10 datasets are employed with batch sizes of 64 and 256, respectively.

\begin{figure}[h] 
\vskip -0.1in
\centering
\subfigure{\includegraphics[width=0.240\textwidth]{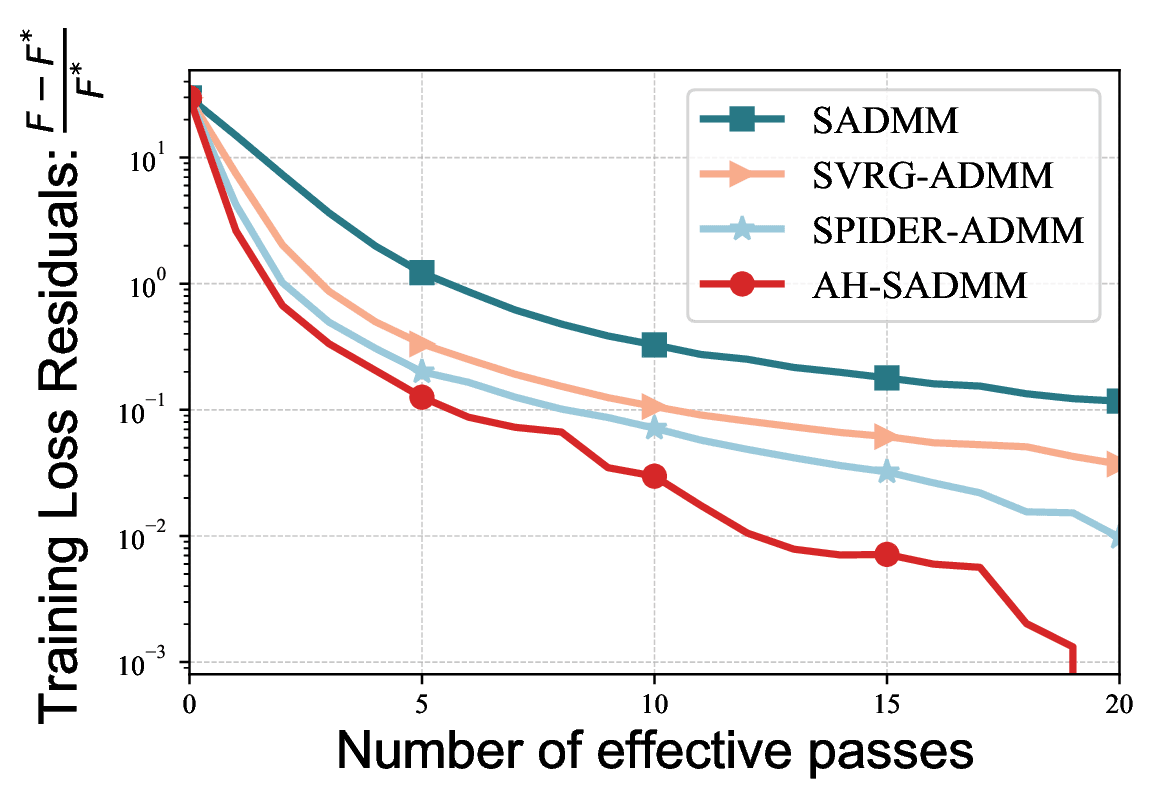}}
\subfigure{\includegraphics[width=0.234\textwidth]{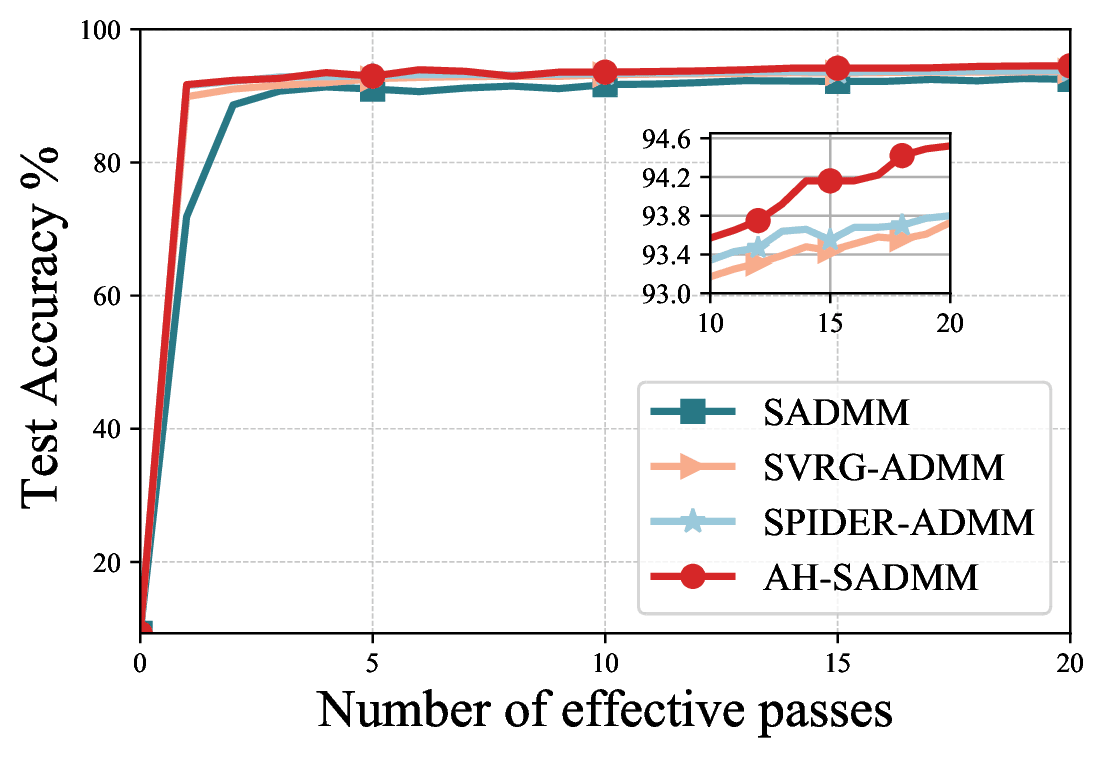}}
\subfigure{\includegraphics[width=0.240\textwidth]{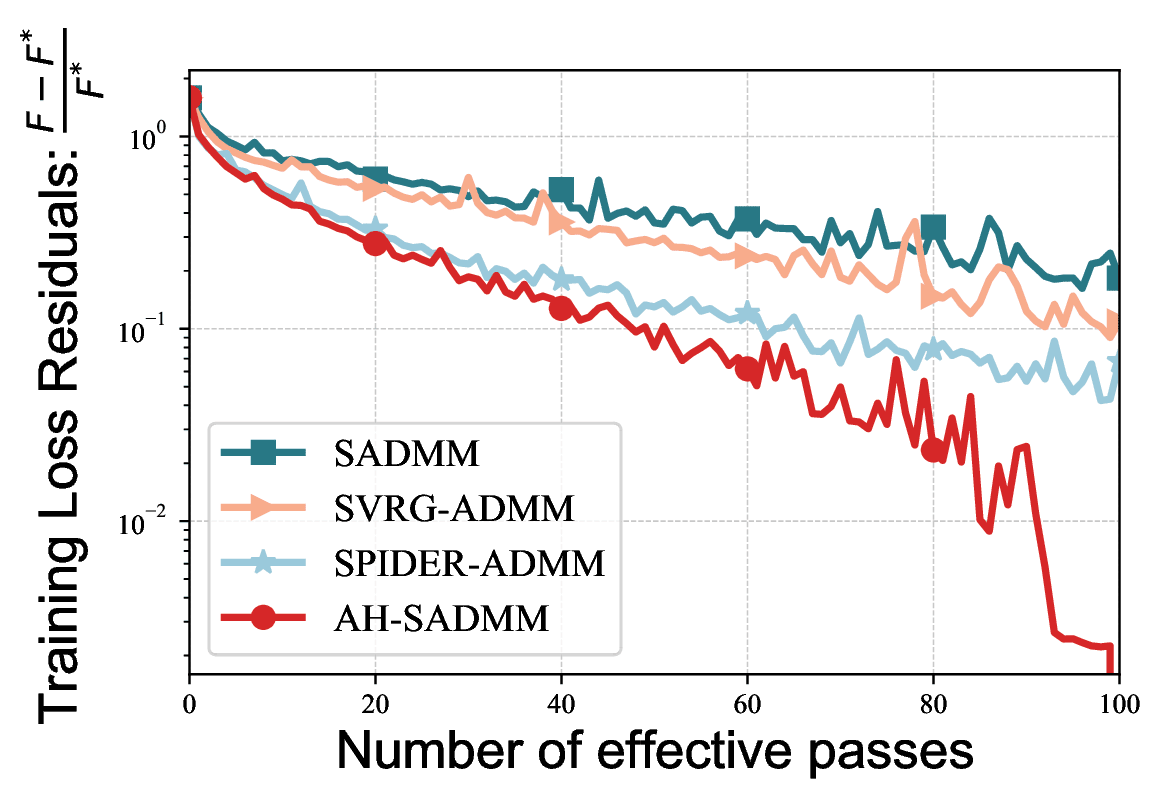}}
\subfigure{\includegraphics[width=0.234\textwidth]{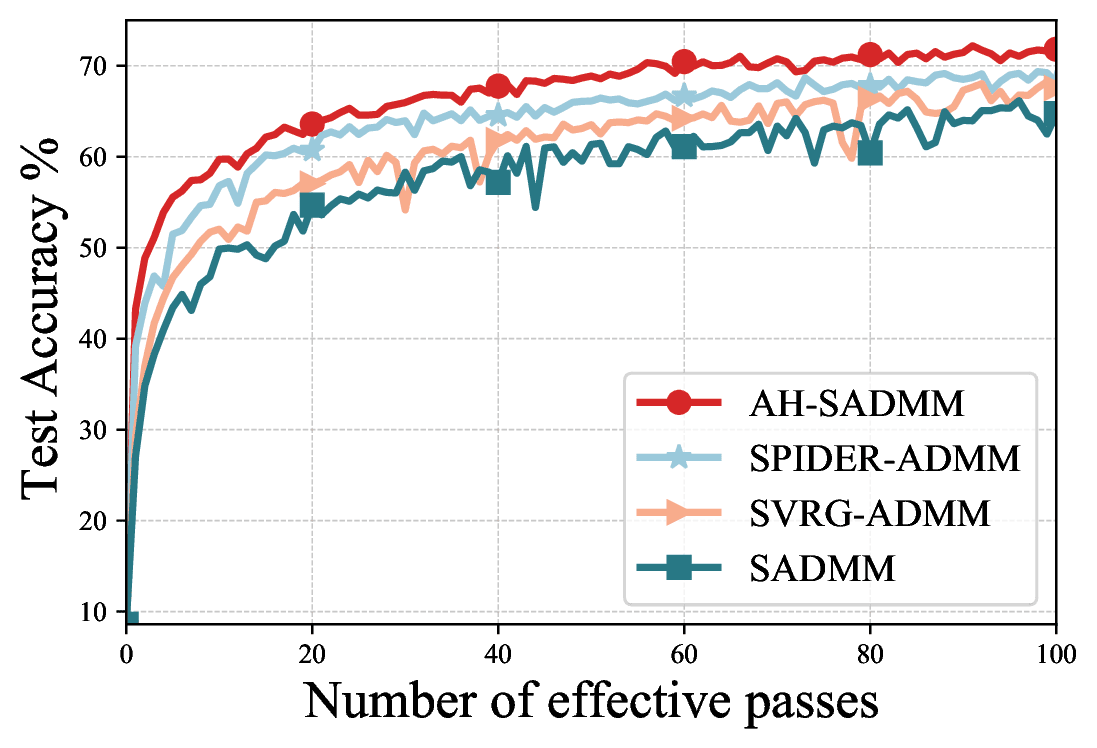}}
\caption{Comparison of algorithms on training LeNet-5 on MNIST (upper) and CIFAR-10 (below).}
\label{fig:LeNet loss and acc}
\vskip -0.1in
\end{figure}

We present and visualize the train loss and test accuracy against effective epoch numbers in \cref{fig:LeNet loss and acc}.
From \cref{fig:LeNet loss and acc}, we observe that (i) AH-SADMM exhibits faster convergence with better generalization than the other algorithms on the MNIST dataset, respectively. 
(ii) On the CIFAR-10 dataset, Our algorithm still outperforms others relatively.


\section{Conclusion}
This paper proposed an inexact stochastic ADMM algorithm with a unified framework for solving nonconvex nonsmooth problems. We have established sublinear convergence rates and linear convergence under local error bound conditions. Additionally, a fast AH-SADMM has been proposed with theoretical analysis. The effectiveness of the algorithm was demonstrated through two nonconvex numerical experiments. Future research will explore the design of stochastic ADMM algorithms for distributed problems.

\vspace{-0.1cm}

\bibliography{refs}

\newpage

\section{Supplementary }
\label{app:remark 1}

\begin{table}[t]
	\caption{Stochastic ADMM methods satisfying the inexact criteria condition (\ref{x update criteria}) for the inexact solution of $\mathbf{x}$ subproblems, under the unified framework.}
	\label{table 1}
	\vskip 0.15in
	\begin{center}
		\begin{small}
			\begin{sc}
				\begin{tabular}{lcccr}
					\toprule
					Methods & Convergence Rate & Inexact Criteria & Variance \\
					\midrule
					SADMM \cite{ouyang2013stochastic}  & $\mathcal{O}(1/T)$ & $\surd$ & Bounded \\
					SVRG-ADMM \cite{huang2016stochastic} & $\mathcal{O}(1/T)$  & $\surd$ & Diminishing \\
					SPIDER-ADMM \cite{huang2019faster} & $\mathcal{O}(1/T)$ & $\surd$  & Diminishing \\
					\bottomrule
				\end{tabular}
			\end{sc}
		\end{small}
	\end{center}
	\vskip -0.1in
\end{table}
\begin{remark}\label{inexact cri}
	SADMM, SVRG-ADMM, and recursive SPIDER-ADMM satisfy the inexact criteria in (\ref{x update criteria}) as listed in \cref{table 1}.
	Next, we will provide a detailed explanation of the reasons behind the fulfillment of this criterion.
	\begin{itemize}
		\item For the general SADMM, the update rule is
		\begin{equation}\label{sto ADMM up}
			\mathbf{x}^{k+1} = \arg \min _{\mathbf{x} \in \mathbb{R}^n x} 
			\left\langle\nabla f\left({\mathbf{x}}^k,\xi_{M}\right), \mathbf{x}-{\mathbf{x}}^k\right\rangle +
			\frac{\beta}{2}\|A \mathbf{x}+B \mathbf{y}^{k+1}-\mathbf{b} - \frac{\mathbf{\lambda}^k}{\beta}\|^2\\
			+\frac{\beta}{2}\left\|\mathbf{x}-{\mathbf{x}}^{k}\right\|_{\mathcal{D}_x^k}^2,
		\end{equation}
		and its optimality condition is
		\begin{equation}\label{sto ADMM opt cond}
			0 = \nabla f\left({\mathbf{x}}^k ,\xi_{M} \right) + \beta A^{\top}\left(A \mathbf{x}^{k+1}+B \mathbf{y}^{k+1}-\mathbf{b} - \frac{\mathbf{\lambda}^k}{\beta}
			\right) + \beta \mathcal{D}_x^k \left(\mathbf{x}^{k+1}-{\mathbf{x}}^{k}
			\right).
		\end{equation}
		
		Using (\ref{sto ADMM opt cond}), the L-smoothness of $\nabla f$, and the bounded variance condition $\mathbb{E}\left[\left\|\nabla f\left({\mathbf{x}}^k,\xi_{M}\right) -
		\nabla f\left({\mathbf{x}}^k\right)\right\|^2\right] \leq \frac{\sigma^2}{M}$, we can obtain
		\begin{equation}\label{sto ADMM update grad}
			\mathbb{E}_{k}\left[\left\|\xi_x^{k+1}\right\|^2 \right]
			\leq 3\beta^{2} \mathbb{E}_{k}\left[
			\frac{\sigma^2}{M\beta^2} + \left(\frac{L^2}{\beta^2} + \left\|\ \mathcal{D}_x^k \right\|^2 \right) \left\|{\mathbf{x}}^{k+1}-\mathbf{x}^{k}\right\|^2
			\right], 
		\end{equation}
		and 
		\begin{align} \label{sto ADMM update value}
			&\frac{\beta}{2} \mathbb{E}_{k}\left[ \left\| \mathbf{x}^{k+1}-{\mathbf{x}}^k\right\|_{\mathcal{D}^k_x}^2 \right] 
			+\mathbb{E}_{k}\left[\mathcal{L}_\beta\left(\mathbf{x}^{k+1}
			, \mathbf{y}^{k+1}, \boldsymbol{\lambda}^k\right)\right] \nonumber\\
			\leq
			&\mathcal{L}_\beta\left({\mathbf{x}}^k, \mathbf{y}^{k+1}, \boldsymbol{\lambda}^k\right) +  \frac{\sigma^2}{2M} + \frac{L+1}{2}\mathbb{E}_{k}\left[ \left\| \mathbf{x}^{k+1}-{\mathbf{x}}^k\right\|^2 \right] - \frac{\beta}{2}\mathbb{E}_{k}\left[ \left\| \mathbf{x}^{k+1}-{\mathbf{x}}^k\right\|_{\mathcal{D}^k_x}^2 \right] 
			.
		\end{align}
		We can choose suitable values for $\beta$ and $\mathcal{D}^k_x$  such that $\frac{L+1}{2}\mathbb{E}_{k}\left[ \left\| \mathbf{x}^{k+1}-{\mathbf{x}}^k\right\|^2 \right] - \frac{\beta}{2}\mathbb{E}_{k}\left[ \left\| \mathbf{x}^{k+1}-{\mathbf{x}}^k\right\|_{\mathcal{D}^k_x}^2 \right]  \leq 0$, and $\hat{c}_x>0$, $c_x>0$ ensuring (\ref{x update criteria}) hold.

		\item As for the stochastic ADMM incorporating the VR technique, for instance, SVRG estimator which is
		\begin{align}\label{f-grad-update}
			\hat{\nabla} f(\mathbf{x}^{k}) = \nabla f\left({\mathbf{x}}^k,\xi_{M}\right) - \nabla f\left(\widetilde{\mathbf{x}},\xi_{M}\right)
			+\nabla f(\widetilde{\mathbf{x}}),\nonumber
		\end{align}
		where $\nabla f\left(\cdot,\xi_{M}\right)$ also denotes the stochastic gradient with batch size $M$, and $\widetilde{\mathbf{x}}$ is a given snapshot point. Recalling Lemma 1 in \cite{huang2016stochastic}and the inexact update rule of $\mathbf{x}$ employing the SVRG estimator is
		\begin{equation}\label{SVRG up}
			\mathbf{x}^{k+1} = \arg \min _{\mathbf{x} \in \mathbb{R}^n x} 
			\left\langle\hat{\nabla} f(\mathbf{x}^{k}), \mathbf{x}-{\mathbf{x}}^k\right\rangle +
			\frac{\beta}{2}\|A \mathbf{x}+B \mathbf{y}^{k+1}-\mathbf{b} - \frac{\mathbf{\lambda}^k}{\beta}\|^2\\
			+\frac{\beta}{2}\left\|\mathbf{x}-{\mathbf{x}}^{k}\right\|_{\mathcal{D}_x^k}^2,
		\end{equation} 
		similarly from the equations (34) and (38) in \cite{huang2016stochastic}, we can get 
		\begin{equation}\label{sto SVRG update grad}
			\mathbb{E}_{k}\left[\left\|\xi_x^{k+1}\right\|^2 \right]
			\leq 3\beta^{2} \mathbb{E}_{k}\left( \frac{L^2}{M\beta^2}\left\|{\mathbf{x}}^{k}-\tilde{\mathbf{x} }\right\|^2 + \left(\frac{L^2}{\beta^2} + \left\|\ \mathcal{D}_x^k \right\|^2 \right) \left\|{\mathbf{x}}^{k+1}-\mathbf{x}^{k}\right\|^2
			\right), 
		\end{equation}
		and 
		\begin{align} \label{SVRG ADMM update value}
			&\frac{\beta}{2} \mathbb{E}_{k}\left[ \left\| \mathbf{x}^{k+1}-{\mathbf{x}}^k\right\|_{\mathcal{D}^k_x}^2 \right] 
			+\mathbb{E}_{k}\left[\mathcal{L}_\beta\left(\mathbf{x}^{k+1}
			, \mathbf{y}^{k+1}, \boldsymbol{\lambda}^k\right)\right] \nonumber\\
			\leq
			&\mathcal{L}_\beta\left({\mathbf{x}}^k, \mathbf{y}^{k+1}, \boldsymbol{\lambda}^k\right) +   \frac{L}{2}\mathbb{E}_{k}\left[ \left\| \mathbf{x}^{k+1}-{\mathbf{x}}^k\right\|^2 \right] - \frac{\beta}{2}\mathbb{E}_{k}\left[ \left\| \mathbf{x}^{k+1}-{\mathbf{x}}^k\right\|_{\mathcal{D}^k_x}^2 \right].
		\end{align}
		
		Recalling Lemma 1 and Theorem 8 in \cite{huang2016stochastic}, it indicates that $\frac{L^2}{M}\mathbb{E}\left\|{\mathbf{x}}^{k}-\tilde{\mathbf{x} }\right\|^2$ serves as the upper bound for the variance. Also, $\sum_{k=1}^{\infty}\left(\mathbb{E}\left\|\mathbf{x}^k-\mathbf{x}^{k-1}\right\|^2+\mathbb{E}\left\|\mathbf{x}^{k-1}-\tilde{\mathbf{x}}\right\|^2\right)<+\infty$, implying that $\frac{L^2}{M}\mathbb{E}\left\|{\mathbf{x}}^{k}-\tilde{\mathbf{x} }\right\|^2$
		diminishes with increasing iterations $k$. Thus
		$\frac{L^2}{M}\mathbb{E}\left\|{\mathbf{x}}^{k}-\tilde{\mathbf{x} }\right\|^2$ is smaller than
		$\mathcal{O}(\frac{\sigma^2}{M})$.
		By selecting suitable values for  $\beta$ and $\mathcal{D}^k_x$ that satisfy $\frac{L}{2}\mathbb{E}_{k}\left[ \left\| \mathbf{x}^{k+1}-{\mathbf{x}}^k\right\|^2 \right] - \frac{\beta}{2}\mathbb{E}_{k}\left[ \left\| \mathbf{x}^{k+1}-{\mathbf{x}}^k\right\|_{\mathcal{D}^k_x}^2 \right]  \leq 0$, along with $\hat{c}_x>0$, $c_x>0$, we can guarantee the satisfaction of  (\ref{x update criteria}) hold.
		
		\item Regarding the recursive method SPIDER-ADMM, the gradient estimator is denoted as
		$$v^k=\nabla f\left({\mathbf{x}}^k ,\xi_{M} \right)-\nabla f\left({\mathbf{x}}^{k-1} ,\xi_{M} \right)+v^{k-1}.$$
		Specially, when the number of iterations $k$ satisfying $\text{mod}(k,q) = 0$ ($q $ is a given integer), $v^k$ adopts the following update
		$$v^k=\nabla f\left({\mathbf{x}}^k \right).$$
		
		The inexact update rule for $\mathbf{x}$, utilizing $v^k$, is defined as
		\begin{equation}\label{SPIDER up}
			\mathbf{x}^{k+1} = \arg \min _{\mathbf{x} \in \mathbb{R}^n x} 
			\left\langle v^k, \mathbf{x}-{\mathbf{x}}^k\right\rangle +
			\frac{\beta}{2}\|A \mathbf{x}+B \mathbf{y}^{k+1}-\mathbf{b} - \frac{\mathbf{\lambda}^k}{\beta}\|^2\\
			+\frac{\beta}{2}\left\|\mathbf{x}-{\mathbf{x}}^{k}\right\|_{\mathcal{D}_x^k}^2,
		\end{equation} 
		
		By examining the aforementioned update rule alongside equations (14) and (29) in \cite{huang2019faster}, we deduce
		\begin{equation}\label{SPIDER update grad}
			\mathbb{E}_{k}\left[\left\|\xi_x^{k+1}\right\|^2 \right]
			\leq 3\beta^{2} \mathbb{E}_{k}\left( \frac{L^2}{\beta^2}\frac{\sum_{i=\left(n_k-1\right) q}^{k-1} \mathbb{E}_k\left\|\mathbf{x}^{i+1}-\mathbf{x}^i\right\|^2}{M}
			+ \left(\frac{L^2}{\beta^2} + \left\|\ \mathcal{D}_x^k \right\|^2 \right) \left\|{\mathbf{x}}^{k+1}-\mathbf{x}^{k}\right\|^2
			\right), 
		\end{equation} 
		and
		\begin{align} \label{SVRG ADMM update value}
			&\frac{\beta}{2} \mathbb{E}_{k}\left[ \left\| \mathbf{x}^{k+1}-{\mathbf{x}}^k\right\|_{\mathcal{D}^k_x}^2 \right] 
			+\mathbb{E}_{k}\left[\mathcal{L}_\beta\left(\mathbf{x}^{k+1}
			, \mathbf{y}^{k+1}, \boldsymbol{\lambda}^k\right)\right] \nonumber\\
			\leq
			&\mathcal{L}_\beta\left({\mathbf{x}}^k, \mathbf{y}^{k+1}, \boldsymbol{\lambda}^k\right) +   L\mathbb{E}_{k}\left[ \left\| \mathbf{x}^{k+1}-{\mathbf{x}}^k\right\|^2 \right] - \frac{\beta}{2}\mathbb{E}_{k}\left[ \left\| \mathbf{x}^{k+1}-{\mathbf{x}}^k\right\|_{\mathcal{D}^k_x}^2 \right]+\frac{L}{2M}\sum_{i=\left(n_k-1\right) q}^{k-1} \mathbb{E}_k\left\|\mathbf{x}^{i+1}-\mathbf{x}^i\right\|^2.
		\end{align}
		
		The result $\frac{L^2}{M}\sum_{i=\left(n_k-1\right) q}^{k-1} \mathbb{E}_k\left\|\mathbf{x}^{i+1}-\mathbf{x}^i\right\|^2$ serving as the upper bound of variance is derived from  Lemma 5 in \cite{huang2019faster}.
		Furthermore, Theorem 5 in \cite{huang2019faster} implies that $\sum_{k=0}^{K-1}\sum_{i=\left(n_k-1\right) q}^{k-1} \mathbb{E}\left\|\mathbf{x}^{i+1}-\mathbf{x}^i\right\|^2<+\infty$.  These results suggest that $\frac{L^2}{M}\sum_{i=\left(n_k-1\right) q}^{k-1} \mathbb{E}\left\|\mathbf{x}^{i+1}-\mathbf{x}^i\right\|^2$ gradually decreases towards 0, and the variance gradually decreases with increasing iterations $k$, being smaller than $\mathcal{O}(\frac{\sigma^2}{M})$. Therefore, the proper selection of $\beta$ and $\mathcal{D}^k_x$ to satisfy $L\mathbb{E}_{k}\left[ \left\| \mathbf{x}^{k+1}-{\mathbf{x}}^k\right\|^2 \right] - \frac{\beta}{2}\mathbb{E}_{k}\left[ \left\| \mathbf{x}^{k+1}-{\mathbf{x}}^k\right\|_{\mathcal{D}^k_x}^2 \right]  \leq 0$, along with $\hat{c}_x>0$, $c_x>0$, ensures the fulfillment of condition (\ref{x update criteria}).

	\end{itemize}

	The variance of SADMM is bounded, while, owing to the utilization of VR techniques, the variances of the latter two gradually diminish.
	This diminishing variance aligns with our inexact update criterion, expressed as 
	$\frac{\sigma^2}{M} \rightarrow 0$.
	Therefore, we illustrate that the proposed UI-SADMM is a general framework encompassing many classical stochastic ADMM algorithms.
	
\end{remark}

\subsection{Convergence Analysis of UI-SADMM}\label{app:dual lemma}
Before introducing the convergence properties of the proposed algorithm , we initially present a crucial lemma concerning the characteristics of the dual variables that is indispensable for the subsequent analysis.

\begin{lemma}\label{ATlambda}
	Suppose that Assumption \ref{assum 1} (a) and (b) hold. Let  $\left\{\mathbf{w}^k:= (\mathbf{x}^k,\mathbf{y}^k,\mathbf{\lambda}^k)\right\}$ be the iterates satisfying the condition (\ref{x update criteria}), then it holds that
	
	\begin{align}
		\mathbb{E}\left[ \left\|A^{\top} \mathbf{d}_\lambda^k\right\|^2 \right]
		&\leq \psi_2(s) \mathbb{E}\left(\left\|A^{\top} \mathbf{d}_\lambda^{k-1}\right\|^2-\left\|A^{\top} \mathbf{d}_\lambda^k\right\|^2\right) + \psi_1(s)\left(2 L_{f}^{2}+4c_x^2\beta^2
		\right)\mathbb{E}\left[ \left\|\mathbf{d}_x^{k} \right\|^2 \right]\nonumber\\
		&\quad +4\psi_1(s)c_x^2\beta^2\mathbb{E}\left[ \left\|\ \mathbf{d}_x^{k-1}  \right\|^2 \right]+ 	4\psi_1(s)c_x^2\beta^2\mathbb{E}\left[ \left\|\ \mathbf{d}_y^{k}  \right\|^2 \right]
		+4\psi_1(s)c_x^2\beta^2\mathbb{E}\left[ \left\|\ \mathbf{d}_y^{k-1}  \right\|^2 \right]
		\nonumber\\
		&\quad + 
		8\psi_1(s)\frac{c_x^2 \beta^{2}\sigma^2}{M},
		\label{ATlambda upper bound}
	\end{align}
	where \begin{equation}
		\psi_1(s)=\max \left\{1, \frac{s^2}{(2-s)^2}\right\} \quad \text { and } \quad \psi_2(s)=\max \left\{\frac{1-s}{s}, \frac{s-1}{2-s}\right\} \text {. }\label{def of psi}
	\end{equation}
\end{lemma} 

\begin{proof}
	Applying a similar proof strategy as presented in \cite{BZZ2023}, we can obtain the following results.
	
	From the definition of $\xi_x^{k+1}=\nabla_x \mathcal{L}_\beta\left({\mathbf{x}}^{k+1}, \mathbf{y}^{k+1}, \boldsymbol{\lambda}^k\right)$, we have
	$$
	\xi_x^{k+1}=\nabla f\left( \mathbf{x}^{k+1}\right)+A^{\top}\left[-\boldsymbol{\lambda}^k+\beta {\mathbf{r}}^{k+1}\right],
	$$
	where ${\mathbf{r}}^{k+1}=A {\mathbf{x}}^{k+1}+B \mathbf{y}^{k+1}-\mathbf{b}$. Then, we further have
	$$
	A^{\top} \boldsymbol{\lambda}^k=\nabla f\left({\mathbf{x}}^{k+1}\right)-\xi_x^{k+1}+\beta A^{\top} {\mathbf{r}}^{k+1},
	$$
	by $\boldsymbol{\lambda}^{k+1}=\boldsymbol{\lambda}^k-s \beta {\mathbf{r}}^{k+1}$ that
	\begin{align}
		s A^{\top} \boldsymbol{\lambda}^k=s\left(\nabla f\left({\mathbf{x}}^{k+1}\right)-\xi_x^{k+1}\right)+A^{\top}\left(\boldsymbol{\lambda}^k-\boldsymbol{\lambda}^{k+1}\right).
	\end{align}
	This yields that 
	\begin{align}
		A^{\top} \boldsymbol{\lambda}^{k+1} & = s\left(\nabla f\left({\mathbf{x}}^{k+1}\right)-\xi_x^{k+1}\right) + (1-s)A^{\top} \boldsymbol{\lambda}^{k}\nonumber\\
		A^{\top} \boldsymbol{\lambda}^{k} & = s\left(\nabla f\left({\mathbf{x}}^{k}\right)-\xi_x^{k}\right) + (1-s)A^{\top} \boldsymbol{\lambda}^{k-1}\nonumber\\
		A^{\top} \mathbf{d}_\lambda^{k } & = s\left(\nabla f\left({\mathbf{x}}^{k+1}\right)-
		\nabla f\left({\mathbf{x}}^{k}\right)+\xi_x^{k}-\xi_x^{k+1}\right)
		+ (1-s)A^{\top}\mathbf{d}_\lambda^{k-1}\nonumber\\
		A^{\top} \mathbf{d}_\lambda^{k } & = s\mathbf{\delta}^{k} + (1-s)A^{\top}\mathbf{d}_\lambda^{k-1}, \label{ATlambda 1}  		
	\end{align}
	where $\mathbf{\delta}^{k}=\nabla f\left({\mathbf{x}}^{k+1}\right)-
	\nabla f\left({\mathbf{x}}^{k}\right)+\xi_x^{k}-\xi_x^{k+1}$.
	

	Now, we consider two different scenarios of $s \in(0,1]$ and $s \in(1,2)$.
	\begin{itemize}
		\item \textbf{Case 1:} $s \in(0,1]$. 
		
		Then, combining (\ref{ATlambda 1}) and the convexity of $\|\cdot\|^2$, one has
		$$
		\left\|A^{\top} \mathbf{d}_\lambda^k\right\|^2 \leq s\left\|\boldsymbol{\delta}^k\right\|^2+(1-s)\left\|A^{\top} \mathbf{d}_\lambda^{k-1}\right\|^2.
		$$
		
		Subtracting $(1-s)\left\|A^{\top} \mathbf{d}_\lambda^k\right\|^2$ and dividing both sides by $s$ from the above inequality, we obtain
		\begin{equation}
			\left\|A^{\top} \mathbf{d}_\lambda^k\right\|^2 \leq\left\|\boldsymbol{\delta}^k\right\|^2+\frac{1-s}{s}\left(\left\|A^{\top} \mathbf{d}_\lambda^{k-1}\right\|^2-\left\|A^{\top} \mathbf{d}_\lambda^k\right\|^2\right). \label{equ 4.7}
		\end{equation}

		\item \textbf{Case 2:} $s \in(1,2)$. 
		
		It follows from (4.5) that
		$$
		\left\|A^{\top} \mathbf{d}_\lambda^k\right\|^2=(1-s)^2\left\|A^{\top} \mathbf{d}_\lambda^{k-1}\right\|^2+s^2\left\|\boldsymbol{\delta}^k\right\|^2+2 s(1-s)\left\langle A^{\top} \mathbf{d}_\lambda^{k-1}, \boldsymbol{\delta}^k\right\rangle .
		$$
		
		Combining the above result with Cauchy-Schwartz inequality, for an $\nu>0$ it follows that
		\begin{equation}\label{ATd lam}
			\begin{aligned}
				\left\|A^{\top} \mathbf{d}_\lambda^k\right\|^2 & \leq(1-s)^2\left\|A^{\top} \mathbf{d}_\lambda^{k-1}\right\|^2+s^2\left\|\boldsymbol{\delta}^k\right\|^2+s(s-1)\left(\nu\left\|A^{\top} \mathbf{d}_\lambda^{k-1}\right\|^2+\frac{1}{\nu}\left\|\boldsymbol{\delta}^k\right\|^2\right) \\
				& =\left((1-s)^2+s(s-1) \nu\right)\left\|A^{\top} \mathbf{d}_\lambda^{k-1}\right\|^2+\left(s^2+\frac{s(s-1)}{\nu}\right)\left\|\boldsymbol{\delta}^k\right\|^2.
			\end{aligned}
		\end{equation}
		
		Selecting $\nu=(2-s) / s$ and reusing (\ref{ATd lam}), one has
		$$
		(1-s)^2+s(s-1) \nu=s-1, \quad  \quad s^2+\frac{s(s-1)}{\nu}=\frac{s^2}{2-s},
		$$
		and
		$$
		\left\|A^{\top} \mathbf{d}_\lambda^k\right\|^2 \leq(s-1)\left\|A^{\top} \mathbf{d}_\lambda^{k-1}\right\|^2+\frac{s^2}{2-s}\left\|\boldsymbol{\delta}^k\right\|^2 .
		$$

		Subtracting $(s-1)\left\|A^{\top} \mathbf{d}_\lambda^k\right\|^2$ and dividing both sides by $2-s$ from the above inequality, we obtain
		\begin{equation}
			\left\|A^{\top} \mathbf{d}_\lambda^k\right\|^2 \leq \frac{s^2}{(2-s)^2}\left\|\boldsymbol{\delta}^k\right\|^2+\frac{s-1}{2-s}\left(\left\|A^{\top} \mathbf{d}_\lambda^{k-1}\right\|^2-\left\|A^{\top} \mathbf{d}_\lambda^k\right\|^2\right). \label{equ4.10}
		\end{equation}

	\end{itemize}
	
	Combining \eqref{equ 4.7} and \eqref{equ4.10} and considering the definition of $\psi_1$ and $\psi_2$ in \eqref{def of psi}, 
	we can further derive
	\begin{equation}
		\left\|A^{\top} \mathbf{d}_\lambda^k\right\|^2 \leq \psi_1(s)\left\|\boldsymbol{\delta}^k\right\|^2+\psi_2(s)\left(\left\|A^{\top} \mathbf{d}_\lambda^{k-1}\right\|^2-\left\|A^{\top} \mathbf{d}_\lambda^k\right\|^2\right).\label{ATlambda 2}
	\end{equation}
	
	By  (\ref{x update criteria}) and the property $\mathbb{E}[\mathbb{E}[\cdot \mid \mathcal{F}_{k}]]=\mathbb{E}[\cdot]$, for the term 
	$\mathbb{E}\left[ \left\|\boldsymbol{\delta}^k\right\|^2 \right]$ we can drive
	\begin{align}
		\mathbb{E}\left[ \left\|\boldsymbol{\delta}^k\right\|^2 \right] &= \mathbb{E}\left[ \left\|\ \nabla f\left({\mathbf{x}}^{k+1}\right)-
		\nabla f\left({\mathbf{x}}^{k}\right)+\xi_x^{k}-\xi_x^{k+1} \right\|^2 \right]	\nonumber\\
		& \leq \mathbb{E}\left[ \left\|\ \nabla f\left({\mathbf{x}}^{k+1}\right)-
		\nabla f\left({\mathbf{x}}^{k}\right) \right\| +
		\left\|\ \xi_x^{k} \right\| + \left\|\ \xi_x^{k+1} \right\|  \right]^2 \nonumber\\
		& \mathop{\leq}^{(i)} 2 L_{f}^{2}  \mathbb{E}\left[ \left\|\mathbf{d}_x^{k} \right\|^2 \right]  + 4\mathbb{E}\left[ \left\|\ \xi_x^{k} \right\|^2 \right]+ 4\mathbb{E}\left[ \left\|\ \xi_x^{k+1} \right\|^2 \right]\nonumber\\
		& \mathop{\leq}^{(ii)} 2 L_{f}^{2}  \mathbb{E}\left[ \left\|\mathbf{d}_x^{k} \right\|^2 \right]  +  4(c_x \beta)^{2} \left(\frac{2\sigma^2}{M} +
		\mathbb{E}\left[ \left\| \mathbf{d}_x^{k} \right\|^2 \right] +
		\mathbb{E}\left[ \left\| \mathbf{d}_y^{k} \right\|^2 \right] +
		\mathbb{E}\left[ \left\| \mathbf{d}_x^{k-1} \right\|^2 \right] +
		\mathbb{E}\left[ \left\| \mathbf{d}_y^{k-1} \right\|^2 \right] 
		\right)\nonumber\\
		& \leq
		(2 L_{f}^{2} + 4c_x^2 \beta^{2})\mathbb{E}\left[ \left\| \mathbf{d}_x^{k} \right\|^2 \right] + 4c_x^2 \beta^{2}\mathbb{E}\left[ \left\| \mathbf{d}_x^{k-1} \right\|^2 \right]+
		4c_x^2 \beta^{2}\mathbb{E}\left[ \left\| \mathbf{d}_y^{k} \right\|^2 \right]+
		4c_x^2 \beta^{2}\mathbb{E}\left[ \left\| \mathbf{d}_y^{k-1} \right\|^2 \right]\nonumber\\
		&\quad+ \frac{8c_x^2 \beta^{2}\sigma^2}{M}
		\label{exp of delta k},
	\end{align}
	where (i) is due to $
	\left\|\ \nabla f\left(  \mathbf{x}\right) - \nabla f\left( \mathbf{y}\right) \right\| \leq L_f \left\|\ \mathbf{x} - \mathbf{y} \right\|
	$ and Cauchy-Schwartz inequality, and (ii) is from the optimal condition (\ref{x update criteria}).

	Combining (\ref{ATlambda 2}) and (\ref{exp of delta k}), we have completed the proof 
	\begin{align}
		\mathbb{E}\left[ \left\|A^{\top} \mathbf{d}_\lambda^k\right\|^2 \right]
		&\leq \psi_2(s) \mathbb{E}\left(\left\|A^{\top} \mathbf{d}_\lambda^{k-1}\right\|^2-\left\|A^{\top} \mathbf{d}_\lambda^k\right\|^2\right) + \psi_1(s)\left(2 L_{f}^{2}+4c_x^2\beta^2
		\right)\mathbb{E}\left[ \left\|\mathbf{d}_x^{k} \right\|^2 \right]\nonumber\\
		&\quad +4\psi_1(s)c_x^2\beta^2\mathbb{E}\left[ \left\|\ \mathbf{d}_x^{k-1}  \right\|^2 \right]+ 	4\psi_1(s)c_x^2\beta^2\mathbb{E}\left[ \left\|\ \mathbf{d}_y^{k}  \right\|^2 \right]
		+4\psi_1(s)c_x^2\beta^2\mathbb{E}\left[ \left\|\ \mathbf{d}_y^{k-1}  \right\|^2 \right]
		\nonumber\\
		&\quad + 
		8\psi_1(s)\frac{c_x^2 \beta^{2}\sigma^2}{M}.
		\nonumber 
	\end{align}
\end{proof}

\begin{remark}
	For the nonconvex and nonsmooth problems, this paper extends the range of the dual variable's stepsize to $(0, 2)$, in contrast to $(0, \frac{1+\sqrt{5}}{2})$ as in \cite{yang2017alternating}. 
	This extension is primarily attributed to a novel technique applied in the inequality of  \eqref{ATd lam}, specifically employing the Cauchy-Schwartz inequality to introduce the free parameter $\nu$. We detail the distinctions in the proofs concerning the selection of dual stepsize between our work and \cite{yang2017alternating} as follows.
	\begin{itemize}
		\item For $s>1$ in \cite{yang2017alternating}, the convexity inequality was employed for  
		$$\begin{aligned} \frac{1}{s}\left(\Lambda^{k+1}-\Lambda^k\right) & =\frac{1}{s} s \mathcal{A}^* \mathcal{A}\left(Z^k-Z^{k+1}\right)+\left(1-\frac{1}{s}\right)\left(\Lambda^{k-1}-\Lambda^k\right)
		\end{aligned}$$
		to derive 
		$$
		\left\|\Lambda^{k+1}-\Lambda^k\right\|_F^2 \leq s^3 \lambda_{\max }^2\left\|Z^{k+1}-Z^k\right\|_F^2+\left(s^2-s\right)\left\|\Lambda^k-\Lambda^{k-1}\right\|_F^2.
		$$
		
		Then, subtracting $\left(s^2-s\right)\left\|\Lambda^{k+1}-\Lambda^k\right\|_F^2$ from both sides of the above inequality, one has that
		\begin{equation}
			(1+s-s^2)\left\|\Lambda^{k+1}-\Lambda^k\right\|_F^2 \leq s^3 \lambda_{\max }^2\left\|Z^{k+1}-Z^k\right\|_F^2+\left(s^2-s\right)(\left\|\Lambda^k-\Lambda^{k-1}\right\|_F^2 - \left\|\Lambda^{k+1}-\Lambda^k\right\|_F^2).\nonumber
		\end{equation}
		
		It can be observed that the coefficient $1+s-s^2$ on the left side must be greater than 0 for convergence analysis. This leads to the classical result: $s \in (0, \frac{1+\sqrt{5}}{2})$.
		
		\item For $s>1$, unlike \cite{yang2017alternating}, our proof utilizes the Cauchy-Schwartz inequality, with a crucial technique to introduce a new free variable $\nu$. Selecting a specific $\nu$ results in the derivation of \eqref{equ4.10}. Consequently, the dual stepsize is further relaxed, falling within the range of $(0 ,2)$.

	\end{itemize}

\end{remark}

\subsection{Proof of Theorem \ref{x y w grad 0}}\label{app:thm1}
Applying  (\ref{y update criteria}), (\ref{x update criteria}), and the update rule of $\boldsymbol{\lambda}$: $\boldsymbol{\lambda}^{k+1}=\boldsymbol{\lambda}^k-s \beta\left(A \mathbf{x}^{k+1}+B \mathbf{y}^{k+1}-\mathbf{b}\right)$, we have
\begin{equation}
	\begin{cases}
		\mathcal{L}_\beta\left(\mathbf{x}^k, \mathbf{y}^{k+1}, \boldsymbol{\lambda}^k\right)- \mathcal{L}_\beta\left(\mathbf{x}^k, \mathbf{y}^{k}, \boldsymbol{\lambda}^k\right) &\leq -\frac{\beta}{2}\left\|\mathbf{y}^{k+1}-{\mathbf{y}}^{k}\right\|_{\mathcal{D}^{k}_y}^2,
		\\ 
		\mathbb{E}_{k}\mathcal{L}_\beta\left(\mathbf{x}^{k+1}, \mathbf{y}^{k+1}, \boldsymbol{\lambda}^k\right)- \mathcal{L}_\beta\left(\mathbf{x}^{k}, \mathbf{y}^{k+1}, \boldsymbol{\lambda}^k\right)&\leq \frac{ (\hat{c}_x \beta)^2}{2} \frac{\sigma^2}{M} - \frac{\beta}{2}\mathbb{E}_{k}\left\|\mathbf{x}^{k+1}-{\mathbf{x}}^{k}\right\|_{\mathcal{D}^{k}_x}^2,
		\\
		\mathcal{L}_\beta\left(\mathbf{x}^{k+1}, \mathbf{y}^{k+1}, \boldsymbol{\lambda}^{k+1}\right)
		- \mathcal{L}_\beta\left(\mathbf{x}^{k+1}, \mathbf{y}^{k+1}, \boldsymbol{\lambda}^{k}\right)&= \frac{1}{s\beta} \left\|\ \mathbf{d}_\lambda^{k }  \right\|^2,
	\end{cases}\label{des inequs}
\end{equation}
then 
taking expectation on both sides in (\ref{des inequs})  implies that
\begin{equation}\label{descent 1}
	\begin{aligned}
		\mathbb{E}\mathcal{L}_\beta\left(\mathbf{x}^{k+1}, \mathbf{y}^{k+1}, \boldsymbol{\lambda}^{k+1}\right) \leq&
		\mathbb{E}\mathcal{L}_\beta\left(\mathbf{x}^{k}, \mathbf{y}^{k}, \boldsymbol{\lambda}^k\right)
		+\frac{1}{s\beta} \mathbb{E}\left[\left\|\  \mathbf{d}_\lambda^{k }  \right\|^2\right] -\frac{\beta}{2} \mathbb{E}\left[\left\|\mathbf{d}_y^{k}\right\|_{\mathcal{D}^{k}_y}^2\right] - \frac{\beta}{2} \mathbb{E}\left[\left\|
		\mathbf{d}_x^{k}
		\right\|_{\mathcal{D}^{k}_x}^2\right]+\frac{ (\hat{c}_x \beta)^2}{2}\frac{\sigma^2}{M}\\
		\mathop{\leq}^{(i)}&
		\mathbb{E}\mathcal{L}_\beta\left(\mathbf{x}^{k}, \mathbf{y}^{k}, \boldsymbol{\lambda}^k\right)
		+
		\frac{1+\tau}{s\beta\sigma_{A}} \mathbb{E}\left[\left\|\  A\mathbf{d}_\lambda^{k }  \right\|^2\right]
		-
		\frac{\tau}{s\beta} \mathbb{E}\left[\left\|\  \mathbf{d}_\lambda^{k }  \right\|^2\right]
		-\frac{\beta}{2} \mathbb{E}\left[\left\|\mathbf{d}_y^{k}\right\|_{\mathcal{D}^{k}_y}^2\right] \\
		&- \frac{\beta}{2} \mathbb{E}\left[\left\|
		\mathbf{d}_x^{k}
		\right\|_{\mathcal{D}^{k}_x}^2\right]+\frac{ (\hat{c}_x \beta)^2}{2}\frac{\sigma^2}{M}
	\end{aligned}
\end{equation}
where (i) is from $\frac{1}{s\beta} \mathbb{E}\left[\left\|\  \mathbf{d}_\lambda^{k } \right\|^2\right]\leq \frac{1+\tau}{s\beta\sigma_{A}} \mathbb{E}\left[\left\|\  A\mathbf{d}_\lambda^{k } \right\|^2\right]- \frac{\tau}{s\beta} \mathbb{E}\left[\left\|\  \mathbf{d}_\lambda^{k } \right\|^2\right]$, $\tau \in \left(0,1\right)$, and $\sigma_A$ is square root of the smallest positive eigenvalue of $A^{\top} A$ (or the smallest positive eigenvalue of $A A^{\top}$).

Applying Lemma \ref{ATlambda}
to (\ref{descent 1}), we have
\begin{equation}
	\begin{aligned}
		&\mathbb{E}\mathcal{L}_\beta\left(\mathbf{x}^{k+1}, \mathbf{y}^{k+1}, \boldsymbol{\lambda}^{k+1}\right)\\
		\leq &
		\mathbb{E}\mathcal{L}_\beta\left(\mathbf{x}^{k}, \mathbf{y}^{k}, \boldsymbol{\lambda}^k\right)
		-
		\frac{\tau}{s\beta} \mathbb{E}\left[\left\|\  \mathbf{d}_\lambda^{k }  \right\|^2\right]
		-\frac{\beta}{2} \mathbb{E}\left[\left\|\mathbf{d}_y^{k}\right\|_{\mathcal{D}^{k}_y}^2\right] - \frac{\beta}{2} \mathbb{E}\left[\left\|
		\mathbf{d}_x^{k}
		\right\|_{\mathcal{D}^{k}_x}^2\right]+\frac{ (\hat{c}_x \beta)^2}{2}\frac{\sigma^2}{M}\\
		& +
		4\frac{1+\tau}{s\beta\sigma_{A}}\psi_1(s)c_x^2\beta^2\Big(
		\mathbb{E}\left[ \left\|\ \mathbf{d}_x^{k-1}  \right\|^2 \right]
		+
		\mathbb{E}\left[ \left\|\ \mathbf{d}_y^{k}  \right\|^2 \right] +
		\mathbb{E}\left[ \left\|\ \mathbf{d}_y^{k-1}  \right\|^2 \right]
		\Big)\\
		& +  \frac{1+\tau}{s\beta\sigma_{A}}\psi_1(s)\left(2 L_{f}^{2}+4c_x^2\beta^2
		\right)\mathbb{E}\left[ \left\|\mathbf{d}_x^{k} \right\|^2 \right] + \frac{1+\tau}{s\beta\sigma_{A}}\psi_2(s)\mathbb{E}\left(\left\|A^{\top} \mathbf{d}_\lambda^{k-1}\right\|^2-\left\|A^{\top} \mathbf{d}_\lambda^k\right\|^2\right)\\
		&+8\frac{1+\tau}{s\beta\sigma_{A}}\psi_1(s)\frac{c_x^2 \beta^{2}\sigma^2}{M}.
	\end{aligned}
\end{equation}

Now, we define
\begin{equation}
	\begin{cases}
		\mathcal{L}_\beta\left(k\right) := 
		\mathbb{E}\mathcal{L}_\beta\left(\mathbf{x}^{k}, \mathbf{y}^{k}, \boldsymbol{\lambda}^{k}\right),\\ 	
		\hat{B} := \frac{1+\tau}{s\beta\sigma_{A}}\psi_1(s)\left(2 L_{f}^{2}+4c_x^2\beta^2
		\right),
	\end{cases}
\end{equation}
to drive
\begin{equation}\label{lag des 1}
	\begin{aligned}
		&\mathcal{L}_\beta\left(k+1\right)\\
		\leq 
		& \mathcal{L}_\beta\left(k\right) + \hat{A}(\mathbb{E}\left[ \left\|\ \mathbf{d}_x^{k-1}  \right\|^2 \right] - \mathbb{E}\left[ \left\|\ \mathbf{d}_x^{k}  \right\|^2 \right]) + \left(\hat{B}+\hat{A}-\frac{\beta}{2}\sigma_{min}( \mathcal{D}^{k}_x )\right)\mathbb{E}\left[ \left\|\ \mathbf{d}_x^{k}  \right\|^2 \right]\\
		&+ \hat{A}\left(\mathbb{E}\left[ \left\|\ \mathbf{d}_y^{k-1}\right\|^2\right]- \mathbb{E}\left[ \left\|\ \mathbf{d}_y^{k}  \right\|^2\right]  
		\right)+\left(2\hat{A}-\frac{\beta}{2}\sigma_{min}( \mathcal{D}^{k}_y )
		\right)\mathbb{E}\left[ \left\|\ \mathbf{d}_y^{k}\right\|^2\right] -\frac{\tau}{s\beta} \mathbb{E}\left[\left\|\  \mathbf{d}_\lambda^{k }  \right\|^2\right]\\
		&+ \frac{1+\tau}{s\beta\sigma_{A}}\psi_2(s)\mathbb{E}\left(\left\|A^{\top} \mathbf{d}_\lambda^{k-1}\right\|^2-\left\|A^{\top} \mathbf{d}_\lambda^k\right\|^2\right)+\frac{ (\hat{c}_x \beta)^2}{2}\frac{\sigma^2}{M}
		+8\frac{1+\tau}{s\beta\sigma_{A}}\psi_1(s)\frac{c_x^2 \beta^{2}\sigma^2}{M}.
	\end{aligned}
\end{equation}

Recalling the definition of potential function $\mathcal{P}^k$ in (\ref{def of potential}) and substituting it into (\ref{lag des 1}), one has
\begin{equation}\label{poten inequ1}
	\begin{aligned}
		\mathbb{E}\mathcal{P}^{k+1}\leq &\mathbb{E}\mathcal{P}^{k}+ \left(\hat{B}+\hat{A}-\frac{\beta}{2}\sigma_{min}( \mathcal{D}^{k}_x )\right)\mathbb{E}\left[ \left\|\ \mathbf{d}_x^{k}  \right\|^2 \right]
		+ \left(2\hat{A}-\frac{\beta}{2}\sigma_{min}( \mathcal{D}^{k}_y )
		\right)\mathbb{E}\left[ \left\|\ \mathbf{d}_y^{k}\right\|^2\right]\\
		&-\frac{\tau}{s\beta} \mathbb{E}\left[\left\|\  \mathbf{d}_\lambda^{k }  \right\|^2\right]
		+\frac{ (\hat{c}_x \beta)^2}{2}\frac{\sigma^2}{M}
		+8\frac{1+\tau}{s\beta\sigma_{A}}\psi_1(s)\frac{c_x^2 \beta^{2}\sigma^2}{M}.
	\end{aligned}
\end{equation}

Choosing $\mathcal{D}^{k}_x=w_1 I_{n_x}$, $\mathcal{D}^{k}_y=w_2 I_{n_y}$ to satisfy
\begin{equation}\label{params constraint}
	\left\{\begin{array} { l } 
		{ \widehat { B } + \widehat { A } - \frac { \beta } { 2 } \sigma _ { \operatorname { m i n } } ( D _ { x } ^ { k } ) \leq - w } \\
		{ 2 \widehat { A } - \frac { \beta } { 2 } \sigma _ { \operatorname { m i n } } ( D _ { y } ^ { k } ) \leq - w }
	\end{array} \Rightarrow \left\{\begin{array}{l}
		w_1 \geq \frac{\widehat{B}+\widehat{A}+w}{\frac{\beta}{2}} \\
		w_2 \geq \frac{2 \widehat{A}+w}{\frac{\beta}{2}}, 
	\end{array}\right.\right.
\end{equation}
where $w>0$ is a constant.

Denote $\mu = \min\left\{w, \frac{\tau}{s\beta}
\right\}$, we see further that
\begin{equation}\label{x+y+lam}
	\mu \Biggl\{ \mathbb{E}\left[\left\|\mathbf{d}_y^k\right\|^2\right] +\mathbb{E}\left[\left\|\mathbf{d}_x^k\right\|^2\right]+\mathbb{E}\left[\left\|\mathbf{d}_\lambda^k\right\|^2\right]\Biggr\} \leq
	\mathbb{E}\mathcal{P}^{k} - \mathbb{E}\mathcal{P}^{k+1}
	+\left(\frac{ (\hat{c}_x \beta)^2}{2}
	+8\frac{(1+\tau)c_x^2 \beta^{2}}{s\beta\sigma_{A}}\psi_1(s)\right)\frac{\sigma^2}{M},
\end{equation}
and
\begin{align}\label{sum of poten}
	&\mu \sum_{k=0}^{K} \Biggl\{ \mathbb{E}\left[\left\|\mathbf{d}_y^k\right\|^2\right] +\mathbb{E}\left[\left\|\mathbf{d}_x^k\right\|^2\right]+\mathbb{E}\left[\left\|\mathbf{d}_\lambda^k\right\|^2\right]\Biggr\} \nonumber\\
	\leq
	&\mathbb{E}\mathcal{P}\left({\mathbf{x}}^{0}, \mathbf{y}^{0}, \boldsymbol{\lambda}^{0}\right) - \mathbb{E}\mathcal{P}\left({\mathbf{x}}^{k+1}, \mathbf{y}^{k+1}, \boldsymbol{\lambda}^{k+1}\right)+(K+1)\frac{\sigma^2}{M}\mu_{\sigma,M},
\end{align}
where $\mu_{\sigma,M} := 
\frac{ \widehat{c}_x^2 \beta^2}{2 }+\frac{8(1+\tau) \psi_1(s)c_x^2 \beta^2}{s \beta \sigma_A}$. 

By setting $M = \mathcal{O}(K+1)$, we have from (\ref{sum of poten}) that
\begin{align}
	&\mu \sum_{k=0}^{K} \Biggl\{ \mathbb{E}\left[\left\|\mathbf{d}_y^k\right\|^2\right] +\mathbb{E}\left[\left\|\mathbf{d}_x^k\right\|^2\right]+\mathbb{E}\left[\left\|\mathbf{d}_\lambda^k\right\|^2\right]\Biggr\} \nonumber\\
	\leq
	&\mathbb{E}\mathcal{P}\left({\mathbf{x}}^{0}, \mathbf{y}^{0}, \boldsymbol{\lambda}^{0}\right) - \mathbb{E}\mathcal{P}\left({\mathbf{x}}^{k+1}, \mathbf{y}^{k+1}, \boldsymbol{\lambda}^{k+1}\right)+
	\mathcal{O}(\sigma^2\mu_{\sigma,M})\nonumber\\
	\leq&\mathbb{E}\mathcal{P}\left({\mathbf{x}}^{0}, \mathbf{y}^{0}, \boldsymbol{\lambda}^{0}\right) - \bar{\mathcal{P}}+
	\mathcal{O}(\sigma^2\mu_{\sigma,M}),
\end{align}	

\begin{equation}\label{sum is finite}
	\sum_{k=0}^{K} \Biggl\{ \mathbb{E}\left[\left\|\mathbf{d}_y^k\right\|^2\right] +\mathbb{E}\left[\left\|\mathbf{d}_x^k\right\|^2\right]+\mathbb{E}\left[\left\|\mathbf{d}_\lambda^k\right\|^2\right]\Biggr\} < +\infty,
\end{equation}

\begin{align}
	&\min_{k\in\left\{0,...,K\right\}} \Biggl\{ \mathbb{E}\left[\left\|\mathbf{d}_y^k\right\|^2\right] +\mathbb{E}\left[\left\|\mathbf{d}_x^k\right\|^2\right]+\mathbb{E}\left[\left\|\mathbf{d}_\lambda^k\right\|^2\right]\Biggr\} \nonumber\\
	\leq &\frac{\mathbb{E}\mathcal{P}\left({\mathbf{x}}^{0}, \mathbf{y}^{0}, \boldsymbol{\lambda}^{0}\right) - \mathbb{E}\mathcal{P}\left({\mathbf{x}}^{k+1}, \mathbf{y}^{k+1}, \boldsymbol{\lambda}^{k+1}\right)}{\mu(K+1)} + 
	\frac{\mu_{\sigma,M}}{\mu}
	\mathcal{O}(\frac{\sigma^2}{M})
	\nonumber\\
	&\min_{k\in\left\{0,...,K\right\}} \Biggl\{ \mathbb{E}\left[\left\|\mathbf{d}_y^k\right\|^2\right] +\mathbb{E}\left[\left\|\mathbf{d}_x^k\right\|^2\right]+\mathbb{E}\left[\left\|\mathbf{d}_\lambda^k\right\|^2\right]\Biggr\} = \mathcal{O}(\frac{1}{K}),
\end{align}
and
\begin{align}\label{x,y 0}
	\lim _{k \rightarrow \infty}\mathbb{E}\left\|{\mathbf{d}}_x^k\right\|=0, \quad \lim _{k \rightarrow \infty}\mathbb{E}\left\|\mathbf{d}_y^k\right\|=0 \quad \text{and} \quad \lim _{k \rightarrow \infty}\mathbb{E}\left\|\mathbf{d}_\lambda^k\right\|=0,
\end{align}	
where $\bar{\mathcal{P}}$ is the lower bound of $\mathcal{P}\left({\mathbf{x}}^{k}, \mathbf{y}^{k}, \boldsymbol{\lambda}^{k}\right)$. In addition, by the update rule of $ \boldsymbol{\lambda}^k$ which is $\boldsymbol{r}^{k+1}=-\frac{\mathbf{d}_\lambda^k}{s\beta}$, we have 
\begin{align}
	\lim _{k \rightarrow \infty}\mathbb{E}\left\|\mathbf{r}^k\right\|=0.\label{r limit}
\end{align}	

Now, with the denotation of the iterates generated by Alg.\ref{alg 1} as $\left\{\mathbf{w}^k:= (\mathbf{x}^k,\mathbf{y}^k,\mathbf{\lambda}^k)\right\}$ and some direct calculations, we obtain		
\begin{align}
	\label{partial grad of Lag}
	\partial_x \mathcal{L}_\beta\left(\mathbf{w}^k\right) & =\nabla f\left(\mathbf{x}^k\right)-A^{\top} \lambda^k+\beta A^{\top} \mathbf{r}^k\nonumber \\
	& =\nabla_x \mathcal{L}_\beta\left({\mathbf{x}}^{k}, \mathbf{y}^k, \lambda^{k-1}\right)-A^{\top} \mathbf{d}_\lambda^{k-1} = \xi_x^{k} -A^{\top} \mathbf{d}_\lambda^{k-1}
	, \nonumber\\
	\partial_y \mathcal{L}_\beta\left(\mathbf{w}^k\right) & =\partial_y g\left(\mathbf{y}^k\right)-B^{\top} \lambda^k+\beta B^{\top} \mathbf{r}^k \nonumber\\
	& = \partial_y \mathcal{L}_\beta\left(\mathbf{x}^{k-1}, \mathbf{y}^k, \lambda^{k-1}\right)-B^{\top}\left(\mathbf{d}_\lambda^{k-1}-\beta A \mathbf{d}_x^{k-1}\right), \quad \text { and } \nonumber\\
	\partial_\lambda \mathcal{L}_\beta\left(\mathbf{w}^k\right)&
	=\partial_\lambda \mathcal{L}\left(\mathbf{w}^k\right)=-\mathbf{r}^k.
\end{align}

Using the optimal conditions (\ref{y update criteria}) and (\ref{x update criteria}),  (\ref{x,y 0}), (\ref{r limit}), $M = \mathcal{O}(K+1)$, sufficiently large $K$, we conclude that
\begin{equation} \label{grad conv 1}
	\lim _{k \rightarrow \infty} \mathbb{E}\left[
	\operatorname{dist}\left(\mathbf{0}, \partial \mathcal{L}_\beta\left(\mathbf{w}^k\right)\right)\right]=\mathbb{E}\left(\begin{aligned}
		&\partial_x \mathcal{L}_\beta\left(\mathbf{w}^k\right)\\
		&\partial_y \mathcal{L}_\beta\left(\mathbf{w}^k\right)\\
		&\partial_\lambda \mathcal{L}_\beta\left(\mathbf{w}^k\right)
	\end{aligned}
	\right)=0.
\end{equation}	

\begin{lemma}\label{lemma of cluster}                       	
	Let $S\left(\mathbf{w}_0\right)$ denote the set of limit points of iterates $\left\{\mathbf{w}^k:= (\mathbf{x}^k,\mathbf{y}^k,\mathbf{\lambda}^k)\right\}$ and  Assumptions \ref{assum 1} and conditions in Theorem \ref{x y w grad 0}  hold. Then there exists a $\mathcal{F}^*$ such that
	\begin{itemize}
		\item [(i)]
		When $M$ is sufficiently large or $\sigma^2 \rightarrow 0$ (VR-gradient estimator), the sequence $\left\{ \mathbb{E}\left[\mathcal{P}^{k}\right] \right\}$ is nonincreasing, and 
		\begin{align}\label{obj conv1}
			\lim _{k \rightarrow \infty}\mathbb{E}\left[
			\mathcal{P}\left({\mathbf{x}}^{k}, \mathbf{y}^{k}, \boldsymbol{\lambda}^{k}\right)\right]
			= \lim _{k \rightarrow \infty} \mathbb{E}\left[\mathcal{L}_\beta\left(\mathbf{x}^k,{\mathbf{y}}^{k}, \boldsymbol{\lambda}^k\right)\right]
			=\mathcal{F}^{*}.
		\end{align}
		
		\item[(ii)] $S\left(\mathbf{w}_0\right) \subset \operatorname{crit} \mathcal{L}_\beta$ a.s.
		
	\end{itemize}
\end{lemma}

\begin{proof}
	Combining sufficiently large $M$ or $\sigma^2 \rightarrow 0$ with (\ref{poten inequ1}), (\ref{params constraint}), and $\mu = \min\left\{w, \frac{\tau}{s\beta}
	\right\}$, it yields that
	\begin{equation}
		\begin{aligned}
			\mathbb{E}\mathcal{P}^{k+1}\leq &\mathbb{E}\mathcal{P}^{k}
			-\mu \mathbb{E}\left[ \left\|\ \mathbf{d}_x^{k}  \right\|^2 \right]
			-\mu \mathbb{E}\left[ \left\|\ \mathbf{d}_y^{k}\right\|^2\right]-\mu \mathbb{E}\left[\left\|\  \mathbf{d}_\lambda^{k }  \right\|^2\right].
		\end{aligned}
	\end{equation}
	So, we have the sequence $\mathbb{E}\mathcal{P}^{k}$ is monotonically nonincreasing, which together
	with $\left\{\mathcal{P}^{k} \right\}$ being bounded from below gives $\lim _{k \rightarrow \infty}\mathbb{E}\left[
	\mathcal{P}\left({\mathbf{x}}^{k}, \mathbf{y}^{k}, \boldsymbol{\lambda}^{k}\right)\right]=\mathcal{F}^{*}$ for some $\mathcal{F}^{*}.$ 
	
	Then, it follows from the definition of $\mathcal{P}^{k}$ in (\ref{def of potential}), (\ref{x,y 0}) and (\ref{r limit}) that (\ref{obj conv1}) holds. So, item (i) is derived.

	\quad \quad For item (ii), it's sufficient to prove $\bar{\mathbf{w}} \in \operatorname{crit} \mathcal{L}_\beta$ for any $\bar{\mathbf{w}} \in S\left(\mathbf{w}_0\right)$. From $\bar{\mathbf{w}} \in S\left(\mathbf{w}_0\right)$, we have $\mathbf{w}^{k_q} \rightarrow \bar{\mathbf{w}}$, $d^{k_q} \in \partial \mathcal{L}_\beta\left(\mathbf{w}^{k_q} \right)$ and $d^{k_q} \rightarrow 0 \quad \text{a.s.}$ by Lemma \ref{x y w grad 0}. Noting that the {{outer semicontinuity of the Clarke subgradient}} $\partial \mathcal{L}_\beta\left(\mathbf{w}^{k_q} \right)$, it implies that $0 \in \partial \mathcal{L}_\beta\left(\bar{\mathbf{w}} \right)$.
	Therefore, we prove that $\bar{\mathbf{w}} \in \operatorname{crit} \mathcal{L}_\beta$ for any $\bar{\mathbf{w}} \in S\left(\mathbf{w}_0\right)$ which is equivalent to $S\left(\mathbf{w}_0\right) \subset \operatorname{crit} \mathcal{L}_\beta$ a.s.
\end{proof}

\subsection{Remark of Error Bound Conditions}\label{app:remark of error bound}
\begin{remark}
	The Error Bound assumption has been employed in the convergence analysis of various algorithms, including block coordinate gradient descent and ADMM methods, as referenced in works like  \cite{zhou2017unified,wen2017linear}. This assumption consists of three parts: the first part is a condition on the local error bound, the second assumption indicates that when restricted to the set $\Omega^*$, the isocost surfaces of the function $F$ can be suitably separated. Many functions can satisfy assumptions (a) and (b), such as nonconvex quadratic function and polyhedral function. We recommend that readers explore \cite{tseng2009coordinate,zhou2017unified} and the references therein for more examples and in-depth discussions on error bound condition. The last part of assumption highlights that  function g should exhibit local weak convexity near the projection of the stationary point set $\Omega^*$ onto the $\mathbf{y}$-coordinates.
\end{remark}

\subsection{Proof of \cref{thm:linear conv}}
\begin{proof}	
	For item (i), recalling from (\ref{grad conv 1}) and (\ref{obj conv1}) that there exists a constant $\zeta \geq \inf _{\mathbf{w}} \mathcal{L}_\beta(\mathbf{w})$ such that $ {{\mathbb{E}}}\mathcal{L}_\beta\left(\mathbf{w}^k\right) \leq \zeta$ for all $k$ and $\lim _{k \rightarrow \infty} \operatorname{dist}\left(\mathbf{0}, \partial \mathcal{L}_\beta\left(\mathbf{w}^k\right)\right)=0$ a.s. Thus, conclusion (i) follows from Assumption \ref{assump error bound} (a) with $\xi=\zeta$.

	As for the item (ii),  let us define $\overline{\mathbf{w}}^k \in \Omega^*$ such that, for any iterate $\mathbf{w}^k$,  $\operatorname{dist}\left(\mathbf{w}^k, \Omega^*\right)=\left\|\mathbf{w}^k-\overline{\mathbf{w}}^k\right\|$.Given the closness of the set  $\Omega^*$, such a $\overline{\mathbf{w}}^k$ must exist. Then, by virtue of conclusion (i), we have
	\begin{equation}\label{w and barw}
		\lim _{k \rightarrow \infty}\left\|\mathbf{w}^k-\overline{\mathbf{w}}^k\right\|=0 \quad \text{a.s.}  
	\end{equation}

	In addition, we obtain from (\ref{x,y 0}) and $\left\|\mathbf{w}^k-\mathbf{w}^{k-1}\right\| \leq\left\|\mathbf{d}_x^{k-1}\right\|+\left\|\mathbf{d}_y^{k-1}\right\|+\left\|\mathbf{d}_\lambda^{k-1}\right\|$ that
	\begin{equation} \label{diff of wk}
		\lim _{k \rightarrow \infty}\left\|\mathbf{w}^k-\mathbf{w}^{k-1}\right\|=0 \quad \text{a.s.}
	\end{equation}

	Therefore, from $\left\|\overline{\mathbf{w}}^k-\overline{\mathbf{w}}^{k-1}\right\| \leq\left\|\overline{\mathbf{w}}^k-\mathbf{w}^k\right\|+\left\|\mathbf{w}^k-\mathbf{w}^{k-1}\right\|+\left\|\mathbf{w}^{k-1}-\overline{\mathbf{w}}^{k-1}\right\|$, (\ref{w and barw}) and (\ref{diff of wk}), it further yields that
	\begin{align}
		\lim _{k \rightarrow \infty}\left\|\overline{\mathbf{w}}^k-\overline{\mathbf{w}}^{k-1}\right\| &=0 \quad \text{a.s.}
	\end{align}
	
	Together with Assumption \ref{assump error bound} (b) and 
	$\overline{\mathbf{w}}^k \in \Omega$, there exists a constant $\bar{F}^*$ such that for all sufficiently large $k$ 
	\begin{equation}\label{obj of bar w}
		\mathcal{L}_\beta\left(\overline{\mathbf{w}}^k\right)=\mathcal{L}_\beta\left(\overline{\mathbf{x}}^k, \overline{\mathbf{y}}^k, \bar{\lambda}^k\right)=F\left(\overline{\mathbf{x}}^k, \overline{\mathbf{y}}^k\right)=\bar{F}^* \quad \text{a.s.}
	\end{equation}

	Utilizing our assumption, the sequence $\left\{\mathbf{w}^k\right\}$possesses a cluster point denoted as $\mathbf{w}^*$. In other words, there exists a subsequence $\left\{\mathbf{w}^{k_i}\right\}$ converging to $\mathbf{w}^*$. Consequently, we deduce from Lemma \ref{lemma of cluster} that $\mathbf{w}^* \in \Omega$. Moreover, by (\ref{w and barw}), it follows that
	\begin{equation}
		\lim _{i \rightarrow \infty}\left\|\overline{\mathbf{w}}^{k_i}-\mathbf{w}^*\right\| \leq \lim _{i \rightarrow \infty}\left(\left\|\overline{\mathbf{w}}^{k_i}-\mathbf{w}^{k_i}\right\|+\left\|\mathbf{w}^{k_i}-\mathbf{w}^*\right\|\right)=0 \quad \text{a.s.}
	\end{equation}
	
	From (\ref{obj of bar w}), $\mathbf{w}^* \in \Omega^*$ and  Assumption \ref{assump error bound} (b) again, we have 
	\begin{equation}
		\mathcal{L}_\beta\left(\mathbf{w}^*\right) =\bar{F}^* \quad \text{a.s.}
	\end{equation}
	Then, by the lower semicontinuity of the function $	\mathcal{L}_\beta\left(\cdot\right)$, one has
	\begin{equation}\label{F*}
		\bar{F}^*=\mathcal{L}_\beta\left(\mathbf{w}^*\right) \leq \lim _{i \rightarrow \infty} \mathcal{L}_\beta\left(\mathbf{w}^{k_i}\right)=F^*  \quad \text{a.s.},
	\end{equation}
	where $F^*=\lim _{k \rightarrow \infty} \mathcal{P}^k=\lim _{k \rightarrow \infty} \mathcal{L}_\beta\left(\mathbf{w}^k\right) \quad \text{a.s.}$ given in (\ref{obj conv1}).

	From the definition of $\mathcal{L}_\beta\left({\mathbf{x}}, \mathbf{y}, \mathbf{\lambda} \right)$, the update rules of $\mathbf{x}$, $ \mathbf{y}$ and $\mathbf{\lambda}$, and some calculations, it gives
	\begin{equation}\label{diff 1}
		\mathcal{L}_\beta\left({\mathbf{x}}^{k}, \mathbf{y}^k, \mathbf{\lambda}^k\right)-\mathcal{L}_\beta\left({\mathbf{x}}^{k}, \mathbf{y}^k, \mathbf{\lambda}\right)=\frac{1}{s \beta}\left(\mathbf{\lambda}-\mathbf{\lambda}^k\right)^{\top}\left(\mathbf{\lambda}^{k-1}-\mathbf{\lambda}^k\right),
	\end{equation}
	
	\begin{align}\label{diff 2}
		&\mathcal{L}_\beta\left({\mathbf{x}}^{k}, \mathbf{y}^k, \lambda\right)-\mathcal{L}_\beta\left({\mathbf{x}}^{k}, \mathbf{y}, \lambda\right)\nonumber\\
		=&g\left(\mathbf{y}^k\right)-g(\mathbf{y})+\lambda^{\top} B\left(\mathbf{y}-\mathbf{y}^k\right) 
		+\frac{\beta}{2}\left(\left\|A {\mathbf{x}}^{k}+B \mathbf{y}^k-\mathbf{b}\right\|^2-\left\|A {\mathbf{x}}^{k}+B \mathbf{y}-\mathbf{b}\right\|^2\right),
	\end{align}
	
	and
	\begin{equation}\label{diff 3}
		\begin{aligned}
			&\mathcal{L}_\beta\left({\mathbf{x}}^{k}, \mathbf{y}, \boldsymbol{\lambda}\right)-\mathcal{L}_\beta(\mathbf{x}, \mathbf{y}, \boldsymbol{\lambda})\\
			=&f\left({\mathbf{x}}^{k}\right)-f(\mathbf{x})+\boldsymbol{\lambda}^{\top} A\left(\mathbf{x}-{\mathbf{x}}^{k}\right) 
			+\frac{\beta}{2}\left(\left\|A {\mathbf{x}}^{k}+B \mathbf{y}-\mathbf{b}\right\|^2-\|A \mathbf{x}+B \mathbf{y}-\mathbf{b}\|^2\right).
		\end{aligned}
	\end{equation}
	
	Summing (\ref{diff 1}), (\ref{diff 2}) and (\ref{diff 3}), and setting $(\mathbf{x}, \mathbf{y}, \boldsymbol{\lambda})=\overline{\mathbf{w}}^k$, for all sufficiently large $k$, we have from (\ref{obj of bar w}) and (\ref{F*}) that
	
	\begin{align}
		\label{Lag - F*}
		& \mathcal{L}_\beta\left({\mathbf{x}}^{k}, \mathbf{y}^k, \lambda^k\right)-F^* \nonumber\\
		\leq & \mathcal{L}_\beta\left({\mathbf{x}}^{k}, \mathbf{y}^k, \lambda^k\right)-\bar{F}^*=\mathcal{L}_\beta\left({\mathbf{x}}^{k}, \mathbf{y}^k, \lambda^k\right)-\mathcal{L}_\beta\left(\overline{\mathbf{x}}^k, \overline{\mathbf{y}}^k, \bar{\lambda}^k\right) \nonumber\\
		= & \frac{1}{s \beta}\left(\bar{\lambda}^k-\lambda^k\right)^{\top}\left(\lambda^{k-1}-\lambda^k\right)
		+ f\left(\mathbf{x}^k\right)-f\left(\overline{\mathbf{x}}^k\right)+\left\langle A^{\top} \bar{\lambda}^k, \overline{\mathbf{x}}^k-\mathbf{x}^k\right\rangle
		\nonumber\\
		& +g\left(\mathbf{y}^k\right)-g\left(\overline{\mathbf{y}}^k\right)+\left\langle B^{\top} \bar{\lambda}^k, \overline{\mathbf{y}}^k-\mathbf{y}^k\right\rangle
		+\frac{\beta}{2}\left\|A {\mathbf{x}}^{k}+B \mathbf{y}^k-\mathbf{b}\right\|^2,\nonumber\\
		= &  \frac{1}{s \beta}\left(\bar{\lambda}^k-\lambda^k\right)^{\top}\left(\lambda^{k-1}-\lambda^k\right)
		+ f\left(\mathbf{x}^k\right)-f\left(\overline{\mathbf{x}}^k\right)+\left\langle A^{\top} \bar{\lambda}^k, \overline{\mathbf{x}}^k-\mathbf{x}^k\right\rangle\nonumber\\
		&  +g\left(\mathbf{y}^k\right)-g\left(\overline{\mathbf{y}}^k\right)+\left\langle B^{\top} \bar{\lambda}^k, \overline{\mathbf{y}}^k-\mathbf{y}^k\right\rangle
		+\frac{\beta}{2}\left\|\frac{\mathbf{d}_\lambda^{k -1}}{s\beta}\right\|^2 \quad \text{a.s.},
	\end{align}
	where the first equation comes from $\overline{\mathbf{w}}^k \in \Omega^*: A \overline{\mathrm{x}}^k+B \overline{\mathbf{y}}^k=\mathrm{b}$, and the last equation comes from $\mathbf{d}_\lambda^{k-1}=-s \beta \mathbf{r}^k$.
	
	From Lipschitz continuity of $\nabla f$ and $A^{\top} \bar{\lambda}^k=\nabla f\left(\overline{\mathrm{x}}^k\right)$, we can drive 
	\begin{equation}\label{lip of f 1}
		f\left(\mathbf{x}^k\right)-f\left(\overline{\mathbf{x}}^k\right)+\left\langle A^{\top} \bar{\lambda}^k, \overline{\mathbf{x}}^k-\mathbf{x}^k\right\rangle 
		\leq
		\frac{L_f}{2} \left\|\ \overline{\mathrm{x}}^k - \mathrm{x}^k \right\|^2.
	\end{equation} 
	
	Recalling (\ref{y-update grad}), there exists a $\xi_y^k \in \partial_y \mathcal{L}_\beta\left(\mathbf{x}^{k-1}, \mathbf{y}^k, \boldsymbol{\lambda}^{k-1}\right)$, i.e.,
	$$
	\boldsymbol{\nu}^k:=\xi_y^k+B^{\top} \boldsymbol{\lambda}^{k-1}-\beta B^{\top}\left(A \mathbf{x}^{k-1}+B \mathbf{y}^k-\mathbf{b}\right) \in \partial g\left(\mathbf{y}^k\right)
	$$
	with $\left\|\xi_y^k\right\| \leq c_y \beta\left\|\mathbf{d}_y^{k-1}\right\|$. So, we have
	$$
	\begin{aligned}
		\left\|\boldsymbol{\nu}^k-B^{\top} \overline{\boldsymbol{\lambda}}^k\right\| & \leq\left\|\xi_y^k\right\|+\left\|B^{\top}\left(\boldsymbol{\lambda}^{k-1}-\overline{\boldsymbol{\lambda}}^k\right)\right\|+\beta\left\|B^{\top}\left(A \mathbf{x}^{k-1}+B \mathbf{y}^k-\mathbf{b}\right)\right\| \\
		& \leq c_y \beta\left\|\mathbf{d}_y^{k-1}\right\|+\|B\|\left(\left\|\mathbf{d}_\lambda^{k-1}\right\|+\left\|\boldsymbol{\lambda}^k-\overline{\boldsymbol{\lambda}}^k\right\|\right)+\beta\|B\|\left(\left\|{\mathbf{r}}^k\right\|+\left\|A {\mathbf{d}}_x^{k-1}\right\|\right).
	\end{aligned}
	$$
	
	Now, using Assumption \ref{assump error bound} (c), it follows that 
	\begin{equation}
		g\left(\mathbf{y}^k\right)-g\left(\overline{\mathbf{y}}^k\right) + \left\langle\boldsymbol{\nu}^k, \overline{\mathbf{y}}^k-\mathbf{y}^k\right\rangle
		\leq \sigma\left\|\overline{\mathbf{y}}^k-\mathbf{y}^k\right\|^2.  \nonumber
	\end{equation}
	
	Then inserting the above inequality and (\ref{lip of f 1}) into (\ref{Lag - F*}) and the inequality $ab\leq \frac{a^2+b^2}{2}$, we obtain
	\begin{align}\label{lag diff 2}
		& \mathcal{L}_\beta\left({\mathbf{x}}^{k}, \mathbf{y}^k, \lambda^k\right)-F^* \nonumber\\
		\leq 
		& \frac{1}{s \beta}\left(\bar{\lambda}^k-\lambda^k\right)^{\top}\left(\lambda^{k-1}-\lambda^k\right)
		+ \frac{L_f}{2} \left\|\ \overline{\mathrm{x}}^k - \mathrm{x}^k \right\|^2 + 
		\frac{1}{2\beta s^2}\left\|
		\mathbf{d}_\lambda^{k -1}
		\right\|^2 + \sigma\left\|\overline{\mathbf{y}}^k-\mathbf{y}^k\right\|^2
		\nonumber\\
		& +  \left\|\ \boldsymbol{\nu}^k-B^{\top} \overline{\boldsymbol{\lambda}}^k  \right\| \left\|\overline{\mathbf{y}}^k-\mathbf{y}^k\right\| \nonumber\\
		\leq &\left(\frac{1}{2s \beta} + \frac{ \left\|\ B \right\|^2}{2} \right)\left\|\ \bar{\lambda}^k-\lambda^k\right\|^{2} + \frac{L_f}{2} \left\|\ \overline{\mathrm{x}}^k - \mathrm{x}^k \right\|^2 +
		\left( \sigma + \frac{5}{2} \right)\left\|\overline{\mathbf{y}}^k-\mathbf{y}^k\right\|^2\nonumber\\
		&+\left(\frac{1}{2\beta s} + \frac{1}{2\beta s^2} +
		\frac{ \left\|\ B \right\|^2}{2} +
		\frac{ \left\|\ B \right\|^2}{2s^2}
		\right)\left\|
		\mathbf{d}_\lambda^{k -1}
		\right\|^2  + \frac{c_y^2 \beta^2}{2}
		\left\|\mathbf{d}_y^{k-1}\right\|^2 + \frac{\beta^2 \left\|\ A \right\|^2 \left\|\ B \right\|^2}{2}
		\left\|\mathbf{d}_x^{k-1}\right\|^2\quad \text{a.s.}
	\end{align}
	
	Denote
	$
	\begin{cases}
		C_1 = &
		\max \left\{\frac{1}{2\beta s} + \frac{1}{2\beta s^2} +
		\frac{ \left\|\ B \right\|^2}{2} +
		\frac{ \left\|\ B \right\|^2}{2s^2},
		\frac{c_y^2 \beta^2}{2},
		\frac{\beta^2 \left\|\ A \right\|^2 \left\|\ B \right\|^2}{2}
		\right\}
		\\
		C_2 = &\max \left\{
		\frac{1}{2s \beta} + \frac{ \left\|\ B \right\|^2}{2},
		\frac{L_f}{2},
		\sigma + \frac{5}{2}
		\right\} 
	\end{cases}
	$, we combine these denotations and (\ref{lag diff 2}) to get
	\begin{equation}\label{Lag - F*2}
		\mathcal{L}_\beta\left({\mathbf{x}}^{k}, \mathbf{y}^k, \lambda^k\right)-F^* \leq C_1 \left(
		\left\|\mathbf{d}_x^{k-1}\right\|^2 + \left\|\mathbf{d}_y^{k-1}\right\|^2 +
		\left\|\mathbf{d}_\lambda^{k-1}\right\|^2
		\right) + C_2 \left\|\overline{\mathbf{w}}^k-{\mathbf{w}}^{k}\right\|^2\quad \text{a.s.}
	\end{equation}
	
	Using the results in (\ref{partial grad of Lag}) and Assumption \ref{assump error bound} (a) again, it gives
	\begin{equation}
		\mathbb{E}\operatorname{dist}\left(\mathbf{0}, \partial \mathcal{L}_\beta\left(\mathbf{w}^k\right)\right) \leq
		\mathbb{E}\left\{
		\left\|\ \xi_x^{k} \right\|
		+ \left(\left\|\ A \right\|+ \left\|\ B \right\| +\frac{1}{s\beta}
		\right)\left\|\  \mathbf{d}_\lambda^{k-1} \right\| + \beta c_y \left\|\  \mathbf{d}_y^{k-1} \right\|
		+\beta \left\|\ B^{\top}A
		\right\|\left\|\  \mathbf{d}_x^{k-1} \right\|
		\right\},\nonumber
	\end{equation}
	
	\begin{align}
		&\mathbb{E}\left\|\mathbf{w}^k-\overline{\mathbf{w}}^k\right\|=\mathbb{E}
		\operatorname{dist}\left(\mathbf{w}^k, \Omega\right) \leq \tau \mathbb{E}\operatorname{dist}\left(\mathbf{0}, \partial \mathcal{L}_\beta\left(\mathbf{w}^k\right)\right) \nonumber\\
		\leq 
		&\tau \mathbb{E}\left\{
		\left\|\ \xi_x^{k} \right\| + \left(\left\|\ A \right\|+ \left\|\ B \right\| +\frac{1}{s\beta}
		\right)\left\|\  \mathbf{d}_\lambda^{k-1} \right\| + \beta c_y \left\|\  \mathbf{d}_y^{k-1} \right\|
		+\beta \left\|\ B^{\top}A
		\right\|\left\|\  \mathbf{d}_x^{k-1} \right\|
		\right\}.
		\label{diff of w}
	\end{align}
	
	Combining (\ref{Lag - F*2}) with (\ref{diff of w}) to drive
	\begin{equation}\label{Lag-F*3}
		\begin{aligned}
			& \mathbb{E} \mathcal{L}_\beta\left(\mathbf{x}^k, \mathbf{y}^k, \lambda^k\right)-F^* \\
			\leq & C_1
			\mathbb{E}\left(\left\|\mathbf{d}_x^{k-1}\right\|^2+\left\|\mathbf{d}_y^{k-1}\right\|^2+\left\|\mathbf{d}_\lambda^{k-1}\right\|^2\right)+4 C_2 \tau^2
			\mathbb{E}
			\left[\left\|\xi_x^k\right\|^2+\left(\|A\|+\|B\|+\frac{1}{s \beta}\right)^2\left\|\mathbf{d}_\lambda^{k-1}\right\|\right. \\
			& \left.+\beta^2 c_y^2\left\|\mathbf{d}_y^{k-1}\right\|^2+\beta^2\left\|B^{\top} A\right\|^2\left\|\mathbf{d}_x^{k-1}\right\|^2\right]\\
			\leq& C_3  \mathbb{E}\left(\left\|\mathbf{d}_x^{k-1}\right\|^2
			+
			\left\|\mathbf{d}_y^{k-1}\right\|^2+
			\left\|\mathbf{d}_\lambda^{k-1}\right\|^2
			\right) + 4C_2\tau^2c_x^2\beta^2 \frac{\sigma^2}{M}\quad \text{a.s.},
		\end{aligned}
	\end{equation}
	where the positive constant
	$$\begin{aligned}
		C_3 = \max \Biggl\{ 
		&C_1 +4C_2\tau^2\beta^2\beta^2 \left\|\ B^{\top}A \right\|^2+4C_2\tau^2c_x^2\beta^2, C_1+4C_2\tau^2c_y^2\beta^2+
		4C_2\tau^2c_x^2\beta^2,\nonumber\\
		&C_1+4C_2\tau^2\left(\left\|A \right\| + \left\|B \right\| + \frac{1}{s\beta}
		\right)^2
		\Biggr\}.
	\end{aligned}
	$$
	Hence, defining
	$\mathbf{d}^k=\left\|{\mathbf{d}}_x^k\right\|^2+\left\|\mathbf{d}_y^k\right\|^2+\left\|\mathbf{d}_\lambda^k\right\|^2$, adding
	$\hat{A} \mathbb{E}\left\|\ \mathbf{d}_x^{k-1}  \right\|^2  +  \hat{A} \mathbb{E}\left\|\ \mathbf{d}_y^{k-1}  \right\|^2  + \frac{1+\tau}{s\beta\sigma_{A}}\psi_2(s)  \mathbb{E}\left\|\  A\mathbf{d}_\lambda^{k-1} \right\|^2
	$ on both sides of (\ref{Lag-F*3}) and recalling the definition of $\mathcal{P}^{k}$, we obtain
	\begin{align}
		&\mathbb{E} \mathcal{P}^k-F^* \nonumber\\
		\leq& \mathbb{E} \mathcal{L}_\beta\left(\mathbf{x}^k, \mathbf{y}^k, \lambda^k\right)-F^* + 
		\hat{A} \mathbb{E}\left\|\ \mathbf{d}_x^{k-1}  \right\|^2  +  \hat{A} \mathbb{E}\left\|\ \mathbf{d}_y^{k-1}  \right\|^2  + \frac{1+\tau}{s\beta\sigma_{A}}\psi_2(s)  \mathbb{E}\left\|\  A\mathbf{d}_\lambda^{k-1 } \right\|^2 \nonumber\\
		\leq & 
		C_3  \mathbb{E}\left(\left\|\mathbf{d}_x^{k-1}\right\|^2
		+
		\left\|\mathbf{d}_y^{k-1}\right\|^2+
		\left\|\mathbf{d}_\lambda^{k-1}\right\|^2
		\right) + 4C_2\tau^2c_x^2\beta^2 \frac{\sigma^2}{M}+
		\hat{A} \mathbb{E}\left\|\ \mathbf{d}_x^{k-1}  \right\|^2  \nonumber\\
		&+  \hat{A} \mathbb{E}\left\|\ \mathbf{d}_y^{k-1}  \right\|^2  + \frac{1+\tau}{s\beta\sigma_{A}}\psi_2(s)  \mathbb{E}\left\|\  A\mathbf{d}_\lambda^{k-1} \right\|^2 \nonumber\\
		\leq & \bar{C} \mathbf{d}^{k-1}+
		4C_2\tau^2c_x^2\beta^2 \frac{\sigma^2}{M}\quad \text{a.s.},
	\end{align}
	where $\bar{C}=\max \left\{
	C_3+\hat{A},C_3+\frac{1+\tau}{s\beta\sigma_{A}}\psi_2(s) \left\|\ A \right\|^2
	\right\}.$
	
	Reusing (\ref{x+y+lam}), it follows from the above result that 
	\begin{align}
		\mathbb{E} \mathcal{P}^k-F^* \leq\frac{\bar{C}}{\mu}
		\left[\mathbb{E}\mathcal{P}^{k-1}-\mathbb{E}\mathcal{P}^k +\left(\frac{ (\hat{c}_x \beta)^2}{2}
		+8\frac{(1+\tau)c_x^2 \beta^{2}}{s\beta\sigma_{A}}\psi_1(s)\right)\frac{\sigma^2}{M}
		\right] + 4C_2\tau^2c_x^2\beta^2 \frac{\sigma^2}{M},\nonumber
	\end{align}
	
	\begin{align}
		\left(1+\frac{\bar{C}}{\mu}
		\right)\left(\mathbb{E} \mathcal{P}^k-F^*
		\right)\leq \frac{\bar{C}}{\mu}\left(\mathbb{E} \mathcal{P}^{k-1}-F^*
		\right)+\left[\frac{\bar{C}}{\mu}\left(\frac{ (\hat{c}_x \beta)^2}{2}
		+8\frac{(1+\tau)c_x^2 \beta^{2}}{s\beta\sigma_{A}}\psi_1(s)\right)+
		4C_2\tau^2c_x^2\beta^2 
		\right]\frac{\sigma^2}{M},\nonumber
	\end{align}
	and
	\begin{align}
		\mathbb{E} \mathcal{P}^k-F^*
		\leq \frac{\bar{C}}{\mu+\bar{C}}\left(\mathbb{E} \mathcal{P}^{k-1}-F^*
		\right)+\hat{C}\frac{\sigma^2}{M}\quad \text{a.s.},
	\end{align}
	where $\hat{C}:=
	\frac{\left[\frac{\bar{C}}{\mu}\left(\frac{ (\hat{c}_x \beta)^2}{2}
		+8\frac{(1+\tau)c_x^2 \beta^{2}}{s\beta\sigma_{A}}\psi_1(s)\right)+
		4C_2\tau^2c_x^2\beta^2 
		\right]}{\left(1+\frac{\bar{C}}{\mu}\right)}.
	$
	
	Then, we can rearrange this equation and further obtain
	\begin{align}
		&\left[\mathbb{E} \mathcal{P}^k-F^* - \hat{C}\frac{\sigma^2}{M}\left(1+\frac{\bar{C}}{\mu}
		\right)
		\right]\leq 
		\frac{\bar{C}}{\mu+\bar{C}}
		\left[\mathbb{E} \mathcal{P}^{k-1}-F^*-
		\hat{C}\frac{\sigma^2}{M}\left(1+\frac{\bar{C}}{\mu}
		\right)
		\right]\nonumber\\
		& \mathbb{E} \mathcal{P}^k-F^* \leq \left(\frac{\bar{C}}{\mu+\bar{C}}
		\right)^k \left(\mathbb{E} \mathcal{P}^0-F^*
		\right)+ \left(1-\left(\frac{\bar{C}}{\mu+\bar{C}}
		\right)^k
		\right)\hat{C}\frac{\sigma^2}{M}\left(1+\frac{\bar{C}}{\mu}\right)\quad \text{a.s.}
	\end{align}
	With the notations $\tilde{C}=\frac{\bar{C}}{\mu+\bar{C}}$ and $\breve{C}=\hat{C}\left(1+\frac{\bar{C}}{\mu}
	\right)$, we complete the proof.
	
\end{proof}

\subsection{Hybrid SARAH with Bounded Variance}\label{app:hy sto}
We begin by revisiting the property of the hybrid stochastic gradient estimator, as stated in \cite{tran2022hybrid}:
\begin{lemma}[Lemma 2.2 in \cite{tran2022hybrid}]
	Assume that $f(\cdot, \cdot)$ is $L$-smooth and $u_t$ represents an unbiased SGD estimator. Consequently, we derive the upper bound on the ``variance" $\mathbb{E}\left[\left\|v_t-\nabla f\left(x_t\right)\right\|^2\right]$ for $v_t$ :
	$$
	\mathbb{E}\left[\left\|v_t-\nabla f\left(x_t\right)\right\|^2\right] \leq \omega_t \mathbb{E}\left[\left\|v_0-\nabla f\left(x^0\right)\right\|^2\right]+L^2 \sum_{i=0}^{t-1} \omega_{i, t} \mathbb{E}\left[\left\|x_{i+1}-x_i\right\|^2\right]+S_t,
	$$
	where the expectation is taking over all the randomness $\mathcal{F}_t:=\sigma\left(v_0, v_1, \cdots, v_t\right), \omega_t:=\prod_{i=1}^t \alpha_{i-1}^2$, $\omega_{i, t}:=\prod_{j=i+1}^t \alpha_{j-1}^2$ for $i=0, \cdots, t$, and $S_t:=\sum_{i=0}^{t-1}\left(\prod_{j=i+2}^t \alpha_{j-1}^2\right)\left(1-\alpha_i\right)^2 
	{{\sigma^2}}
	$ for $t \geq 0$.
\end{lemma}

By combining the findings from Lemma 2 in \cite{tran2022hybrid} with $\alpha_{t} = \alpha \in \left[0,1\right]$, we have the subsequent outcome related to the hybrid stochasitc estimator $\tilde{v}_t$ in our Algorithm \ref{alg 5.1}:
\begin{equation}\label{hybrid var up bound}
	\mathbb{E}\left[\left\|\tilde{v}_t-\nabla h\left(\widehat{\mathbf{x}}_t\right)\right\|^2\right] \leq \alpha^{2t} \mathbb{E}\left[\left\|\tilde{v}_0-\nabla h\left(\widehat{\mathbf{x}}_0\right)\right\|^2\right]+\Lambda^2 \sum_{i=0}^{t-1} \alpha^{2(t-i)} \mathbb{E}\left[\left\|\widehat{\mathbf{x}}_{i+1}-\widehat{\mathbf{x}}_i\right\|^2\right]+ \frac{1-\alpha}{1+\alpha}{{ \sigma^2}}.
\end{equation}

\begin{remark}
	\begin{itemize}
		\item[1.]  
		From equation (\ref{def of tilde v}), it's evident that $\tilde{v}_t$ remains a biased estimator when $\alpha \in \left(0,1\right]$. Although $\tilde{v}_t$ is a biased estimator, we still use the term ``variance" to denote $\mathbb{E}\left[\left\|\tilde{v}_t-\nabla h\left(\widehat{\mathbf{x}}_t\right)\right\|^2\right]$.

		\item[2.]\textbf{(Variance-bias trade-off \cite{tran2022hybrid})}
		When $\alpha \in (0, 1)$, the bias of $\tilde{v}_t$ can be formulated as 
		\begin{equation}\label{trade-off}
			\begin{aligned}
				\operatorname{Bias}\left[\tilde{v}_t \mid \mathcal{F}_t\right]= & \left\|\mathbb{E}_{\left(\xi_t, \zeta_t\right)}\left[ \tilde{v}_t-\nabla h\left(\widehat{\mathbf{x}}_t\right)
				\mid \mathcal{F}_t\right]\right\|=\alpha \left\| \tilde{v}_{t-1}-\nabla h\left(\widehat{\mathbf{x}}_{t-1}\right)
				\right\| \\
				& <\left\|\tilde{v}_{t-1}-\nabla h\left(\widehat{\mathbf{x}}_{t-1}\right)\right\| .
			\end{aligned}
		\end{equation}
		Notably, the bias of $\tilde{v}_t$ is  smaller than that of $\tilde{v}_{t-1}$. Regarding the SARAH estimator, defined as
		$v_{t}^{\mathrm{sarah}} := v_{t-1}^{\mathrm{sarah}}+\nabla h\left(\widehat{\mathbf{x}}_t ; \xi_t\right)
		-\nabla h\left(\widehat{\mathbf{x}}_{t-1} ; \xi_{t}\right)$ with $\operatorname{Bias}\left[v_t^{\text {sarah }} \mid \mathcal{F}_t\right]=\left\|v_{t-1}^{\text {sarah }}-\nabla h\left(\widehat{\mathbf{x}}_{t-1}\right)\right\|$.
		This leads to the conclusion that the bias of hybrid estimator is smaller than that of SARAH estimator. Furthermore, by analyzing (\ref{hybrid var up bound}) and (\ref{trade-off}), it indicates that the parameter $\alpha$ plays a crucial role in regulating the trade-off between bias and variance.

	\end{itemize}
\end{remark}

\subsection{Proof of \cref{thm 4.3}} \label{app: pf of inner grad}

\textbf{Theorem 5.4}
Assume that Assumption \ref{assum of func h} and conditions in \cref{x y w grad 0} hold. Let $\left\{\left(\widehat{\mathbf{x}}_t,\breve{\mathbf{x}}_t,\mathbf{x}_t \right) \right\}$ be generated by Algorithm \ref{alg 5.1}, we have
\begin{align}
	&\frac{1}{m+1}\left[ \sum_{t=0}^{m}\frac{\beta_t}{2}
	\mathbb{E}\left[ \left\|\breve{\mathbf{x}}_{t+1}-\mathbf{x}_t\right\|_{\mathcal{M}}^2\right]
	- \sum_{t=0}^{m} \Gamma_t
	\mathbb{E}\left[ \left\|\mathbf{s}_t\right\|^2\right]
	\right]\nonumber\\
	\leq &	\frac{\mathbf{H}_{0}-\mathbf{H}_{m+1}}{m+1} + \frac{\sigma^{2-\frac{k_1}{2}}}{2{r}(1+{\alpha})}(c_1+\frac{1}{c_1})\frac{1}{M(m+1)},\nonumber
\end{align}

where $\mathcal{M}=\beta A^{\top} A$, $\Gamma_t=\frac{\gamma_t \beta_t - \Lambda \beta_t^2}{2}$, $\mathbf{s}_t=\breve{\mathbf{x}}_{t+1}-\breve{\mathbf{x}}_t$, $k_1$ and $c_1$ are positive constants, and $\gamma_{t}$ and $\alpha$ are given by
\begin{align}
	\gamma_{t} &: = \beta_{t}(\frac{\mu + 2g(\alpha)}{2\tau -\tau^2} +\mu) \frac{t+1}{t} \leq \beta_t\left[\frac{\mu +2\sqrt{ \frac{2l_3}{c_1}} \Lambda \left[ M(m+1)   \right]^{\frac{1}{4} }}{\tau(2-\tau)}+\mu   \right] \frac{t+1}{t},\nonumber\\
	\alpha &:=1-\frac{c_1}{\sqrt{M(m+1)}} \in \left(\underline{\alpha}, 1\right), \quad \quad
	\underline{r}=\sqrt{\frac{2l_3 \Lambda^2 \underline{\alpha}^2}{  (1-^2)}}
\end{align}

\begin{remark}
	We make the following remarks:
	\begin{itemize}
		\item The parameter $1/\gamma_t$ serves as the step size.
		
		\item  The selection of the hybrid parameter, $\alpha$, is influenced by the batch size $M$ used for the initial $\tilde{v}_0$ and the number of inner iterations $m$. It is noticeable that as the number of inner iterations increases, the choice of $\alpha$ tends to converge towards 1.
	\end{itemize}
\end{remark}

\begin{proof}(Proof of Theorem \ref{thm 4.3})
	
	From the updates of ${\mathbf{x}}_{t+1}$ and $\widehat{\mathbf{x}}_t$, we have 
	$$\beta_t\left(\breve{\mathbf{x}}_{t+1}-\widehat{\mathbf{x}}_t\right)+\left(1-\beta_t\right)\left(\mathbf{x}_t-\widehat{\mathbf{x}}_t\right)=\mathbf{x}_{t+1}-\widehat{\mathbf{x}}_t=\beta_t(\breve{\mathbf{x}}_{t+1}-\breve{\mathbf{x}}_t).$$
	
	Then, using the \eqref{h lip} and the above relation, it gives
	\begin{equation}\label{h descent}
		\begin{aligned}
			h\left(\mathbf{x}_{t+1}\right) & \leq h\left(\widehat{\mathbf{x}}_t\right)+\left\langle\nabla h\left(\widehat{\mathbf{x}}_t\right), \mathbf{x}_{t+1}-\widehat{\mathbf{x}}_t\right\rangle+\frac{\Lambda}{2}\left\|\mathbf{x}_{t+1}-\widehat{\mathbf{x}}_t\right\|^2 \\
			& =h\left(\widehat{\mathbf{x}}_t\right)+\left\langle\nabla h\left(\widehat{\mathbf{x}}_t\right), \mathbf{x}_t-\widehat{\mathbf{x}}_t\right\rangle+\left\langle\nabla h\left(\widehat{\mathbf{x}}_t\right), \mathbf{x}_{t+1}-\mathbf{x}_t\right\rangle+\frac{\Lambda \beta_t^2}{2}\left\|\mathbf{s}_t\right\|^2 \\
			& \leq h\left(\mathbf{x}_t\right)+\frac{\mu}{2}\left\|\mathbf{x}_t-\widehat{\mathbf{x}}_t\right\|^2+\left\langle\nabla h\left(\widehat{\mathbf{x}}_t\right), \mathbf{x}_{t+1}-\mathbf{x}_t\right\rangle+\frac{\Lambda \beta_t^2}{2}\left\|\mathbf{s}_t\right\|^2,
		\end{aligned}
	\end{equation}
	where $\mathbf{s}_t=\breve{\mathbf{x}}_{t+1}-\breve{\mathbf{x}}_t$.

	Combining \eqref{h descent}, $\mathbf{x}_{t+1}=\beta_t \breve{\mathbf{x}}_{t+1}+\left(1-\beta_t\right) \mathbf{x}_t$, and the convexity of function $\phi$, it's derived that
	\begin{equation}\label{equ 5.8}
		\begin{aligned}
			& \Phi\left(\mathbf{x}_{t+1}\right)=h\left(\mathbf{x}_{t+1}\right)+\phi\left(\mathbf{x}_{t+1}\right) \\
			\leq & \beta_t\left[h\left(\mathbf{x}_t\right)+\left\langle\nabla h\left(\widehat{\mathbf{x}}_t\right), \breve{\mathbf{x}}_{t+1}-\mathbf{x}_t\right\rangle+\phi\left(\breve{\mathbf{x}}_{t+1}\right)\right]+\left(1-\beta_t\right)\left[h\left(\mathbf{x}_t\right)+\phi\left(\mathbf{x}_t\right)\right] \\
			& +\frac{\mu}{2}\left\|\mathbf{x}_t-\widehat{\mathbf{x}}_t\right\|^2+\frac{\Lambda \beta_t^2}{2}\left\|\mathbf{s}_t\right\|^2 \\
			= & \beta_t\left[h\left(\mathbf{x}_t\right)+\left\langle\nabla h\left(\widehat{\mathbf{x}}_t\right), \breve{\mathbf{x}}_{t+1}-\mathbf{x}_t\right\rangle+\frac{\gamma t}{2}\left\|\mathbf{s}_t\right\|^2+\phi\left(\breve{\mathbf{x}}_{t+1}\right)\right] \\
			& +\left(1-\beta_t\right) \Phi\left(\mathbf{x}_t\right)+\frac{\mu}{2}\left\|\mathbf{x}_t-\widehat{\mathbf{x}}_t\right\|^2+\frac{\Lambda \beta_t^2-\gamma_t \beta_t}{2}\left\|\mathbf{s}_t\right\|^2 .
		\end{aligned}
	\end{equation}
	

	Then, it follows from the update rule of $\breve{\mathbf{x}}_{t+1}$
	\begin{equation} 
		\breve{\mathbf{x}}_{t+1}=\arg \min \left\{\left\langle \tilde{v}_t , \mathbf{x}\right\rangle+\frac{\gamma_t}{2}\left\|\mathbf{x}-\breve{\mathbf{x}}_t\right\|^2+\phi(\mathbf{x})\right\},
	\end{equation}
	that 
	\begin{equation}
		\begin{aligned}
			& \left\langle\tilde{v}_t, \breve{\mathbf{x}}_{t+1}-\mathbf{x}_t\right\rangle+\frac{\gamma_t}{2}\left\|\mathbf{s}_t\right\|^2+\phi\left(\breve{\mathbf{x}}_{t+1}\right) \\
			\leq & \frac{\gamma_t}{2}\left(\left\|\mathbf{x}_t-\breve{\mathbf{x}}_t\right\|^2-\left\|\mathbf{x}_t-\breve{\mathbf{x}}_{t+1}\right\|^2\right)+\phi\left(\mathbf{x}_t\right)-\frac{1}{2}\left\|\mathbf{x}_t-\breve{\mathbf{x}}_{t+1}\right\|_{\mathcal{M}}^2,
		\end{aligned} \label{equ 5.10}
	\end{equation}
	where $\mathcal{M}=\beta A^{\top} A$, and \begin{equation}
		\tilde{v}_t+\gamma_t \mathbf{s}_t+\nabla \phi\left(\breve{\mathbf{x}}_{t+1}\right)=0.\label{x-opti condi}
	\end{equation}
	
	Combining (\ref{equ 5.8}) and (\ref{equ 5.10}), we have
	\begin{align}
		\Phi\left(\mathbf{x}_{t+1}\right)	\leq & \Phi\left(\mathbf{x}_t\right)+\frac{\mu}{2}\left\|\mathbf{x}_t-\widehat{\mathbf{x}}_t\right\|^2+\frac{\beta_t \gamma_t}{2}\left(\left\|\mathbf{x}_t-\breve{\mathbf{x}}_t\right\|^2-\left\|\mathbf{x}_t-\breve{\mathbf{x}}_{t+1}\right\|^2\right)\nonumber \\
		& -\frac{\beta_t}{2}\left\|\mathbf{x}_t-\breve{\mathbf{x}}_{t+1}\right\|_{\mathcal{M}}^2-\frac{\gamma_t \beta_t - \Lambda \beta_t^2}{2}\left\|\mathbf{s}_t\right\|^2 + 
		+\beta_t\left\langle\nabla h\left(\widehat{\mathbf{x}}_t\right) - \tilde{v}_t, \breve{\mathbf{x}}_{t+1}-\mathbf{x}_t\right\rangle, \label{psi value1}
	\end{align}
	Now, it's noted that 
	\begin{equation} \label{x-rela 1}
		\breve{\mathbf{x}}_t-\mathbf{x}_t=\frac{1}{\beta_t}\left(\widehat{\mathbf{x}}_t-\mathbf{x}_t\right) \quad \text { and } \quad \breve{\mathbf{x}}_{t+1}-\mathbf{x}_t=\frac{1}{\beta_t}\left(\mathbf{x}_{t+1}-\mathbf{x}_t\right),
	\end{equation}
	
	\begin{align} \label{x-rela 2}
		\widehat{\mathbf{x}}_t-\mathbf{x}_t & =\beta_t\left(\breve{\mathbf{x}}_t-\mathbf{x}_t\right)=\beta_t\left(\breve{\mathbf{x}}_t-\mathbf{x}_{t-1}+\mathbf{x}_{t-1}-\mathbf{x}_t\right)\nonumber \\
		& =\beta_t\left(\frac{1}{\beta_{t-1}}\left(\mathbf{x}_t-\mathbf{x}_{t-1}\right)+\mathbf{x}_{t-1}-\mathbf{x}_t\right)\nonumber \\
		& =\theta_t\left(\mathbf{x}_t-\mathbf{x}_{t-1}\right),
	\end{align}
	where $\theta_t=\frac{\beta_t}{\beta_{t-1}}\left(1-\beta_{t-1}\right)$. 
	
	Following from (\ref{x-rela 1}) and (\ref{x-rela 2}), (\ref{psi value1}) becomes
	\begin{align}
		\Phi\left(\mathbf{x}_{t+1}\right) \leq & \Phi\left(\mathbf{x}_t\right)+\frac{\left(\gamma_t / \beta_t+\mu\right) \theta_t^2}{2}\left\|\mathbf{x}_t-\mathbf{x}_{t-1}\right\|^2-\frac{\gamma_t / \beta_t}{2}\left\|\mathbf{x}_{t+1}-\mathbf{x}_t\right\|^2 \nonumber\\
		& -\frac{\beta_t}{2}\left\|\breve{\mathbf{x}}_{t+1}-\mathbf{x}_t\right\|_{\mathcal{M}}^2-\frac{\gamma_t \beta_t - \Lambda \beta_t^2}{2}\left\|\mathbf{s}_t\right\|^2 +  \beta_t \left\langle
		\nabla h\left(\widehat{\mathbf{x}}_t\right)-\tilde{v}_t, \breve{\mathbf{x}}_{t+1}-\mathbf{x}_t\right\rangle.
	\end{align}
	
	Taking the expectation over the entire history random variables up to the  t-th iteration, and using the Cauchy-Schwartz inequality over the term $\beta_t\left\langle
	\nabla h\left(\widehat{\mathbf{x}}_t\right)-\tilde{v}_t, \breve{\mathbf{x}}_{t+1}-\mathbf{x}_t\right\rangle$, one has
	\begin{align}
		\mathbb{E}\left[\Phi\left(\mathbf{x}_{t+1}\right)\right] \leq & \mathbb{E}\left[\Phi\left(\mathbf{x}_t\right)\right]
		+\frac{\left(\gamma_t / \beta_t+\mu\right) \theta_t^2}{2} \mathbb{E}\left[\left\|\mathbf{x}_t-\mathbf{x}_{t-1}\right\|^2 \right]-\frac{\gamma_t / \beta_t}{2} \mathbb{E}\left[\left\|\mathbf{x}_{t+1}-\mathbf{x}_t\right\|^2 \right] \nonumber \\
		& -\frac{\beta_t}{2}\mathbb{E}\left[\left\|\breve{\mathbf{x}}_{t+1}-\mathbf{x}_t\right\|_{\mathcal{M}}^2 \right]-\frac{\gamma_t \beta_t - \Lambda \beta_t^2}{2}\mathbb{E}\left[\left\|\mathbf{s}_t\right\|^2\right] \nonumber\\
		&+  \mathbb{E}\left[\beta_t \left\langle
		\nabla h\left(\widehat{\mathbf{x}}_t\right)-\tilde{v}_t, \breve{\mathbf{x}}_{t+1}-\mathbf{x}_t\right\rangle\right], \label{exp phi0}
	\end{align}
	and 
	\begin{align}
		\mathbb{E}\left[ \beta_t\left\langle
		\nabla h\left(\widehat{\mathbf{x}}_t\right)-\tilde{v}_t, \breve{\mathbf{x}}_{t+1}-\mathbf{x}_t\right\rangle \right] 
		= &  \mathbb{E}\left[ \left\langle
		\nabla h\left(\widehat{\mathbf{x}}_t\right)-\tilde{v}_t, \mathbf{x}_{t+1}-\mathbf{x}_t\right\rangle \right]\nonumber\\
		\leq& \frac{1}{2r_t}  \mathbb{E}\left[\left\|\ \nabla h\left(\widehat{\mathbf{x}}_t\right)-\tilde{v}_t  \right\|^2 \right] + \frac{r_t}{2}  \mathbb{E}\left[\left\|\ \mathbf{x}_{t+1}-\mathbf{x}_t  \right\|^2 \right],\label{exp inner pro}
	\end{align} 
	where the equation is from (\ref{x-rela 1}), $r_t$ is any positive real number. 
	
	Combining (\ref{hybrid var up bound}) and (\ref{exp inner pro}),
	we obtain
	\begin{align}
		&\mathbb{E}\left[ \beta_t\left\langle
		\nabla h\left(\widehat{\mathbf{x}}_t\right)-\tilde{v}_t, \breve{\mathbf{x}}_{t+1}-\mathbf{x}_t\right\rangle \right] \nonumber\\
		\leq & \frac{1}{2r_t} \left\{ \alpha^{2t} \mathbb{E}\left[\left\|\tilde{v}_0-\nabla h\left(\widehat{\mathbf{x}}_0\right)\right\|^2\right]+\Lambda^2 \sum_{i=0}^{t-1} \alpha^{2(t-i)} \mathbb{E}\left[\left\|\widehat{\mathbf{x}}_{i+1}-\widehat{\mathbf{x}}_i\right\|^2\right]+ \frac{1-\alpha}{1+\alpha}{{\sigma^2}}
		\right\}\nonumber\\
		&+ \frac{r_t}{2}  \mathbb{E}\left[ \left\|\ \mathbf{x}_{t+1} - \mathbf{x}_{t}  \right\|^2 \right]. \label{inner hybrid up bound}
	\end{align}
	From the above inequality (\ref{exp phi0}) and (\ref{inner hybrid up bound}), we have 
	\begin{align}
		\mathbb{E}\left[\Phi\left(\mathbf{x}_{t+1}\right)\right] \leq & \mathbb{E}\left[\Phi\left(\mathbf{x}_t\right)\right]
		+\frac{\left(\gamma_t / \beta_t+\mu\right) \theta_t^2}{2} \mathbb{E}\left[\left\|\mathbf{x}_t-\mathbf{x}_{t-1}\right\|^2 \right]-\frac{\gamma_t / \beta_t}{2} \mathbb{E}\left[\left\|\mathbf{x}_{t+1}-\mathbf{x}_t\right\|^2 \right] \nonumber \\
		& -\frac{\beta_t}{2}\mathbb{E}\left[\left\|\breve{\mathbf{x}}_{t+1}-\mathbf{x}_t\right\|_{\mathcal{M}}^2 \right]-\frac{\gamma_t \beta_t - \Lambda \beta_t^2}{2}\mathbb{E}\left[\left\|\mathbf{s}_t\right\|^2\right]+ \frac{r_t}{2}  \mathbb{E}\left[ \left\|\ \mathbf{x}_{t+1} - \mathbf{x}_{t}  \right\|^2 \right]\nonumber\\
		&+ \frac{1}{2r_t} \left\{ \alpha^{2t} \mathbb{E}\left[\left\|\tilde{v}_0-\nabla h\left(\widehat{\mathbf{x}}_0\right)\right\|^2\right]+\Lambda^2 \sum_{i=0}^{t-1} \alpha^{2(t-i)} \mathbb{E}\left[\left\|\widehat{\mathbf{x}}_{i+1}-\widehat{\mathbf{x}}_i\right\|^2\right]+ \frac{1-\alpha}{1+\alpha}{{\sigma^2}}\right\}. \label{exp phi and var}
	\end{align} 
	
	Denoting $\eta_t=\left(\gamma_t / \beta_t+\mu\right) \theta_t^2$, $\gamma_t / \beta_t-\gamma_{t+1} / \beta_{t+1}>0$, and from (\ref{exp phi and var}) it
	further indicates
	\begin{align}
		& \mathbb{E}\left[\Phi\left(\mathbf{x}_{t+1}\right)\right]+\frac{\eta_{t+1}}{2} \mathbb{E}\left[\left\|\mathbf{x}_{t+1}-\mathbf{x}_t\right\|^2\right] \nonumber\\
		\leq & \mathbb{E}\left[\Phi\left(\mathbf{x}_{t}\right)\right]+\frac{\eta_t}{2}\mathbb{E}\left[\left\|\mathbf{x}_t-\mathbf{x}_{t-1}\right\|^2\right]-\frac{\gamma_{t+1} / \beta_{t+1}-\eta_{t+1}}{2}\mathbb{E}\left[\left\|\mathbf{x}_{t+1}-\mathbf{x}_t\right\|^2 \right] \nonumber\\
		& -\frac{\beta_t}{2} \mathbb{E}\left[\left\|\breve{\mathbf{x}}_{t+1}-\mathbf{x}_t\right\|_{\mathcal{M}}^2\right]
		-\frac{\gamma_t \beta_t - \Lambda \beta_t^2}{2} \mathbb{E}\left[\left\|\mathbf{s}_t\right\|^2\right] +\frac{r_t}{2}  \mathbb{E}\left[ \left\|\ \mathbf{x}_{t+1} - \mathbf{x}_{t}  \right\|^2 \right]\nonumber\\
		&+ \frac{1}{2r_t} \left\{ \alpha^{2t} \mathbb{E}\left[\left\|\tilde{v}_0-\nabla h\left(\widehat{\mathbf{x}}_0\right)\right\|^2\right]+\Lambda^2 \sum_{i=0}^{t-1} \alpha^{2(t-i)} \mathbb{E}\left[\left\|\widehat{\mathbf{x}}_{i+1}-\widehat{\mathbf{x}}_i\right\|^2\right]+\frac{1-\alpha}{1+\alpha}\sigma^2\right\}. \label{exp phi 3}
	\end{align}
	For the simplification of symbols, we define
	\begin{equation}\label{def of Hi}
		\mathbf{H}_i = \mathbb{E}\left[\Phi\left(\mathbf{x}_{i}\right)\right]+\frac{\eta_i}{2}\mathbb{E}\left[\left\|\mathbf{x}_i-\mathbf{x}_{i-1}\right\|^2\right].
	\end{equation}
	Summing up (\ref{exp phi 3}) from $t=0$ to $t=m$, we obtain 
	\begin{align}
		\mathbf{H}_{m+1} \leq & \mathbf{H}_{0} -  \sum_{t=0}^{m}\frac{\beta_t}{2}
		\mathbb{E}\left[\left\|\breve{\mathbf{x}}_{t+1}-\mathbf{x}_t\right\|_{\mathcal{M}}^{2}\right]- \sum_{t=0}^{m} \frac{\gamma_t \beta_t - \Lambda \beta_t^2}{2} \mathbb{E}\left[\left\|\mathbf{s}_t\right\|^2\right] + (\sum_{t=0}^{m} \frac{\alpha^{2t}}{2r_t})\mathbb{E}\left[\left\|\tilde{v}_0-\nabla
		h\left(\widehat{\mathbf{x}}_0\right)\right\|^2\right]\nonumber\\
		&+\sum_{t=0}^{m} \frac{1-\alpha}{2r_t(1+\alpha)} \sigma^2 - \sum_{t=0}^{m}\frac{\gamma_{t+1} / \beta_{t+1}-\eta_{t+1} -r_t}{2}\mathbb{E}\left[\left\|\mathbf{x}_{t+1}-\mathbf{x}_t\right\|^2 \right] \nonumber\\
		& + \sum_{t=0}^{m} \frac{\Lambda^2}{2r_t} \sum_{i=0}^{t-1} \alpha^{2(t-i)} \mathbb{E}\left[\left\|\widehat{\mathbf{x}}_{i+1}-\widehat{\mathbf{x}}_i\right\|^2\right]. \label{potential H inequ}
	\end{align}
	

	We have from the above equations (\ref{x-rela 1}) and (\ref{x-rela 2}) that
	\begin{align}
		\widehat{\mathbf{x}}_{t+1} -\widehat{\mathbf{x}}_t & = \widehat{\mathbf{x}}_{t+1} -\mathbf{x}_{t+1} + \mathbf{x}_{t+1} -\mathbf{x}_{t} +\mathbf{x}_{t} -\widehat{\mathbf{x}}_t \nonumber\\
		&=	\theta_{t+1}\left(\mathbf{x}_{t+1}-\mathbf{x}_{t}\right) + \left(\mathbf{x}_{t+1}-\mathbf{x}_{t}\right) -
		\theta_{t}\left(\mathbf{x}_{t}-\mathbf{x}_{t-1}\right)\nonumber\\
		& = (\theta_{t+1}+1)\left(\mathbf{x}_{t+1}-\mathbf{x}_{t}\right)  -
		\theta_{t}\left(\mathbf{x}_{t}-\mathbf{x}_{t-1}\right),\nonumber\\
		\mathbb{E}\left[\left\|\widehat{\mathbf{x}}_{i+1}-\widehat{\mathbf{x}}_i\right\|^2\right] &\leq 2(\theta_{i+1}+1)^2 \mathbb{E}\left[ \left\|\ \mathbf{x}_{i+1}-\mathbf{x}_{i} \right\|^2\right] + 2\theta_{i}^2\mathbb{E}\left[ \left\|\ \mathbf{x}_{i}-\mathbf{x}_{i-1} \right\|^2\right],
	\end{align}
	where the inequality is due to the Cauchy-Schwartz inequality.
	
	Denote $r_t = r, q_i = \mathbb{E}\left[ \left\|\ \mathbf{x}_{i+1}-\mathbf{x}_{i} \right\|^2\right]$, and
	$l_3 = \max_{0\leq i \leq m} \left\{ 2(\theta_{i+1}+1)^2, 2\theta_{i}^2\right\}<8$,
	we first upper bound the term $\sum_{t=0}^{m} \frac{\Lambda^2}{2r_t} \sum_{i=0}^{t-1} \alpha^{2(t-i)} \mathbb{E}\left[\left\|\widehat{\mathbf{x}}_{i+1}-\widehat{\mathbf{x}}_i\right\|^2\right]$ in  (\ref{exp phi 3})
	\begin{align}
		&\sum_{t=0}^{m} \frac{\Lambda^2}{2r_t} \sum_{i=0}^{t-1} \alpha^{2(t-i)} \mathbb{E}\left[\left\|\widehat{\mathbf{x}}_{i+1}-\widehat{\mathbf{x}}_i\right\|^2\right]\nonumber\\
		\leq & \frac{l_3 \Lambda^2}{2 r}\sum_{t=0}^{m}\sum_{i=0}^{t-1} \alpha^{2(t-i)} (q_i + q_{i-1})\nonumber\\
		= & \frac{l_3 \Lambda^2}{2 r} \left\{
		\sum_{i=0}^{0} \alpha^{2(1-i)} (q_i + q_{i-1}) +
		\sum_{i=0}^{1} \alpha^{2(2-i)} (q_i + q_{i-1})+ \cdots + 
		\sum_{i=0}^{m-1} \alpha^{2(m-i)} (q_i + q_{i-1})
		\right\} \nonumber\\
		= &\frac{l_3 \Lambda^2}{2 r} \Big\{ \alpha^{2}q_{0} + 
		\left[ \alpha^{4}q_{0} +\alpha^{2}(q_{1}+q_{0}) \right] +  \left[ \alpha^{6}q_{0} +\alpha^{4}(q_{1}+q_{0})+\alpha^{2}(q_{2}+q_{1}) \right] 
		+\cdots \nonumber\\
		&+ \left[ \alpha^{2m}q_{0} +\alpha^{2(m-1)}(q_{1}+q_{0}) +
		\alpha^{2(m-2)}(q_{2}+q_{1})
		+\alpha^{2}(q_{m-1}+q_{m-2}) \right]
		\Big\}\nonumber\\
		= & \frac{l_3 \Lambda^2}{2 r} \Big\{ (\alpha^{2} + \alpha^{4} + \cdots +\alpha^{2m}) q_0 + (\alpha^{2} + \alpha^{4} + \cdots +\alpha^{2(m-1)})(q_{1}+q_{0}) + \cdots \nonumber\\
		&+ (\alpha^{2} + \alpha^{4})(q_{m-2}+q_{m-3}) + \alpha^{2}(q_{m-1}+q_{m-2})
		\Big\}\nonumber\\
		=&  \frac{l_3 \Lambda^2 \alpha^2}{2 r (1-\alpha^2)}\Big\{
		(2-\alpha^{2m}-\alpha^{2(m-1)})q_{0} + (2-\alpha^{2(m-1)}-\alpha^{2(m-2)})q_{1}+ \cdots \nonumber\\
		&+ (2-\alpha^{4}-\alpha^{2})q_{m-2}+(1-\alpha^2)q_{m-1}
		\Big\}. \label{sum inequ}
	\end{align}
	
	Now,  using these definitions and inequality (\ref{sum inequ}) to deal with  the term 
	\begin{equation}
		\mathcal{T}_m:=	- \sum_{t=0}^{m}\frac{\gamma_{t+1} / \beta_{t+1}-\eta_{t+1} -r}{2}\mathbb{E}\left[\left\|\mathbf{x}_{t+1}-\mathbf{x}_t\right\|^2 \right] + \sum_{t=0}^{m} \frac{\Lambda^2}{2r_t} \sum_{i=0}^{t-1} \alpha^{2(t-i)} \mathbb{E}\left[\left\|\widehat{\mathbf{x}}_{i+1}-\widehat{\mathbf{x}}_i\right\|^2\right]\nonumber
	\end{equation}
	in (\ref{potential H inequ}), we get 
	\begin{align}
		\mathcal{T}_m	\leq & 
		\left[ \frac{l_3 \Lambda^2 \alpha^2}{2 r (1-\alpha^2)}
		(2-\alpha^{2m}-\alpha^{2(m-1)}) - \frac{\gamma_{1} / \beta_{1}-\eta_{1} -r}{2}\right]q_{0} + \cdots\nonumber\\
		&+ \left[ \frac{l_3 \Lambda^2 \alpha^2}{2 r } - \frac{\gamma_{m} / \beta_{m}-\eta_{m} -r}{2} \right]q_{m-1} - \frac{\gamma_{m+1} / \beta_{m+1}-\eta_{m+1} -r}{2} q_{m}
	\end{align}
	To guarantee $\mathcal{T}_m\leq 0$, we can choose 
	\begin{equation}
		\begin{cases}\frac{l_3 \Lambda^2 \alpha^2}{2 r (1-\alpha^2)}(2-\alpha^{2m}-\alpha^{2(m-1)}) - \frac{\gamma_{1} / \beta_{1}-\eta_{1} -r}{2} & \leq 0 \\ 
			\frac{l_3 \Lambda^2 \alpha^2}{2 r (1-\alpha^2)}(2-\alpha^{2(m-1)}-\alpha^{2(m-2)}) - \frac{\gamma_{2} / \beta_{2}-\eta_{2} -r}{2}
			& \leq 0 \\
			\cdots 
			& \cdots \\ 
			\frac{l_3 \Lambda^2 \alpha^2}{2 r } - \frac{\gamma_{m} / \beta_{m}-\eta_{m} -r}{2} & \leq 0 \\
			0- \frac{\gamma_{m+1} / \beta_{m+1}-\eta_{m+1} -r}{2} & \leq 0 .\end{cases}\label{inequ of alpha}
	\end{equation}
	
	We can show that (\ref{inequ of alpha}) holds, if we have
	\begin{align}
		\frac{l_3 \Lambda^2 \alpha^2}{2 r (1-\alpha^2)}(2-\alpha^{2m}-\alpha^{2(m-1)}) - \frac{\gamma_{i} / \beta_{i}-\eta_{i} -r}{2} \leq 0,\forall 0\leq i \leq m,
	\end{align}
	which is satisfied if the following inequality holds
	\begin{align}
		\frac{l_3 \Lambda^2 \alpha^2}{ r (1-\alpha^2)}-g + \frac{r}{2}\leq 0,\label{g inequ 1}
	\end{align}
	where $g=\min_{0\leq i \leq m} \left\{ \frac{\gamma_{i} / \beta_{i}-\eta_{i}}{2} \right\}$.
	
	We choose $r=\sqrt{\frac{2l_3 \Lambda^2 \alpha^2}{  (1-\alpha^2)}}$, then (\ref{g inequ 1}) turns to 
	\begin{align}
		g(\alpha):=\alpha \Lambda \sqrt{\frac{2l_3 }{  (1-\alpha^2)}} = r
		\leq g,\label{g inequ 2}
	\end{align}
	to satisfy (\ref{g inequ 2}), we choose $\gamma_t$ as 
	\begin{align}
		\gamma_t / \beta_t-\eta_t & =\gamma_t / \beta_t-\left(\gamma_t / \beta_t+\mu\right) \theta_t^2\nonumber\\
		&=\gamma_t / \beta_t\left(1-\theta_t^2\right)-\mu \theta_t^2 \nonumber\\
		&\geq 2 g=\min_{0\leq i \leq m} \left\{ \gamma_{i} / \beta_{i}-\eta_{i} \right\}\nonumber\\
		&\geq 2g(\alpha) \qquad \qquad \qquad \forall t,
		\label{gammat inequ 1}
	\end{align}
	then the last inequality $\frac{\gamma_{m+1} / \beta_{m+1}-\eta_{m+1} -r}{2}\geq 0$ in (\ref{inequ of alpha}) satisfies as $2g(\alpha)\geq r$.

	Using $\beta_t / \beta_{t-1} \leq 1$ and $\theta_{t}\leq 1- \beta_{t-1}\leq 1-\tau$, $2\tau -\tau^2\leq 1-\theta_{t}^2$, we choose $\gamma_{t}$ as follows to ensure (\ref{gammat inequ 1}) hold:
	\begin{equation}
		\gamma_{t} = \beta_{t}(\frac{\mu + 2g(\alpha)}{2\tau -\tau^2} -\mu) \frac{t+1}{t}.\label{choice of gammat} 
	\end{equation}
	
	Hence, it follows from (\ref{potential H inequ}) and the property of $\mathcal{T}_m \leq 0$ that
	\begin{align}
		\mathbf{H}_{m+1} \leq & \mathbf{H}_{0} -  \sum_{t=0}^{m}\frac{\beta_t}{2}\mathbb{E}\left[ \left\|\breve{\mathbf{x}}_{t+1}-\mathbf{x}_t\right\|_{\mathcal{M}}^2\right] 
		- \sum_{t=0}^{m} \frac{\gamma_t \beta_t - \Lambda \beta_t^2}{2}
		\mathbb{E}\left[ \left\|\mathbf{s}_t\right\|^2\right] + (\sum_{t=0}^{m} \frac{\alpha^{2t}}{2r})\mathbb{E}\left[\left\|\tilde{v}_0-\nabla h\left(\widehat{\mathbf{x}}_0\right)\right\|^2\right]\nonumber\\
		&+\sum_{t=0}^{m} \frac{1-\alpha}{2r(1+\alpha)} {{\sigma^2}} \nonumber\\
		& \mathop{\leq}^{(i)}
		\mathbf{H}_{0} -  \sum_{t=0}^{m}\frac{\beta_t}{2}
		\mathbb{E}\left[ \left\|\breve{\mathbf{x}}_{t+1}-\mathbf{x}_t\right\|_{\mathcal{M}}^2\right]
		- \sum_{t=0}^{m} \frac{\gamma_t \beta_t - \Lambda \beta_t^2}{2}
		\mathbb{E}\left[ \left\|\mathbf{s}_t\right\|^2 \right]
		+ 
		\frac{(1-\alpha^{2m+2})\sigma^2}{2r (1-\alpha^2)M}
		\nonumber\\
		&\quad  + \frac{(m+1)(1-\alpha){{\sigma^2}}}{2r(1+\alpha)}\nonumber\\
		&\leq
		\mathbf{H}_{0} -  \sum_{t=0}^{m}\frac{\beta_t}{2}
		\mathbb{E}\left[ \left\|\breve{\mathbf{x}}_{t+1}-\mathbf{x}_t\right\|_{\mathcal{M}}^2\right]- \sum_{t=0}^{m} \frac{\gamma_t \beta_t - \Lambda \beta_t^2}{2} \mathbb{E}\left[\left\|\mathbf{s}_t\right\|^2\right] + 
		\frac{\sigma^2}{2r (1-\alpha^2)M} \nonumber\\
		&\quad + \frac{(m+1)(1-\alpha){{\sigma^2}}}{2r(1+\alpha)}
		, \label{potential H inequ 2}
	\end{align}
	where the inequality (i) is due to $\mathbb{E}\left[\left\|\tilde{v}_0-\nabla h\left(\widehat{\mathbf{x}}_0\right)\right\|^2\right]\leq \frac{\sigma^2}{M}$.

	Dividing both side of the equation (\ref{potential H inequ 2}) by $m+1$, and we get
	\begin{align}
		&\frac{1}{m+1}\left[ \sum_{t=0}^{m}\frac{\beta_t}{2} \mathbb{E}\left[\left\|\breve{\mathbf{x}}_{t+1}-\mathbf{x}_t\right\|_{\mathcal{M}}^2\right]+ \sum_{t=0}^{m} \frac{\gamma_t \beta_t - \Lambda \beta_t^2}{2} \mathbb{E}\left[\left\|\mathbf{s}_t\right\|^2\right] \right]\nonumber\\
		\leq& \frac{\mathbf{H}_{0}-\mathbf{H}_{m+1}}{m+1}  + 
		\frac{\sigma^2}{2r (1-\alpha^2)M(m+1)} + \frac{(1-\alpha)\sigma^2}{2r(1+\alpha)}\nonumber\\
		=& \frac{\mathbf{H}_{0}-\mathbf{H}_{m+1}}{m+1} + \frac{\sigma^2}{2r(1+\alpha)}
		\left[ \frac{1}{M(m+1)(1-\alpha)} + (1-\alpha) \right] \label{divide H}
	\end{align}
	
	With the choice of $\alpha:=1-\frac{c_1}{\sqrt{M(m+1)}}$ for some $0<c_1<\sqrt{M(m+1)}$, and from the inequality $1-\alpha^2 > \frac{c_1}{\sqrt{M(m+1)}}$,  we have $\gamma_t \leq \left[\frac{\mu +2\sqrt{ \frac{2l_3}{c_1}} \Lambda \left[ M(m+1)   \right]^{\frac{1}{4} }}{\tau(2-\tau)}   \right]\beta_t \frac{t+1}{t}$ for $\gamma_t$ defined in  (\ref{choice of gammat}). Then the last two terms of the right-hand side of (\ref{divide H}) become
	\begin{align}\label{1-alpha}
		\frac{1}{(1-\alpha) M(m+1)}+(1-\alpha)=\left(c_1+\frac{1}{c_1}\right) \frac{1}{\sqrt{M(m+1)}}.
	\end{align}
	With this setting of $\alpha =1-\frac{c_1}{\sqrt{M(m+1)}}$, (\ref{divide H}) leads to
	\begin{align}
		&\frac{1}{m+1}\left[ \sum_{t=0}^{m}\frac{\beta_t}{2} \mathbb{E}\left[\left\|\breve{\mathbf{x}}_{t+1}-\mathbf{x}_t\right\|_{\mathcal{M}}^2\right]
		+ \sum_{t=0}^{m} \frac{\gamma_t \beta_t - \Lambda \beta_t^2}{2} \mathbb{E}\left[\left\|\mathbf{s}_t\right\|^2\right] \right]\nonumber\\
		\leq &	\frac{\mathbf{H}_{0}-\mathbf{H}_{m+1}}{m+1} + \frac{\sigma^2}{2r(1+\alpha)}(c_1+\frac{1}{c_1})\frac{1}{\sqrt{M(m+1)}}. \label{final potential H 1}
	\end{align}
	
	Taking $ M=\sigma^{k_1}(m+1)^{k_2}$, $k_1>0, k_2>0$, 
	$\alpha \geq \underline{\alpha}$, $r\geq \sqrt{\frac{2l_3 \Lambda^2 \underline{\alpha}^2}{  (1-\underline{\alpha}^2)}}:= \underline{r}$, then from 
	(\ref{final potential H 1}), it further gives that
	\begin{align}
		&\frac{1}{m+1}\left[ \sum_{t=0}^{m}\frac{\beta_t}{2}
		\mathbb{E}\left[ \left\|\breve{\mathbf{x}}_{t+1}-\mathbf{x}_t\right\|_{\mathcal{M}}^2\right]
		- \sum_{t=0}^{m} \frac{\gamma_t \beta_t - \Lambda \beta_t^2}{2}
		\mathbb{E}\left[ \left\|\mathbf{s}_t\right\|^2\right]
		\right]\nonumber\\
		\leq &	\frac{\mathbf{H}_{0}-\mathbf{H}_{m+1}}{m+1} + \frac{\sigma^{2}}{2\underline{r}(1+\underline{\alpha})}(c_1+\frac{1}{c_1})\frac{1}{M(m+1)},\label{final potential H 2}
	\end{align}
	and
	\begin{align}
		&\frac{1}{m+1}\left[ \sum_{t=0}^{m}\frac{\beta_t}{2} \mathbb{E}\left[\left\|\breve{\mathbf{x}}_{t+1}-\mathbf{x}_t\right\|_{\mathcal{M}}^2\right]+ \sum_{t=0}^{m} \frac{\gamma_t \beta_t - \Lambda \beta_t^2}{2} \mathbb{E}\left[\left\|\mathbf{s}_t\right\|^2\right] \right]\nonumber\\
		\leq &	\frac{\mathbf{H}_{0}-\mathbf{H}_{m+1}}{m+1} + \frac{\sigma^{2-\frac{k_1}{2}}}{2\underline{r}(1+\underline{\alpha})}(c_1+\frac{1}{c_1})\frac{1}{(m+1)^{\frac{1+k_2}{2}}}. \nonumber
	\end{align}
	
\end{proof}

\subsection{Proof of \cref{cor 5.5}}\label{app:cor 5.5}

\textbf{Corollary 5.5}
Suppose that Assumption \ref{assum of func h} and conditions in \cref{x y w grad 0} hold. Let $\left\{\left(\widehat{\mathbf{x}}_t,\breve{\mathbf{x}}_t,\mathbf{x}_t \right) \right\}$ be generated by Algorithm \ref{alg 5.1} using the following weight $\alpha$ and $\gamma_t$:
\begin{equation}
	\begin{cases}
		&\alpha=1-\frac{c_1}{\sqrt{M(m+1)}},\\
		&\gamma_{t} = \beta_{t}(\frac{\mu + 2g(\alpha)}{2\tau -\tau^2} -\mu) \frac{t+1}{t},
	\end{cases}
\end{equation}
where $g(\alpha)=\alpha \Lambda \sqrt{\frac{2l_3 }{  (1-\alpha^2)}}$.

Then, we have
\begin{align}
	\lim _{m \rightarrow \infty}	\mathbb{E} \left\|\ \nabla \Phi\left(\widehat{\mathbf{x}}_{\bar{m}}\right) \right\|^2
	=	
	\lim _{m \rightarrow \infty} \frac{1}{m+1}\sum_{t=0}^{m}\mathbb{E}\left[ \left\|\ \nabla \Phi\left(\widehat{\mathbf{x}}_t\right) \right\|^2 \right]
	=0,\nonumber
\end{align}
and \begin{equation}
	\lim _{m \rightarrow \infty}\mathbb{E}\left[ \left\|\breve{\mathbf{x}}_{\bar{m}}-\mathbf{x}_{\bar{m}}\right\|_{\mathcal{M}}^2 \right]=0 \quad \text { and } \quad \lim _{m \rightarrow \infty} \mathbb{E}
	\left\|s_{\bar{m}}\right\|^2= \lim _{m \rightarrow \infty}\mathbb{E}
	\left\|\breve{\mathbf{x}}_{\bar{m}+1}-\breve{\mathbf{x}}_{\bar{m}}\right\|^2=0.\nonumber
\end{equation}

\begin{proof}
	First,with the denotation of $c_2 = \min_{0\leq t \leq m} \left\{ \frac{\beta_t}{2}, \frac{\gamma_t \beta_t - \Lambda \beta_t^2}{2}\right\}$, the output $\left(\widehat{\mathbf{x}}_{\bar{m}},\breve{\mathbf{x}}_{\bar{m}},\mathbf{x}_{\bar{m}} \right) $ of Algorithm \ref{alg 5.1} and from 
	(\ref{final potential H 2}), we have 
	\begin{align}
		c_2 \mathbb{E}\left[ \left\|\breve{\mathbf{x}}_{\bar{m}}-\mathbf{x}_{\bar{m}}\right\|_{\mathcal{M}}^2 \right]
		+ c_2 \mathbb{E}\left[ \left\|\mathbf{s}_{\bar{m}}\right\|^2   \right]\leq \frac{\mathbf{H}_{0}-\mathbf{H}_{m+1}}{m+1} + \frac{\sigma^{2-\frac{k_1}{2}}}{2\underline{r}(1+\underline{\alpha})}(c_1+\frac{1}{c_1})\frac{1}{(m+1)^{\frac{1+k_2}{2}}}. \label{iter inequ 1}
	\end{align}
	
	We can obtain
	from (\ref{iter inequ 1}) that
	\begin{equation}
		\lim _{m \rightarrow \infty}\mathbb{E}\left[ \left\|\breve{\mathbf{x}}_{\bar{m}}-\mathbf{x}_{\bar{m}}\right\|_{\mathcal{M}}^2 \right]=0, \quad \text { and } \quad \lim _{m \rightarrow \infty}\mathbb{E} 
		\left\|s_{\bar{m}}\right\|^2=
		\lim _{m \rightarrow \infty}\mathbb{E} \left\|\breve{\mathbf{x}}_{\bar{m}+1}-\breve{\mathbf{x}}_{\bar{m}}\right\|^2=0. \label{equ s-barm}
	\end{equation}
	
	Then, from (\ref{x-opti condi}), we have 
	\begin{align}
		\tilde{v}_t-\nabla h\left(\widehat{\mathbf{x}}_t\right) + \nabla h\left(\widehat{\mathbf{x}}_t\right)
		+\gamma_t \mathbf{s}_t+
		\nabla \phi\left(\breve{\mathbf{x}}_{t+1}\right)-
		\nabla\phi\left(\widehat{\mathbf{x}}_t \right) + \nabla\phi\left(\widehat{\mathbf{x}}_t \right)
		=0,
	\end{align}
	and from the definition of $\Phi(\mathbf{x})=h(\mathbf{x})+\phi(\mathbf{x})$, we obtain 
	\begin{align}
		\nabla \Phi\left(\widehat{\mathbf{x}}_t\right) = \nabla h\left(\widehat{\mathbf{x}}_t\right) -\tilde{v}_t - \gamma_t \mathbf{s}_t + \nabla\phi\left(\widehat{\mathbf{x}}_t \right) - \nabla \phi\left(\breve{\mathbf{x}}_{t+1}\right). \label{grad of Psi 1} 
	\end{align}
	
	Using the result of (\ref{grad of Psi 1}), it gives that
	\begin{align}
		\mathbb{E}\left[ \left\|\ \nabla \Phi\left(\widehat{\mathbf{x}}_t\right) \right\|^2 \right]\leq 3  \mathbb{E}\left[ \left\|\ \nabla h\left(\widehat{\mathbf{x}}_t\right) -\tilde{v}_t \right\|^2 \right] +
		3 \gamma_t^2  \mathbb{E}\left[ \left\|\  \mathbf{s}_t  \right\|^2 \right] +
		3    \mathbb{E}\left[ \left\| \nabla\phi\left(\widehat{\mathbf{x}}_t \right) - \nabla \phi\left(\breve{\mathbf{x}}_{t+1}\right)  \right\|^2 \right]. \label{grad of Psi 2}  
	\end{align}
	
	Summing up (\ref{grad of Psi 2}) from $t=0$ to $t=m$ and dividing $m+1$, we get 
	\begin{align}
		\frac{1}{m+1}\sum_{t=0}^{m}\mathbb{E}\left[ \left\|\ \nabla \Phi\left(\widehat{\mathbf{x}}_t\right) \right\|^2 \right]\leq& \frac{3}{m+1}\sum_{t=0}^{m}  \mathbb{E}\left[ \left\|\ \nabla h\left(\widehat{\mathbf{x}}_t\right) -\tilde{v}_t \right\|^2 \right] +
		\frac{3 }{m+1}\sum_{t=0}^{m} \mathbb{E}\left[\gamma_t^2 \left\|\  \mathbf{s}_t  \right\|^2 \right]\nonumber\\
		& +
		\frac{3}{m+1}\sum_{t=0}^{m}    \mathbb{E}\left[ \left\| \nabla\phi\left(\widehat{\mathbf{x}}_t \right) - \nabla \phi\left(\breve{\mathbf{x}}_{t+1}\right)  \right\|^2 \right] \label{grad of phi 3}
	\end{align}
	
	We first analyze the middle term $\frac{3 }{m+1}\sum_{t=0}^{m} \mathbb{E}\left[\gamma_t^2 \left\|\  \mathbf{s}_t  \right\|^2 \right]$ in (\ref{grad of phi 3}). From (\ref{final potential H 2}), we drive 
	\begin{align}
		\frac{1 }{m+1}\sum_{t=0}^{m} \frac{\gamma_t \beta_t - \Lambda \beta_t^2}{2} \mathbb{E}\left[ \left\|\  \mathbf{s}_t  \right\|^2 \right] 
		\leq \frac{\mathbf{H}_{0}-\mathbf{H}_{m+1}}{m+1} + \frac{\sigma^2}{2r(1+\alpha)}(c_1+\frac{1}{c_1})\frac{1}{\sqrt{M(m+1)}},\nonumber 
	\end{align}
	which is equivalent to 
	\begin{align}
		\frac{1 }{2(m+1)}\sum_{t=0}^{m} \frac{\gamma_t \beta_t - \Lambda \beta_t^2}{\gamma_t^2}\gamma_t^2 \mathbb{E}\left[ \left\|\  \mathbf{s}_t  \right\|^2 \right] 
		\leq \frac{\mathbf{H}_{0}-\mathbf{H}_{m+1}}{m+1} + \frac{\sigma^2}{2r(1+\alpha)}(c_1+\frac{1}{c_1})\frac{1}{\sqrt{M(m+1)}}.\label{st}
	\end{align}
	Denoting $\tilde{\delta} = \min_{0\leq t \leq m} \left\{ \frac{\gamma_t \beta_t - \Lambda \beta_t^2}{\gamma_t^2}
	\right\}=\min_{0\leq t \leq m} \left\{ \frac{ \beta_t }{\gamma_t}-\Lambda(\frac{\beta_t }{\gamma_t})^2
	\right\}$, and from the setting of $\gamma_t$ in (\ref{choice of gammat}), we can obtain
	\begin{align}
		2g(\alpha) \leq\frac{\gamma_t}{\beta_t}\leq 
		\frac{2\mu+4g(\alpha)}{2\tau - \tau^2},
	\end{align}
	which further implies $\frac{\gamma_t}{\beta_t}=\mathcal{O}(g(\alpha))$, $\frac{\beta_t}{\gamma_t}=\mathcal{O}(1/g(\alpha))$,  and $\tilde{\delta}=\mathcal{O}(1/g(\alpha))$.
	
	Dividing $\tilde{\delta}$ in the both sides of (\ref{st}) to drive
	\begin{align}
		&\frac{1 }{2(m+1)}\sum_{t=0}^{m} \gamma_t^2 \mathbb{E}\left[ \left\|\  \mathbf{s}_t  \right\|^2 \right]\nonumber\\ 
		\leq &\frac{\mathbf{H}_{0}-\mathbf{H}_{m+1}}{\tilde{\delta}(m+1)} + \frac{\sigma^2}{2r\tilde{\delta}(1+\alpha)}(c_1+\frac{1}{c_1})\frac{1}{\sqrt{M(m+1)}}\nonumber\\
		\mathop{\leq}^{(i)}
		&\frac{1}{l_4}\Biggl\{\frac{g(\alpha)(\mathbf{H}_{0}-\mathbf{H}_{m+1})}{m+1} + \frac{\sigma^2g(\alpha)}{2r(1+\alpha)}(c_1+\frac{1}{c_1})\frac{1}{\sqrt{M(m+1)}}
		\Biggr\}\nonumber\\
		\mathop{\leq}^{(ii)}
		&\frac{ \Lambda\sqrt{2l_3}}{l_4}
		\Biggl\{\frac{M^{\frac{1}{4}}(\mathbf{H}_{0}-\mathbf{H}_{m+1})}{\sqrt{c_1}(m+1)^{\frac{3}{4}}} + \frac{\sigma^2(c_1+\frac{1}{c_1})}{2r(1+\alpha)\sqrt{c_1}}\frac{1}{\left[M(m+1)\right]^{\frac{1}{4}}}
		\Biggr\},
	\end{align}
	where the first inequality (i) is due to $\tilde{\delta}=\mathcal{O}(1/g(\alpha))= \frac{l_4}{g(\alpha)}$, $l_4>0$, and the last inequality (ii) is due to $g(\alpha):=\alpha \Lambda \sqrt{\frac{2l_3 }{  (1-\alpha^2)}}$, $\alpha:=1-\frac{c_1}{\sqrt{M(m+1)}}$, and $1-\alpha^2 \geq  \frac{c_1}{\sqrt{M(m+1)}}$ in (\ref{equ of alpha}).
	
	It follows from choosing $M=\mathcal{O}((m+1)^{k_2})$, $k_2\in (0,3)$, and letting $m \rightarrow \infty$ that
	\begin{align}\label{st 2.47}
		\lim _{m \rightarrow \infty}
		\frac{3 }{m+1}\sum_{t=0}^{m} \gamma_t^2\mathbb{E}\left[ \left\|\  \mathbf{s}_t  \right\|^2 \right]=0.
	\end{align}

	Then we focus on
	\begin{align}
		&\lim _{m \rightarrow \infty} \left[\frac{3}{m+1}\sum_{t=0}^{m}    \mathbb{E}\left[ \left\| \nabla\phi\left(\widehat{\mathbf{x}}_t \right) - \nabla \phi\left(\breve{\mathbf{x}}_{t+1}\right)  \right\|^2 \right]\right]\nonumber\\
		\leq & \lim _{m \rightarrow \infty} \left[\frac{3L_{\phi}^2}{m+1}\sum_{t=0}^{m}    \mathbb{E}\left[ \left\| \widehat{\mathbf{x}}_t - \breve{\mathbf{x}}_{t+1}  \right\|^2 \right]\right]\nonumber\\
		= &	\lim _{m \rightarrow \infty} \left[\frac{3L_{\phi}^2}{m+1}\sum_{t=0}^{m}    \mathbb{E}\left[ \left\| (1-\beta_{t})({\mathbf{x}}_t - \breve{\mathbf{x}}_{t+1} ) +\beta_{t}(\breve{\mathbf{x}}_{t}-\breve{\mathbf{x}}_{t+1}) \right\|^2 \right]\right]\nonumber\\
		\mathop{\leq}^{(i)}& \lim _{m \rightarrow \infty} \left[\frac{3l_5 L_{\phi}^2}{m+1}\sum_{t=0}^{m}\left[    
		\mathbb{E}\left[ \left\| {\mathbf{x}}_t - \breve{\mathbf{x}}_{t+1}  \right\|^2 \right] 
		+
		\mathbb{E}\left[\left\|\breve{\mathbf{x}}_{t}-\breve{\mathbf{x}}_{t+1}\right\|^2  \right]
		\right] \right]\nonumber\\
		\mathop{=}^{(ii)} &0, \label{exp phi}
	\end{align}
	where the first equation is from $\widehat{\mathbf{x}}_t - \breve{\mathbf{x}}_{t+1} = (1-\beta_{t})({\mathbf{x}}_t - \breve{\mathbf{x}}_{t+1} ) +\beta_{t}(\breve{\mathbf{x}}_{t}-\breve{\mathbf{x}}_{t+1}) $,
	$l_5 = \max_{0\leq i \leq m} \left\{ 2(1-\beta_{i})^2, 2\beta_{i}^2\right\}$, the inequality (i) is due to $\left\| \nabla\phi\left(\widehat{\mathbf{x}}_t \right) - \nabla \phi\left(\breve{\mathbf{x}}_{t+1}\right)  \right\| \leq L_{\phi}
	\left\| \widehat{\mathbf{x}}_t  - \breve{\mathbf{x}}_{t+1} \right\|$, and the equation (ii) is due to (\ref{final potential H 2}).
	
	We end with dealing the first term $\frac{3}{m+1}\sum_{t=0}^{m}  \mathbb{E}\left[ \left\|\ \nabla h\left(\widehat{\mathbf{x}}_t\right) -\tilde{v}_t \right\|^2 \right]$ in (\ref{grad of phi 3}). Similarly, using the notations $q_i = \mathbb{E}\left[ \left\|\ \mathbf{x}_{i+1}-\mathbf{x}_{i} \right\|^2\right]$ and $l_3 = \max_{0\leq i \leq m} \left\{ 2(\theta_{i+1}+1)^2, 2\theta_{i}^2\right\}$, we give an upper bound of  $\frac{3}{m+1}\sum_{t=0}^{m}  \mathbb{E}\left[ \left\|\ \nabla h\left(\widehat{\mathbf{x}}_t\right) -\tilde{v}_t \right\|^2 \right]$ as
	\begin{align}
		&\frac{3}{m+1}\sum_{t=0}^{m}  \mathbb{E}\left[ \left\|\ \nabla h\left(\widehat{\mathbf{x}}_t\right) -\tilde{v}_t \right\|^2 \right]\nonumber\\
		\leq & \frac{3}{m+1}\left\{\sum_{t=0}^{m} \alpha^{2t} \mathbb{E}\left[\left\|\tilde{v}_0-\nabla h\left(\widehat{\mathbf{x}}_0\right)\right\|^2\right]+l_3\Lambda^2 \sum_{t=0}^{m}\sum_{i=0}^{t-1} \alpha^{2(t-i)}(q_i + q_{i-1}) \right\} + \frac{3(1-\alpha)}{1+\alpha}\sigma^2\nonumber\\
		\leq & \frac{3}{(m+1)(1-\alpha^2)} \mathbb{E}\left[\left\|\tilde{v}_0-\nabla h\left(\widehat{\mathbf{x}}_0\right)\right\|^2\right]+
		\frac{3l_3\Lambda^2}{m+1} \sum_{t=0}^{m}\sum_{i=0}^{t-1} \alpha^{2(t-i)}(q_i + q_{i-1}) + \frac{3(1-\alpha)}{1+\alpha}\sigma^2\nonumber\\
		= &\frac{3}{(m+1)(1-\alpha^2)} \mathbb{E}\left[\left\|\tilde{v}_0-\nabla h\left(\widehat{\mathbf{x}}_0\right)\right\|^2\right]+
		\frac{3l_3\Lambda^2}{m+1} 
		\big\{ \alpha^{2}q_{0} + 
		\left[ \alpha^{4}q_{0} +\alpha^{2}(q_{1}+q_{0}) \right] +   \nonumber\\
		&  \left[ \alpha^{6}q_{0} +\alpha^{4}(q_{1}+q_{0})+\alpha^{2}(q_{2}+q_{1}) \right] 
		+\cdots+ \left[ \alpha^{2m}q_{0} + \alpha^{2}(q_{m-1}+q_{m-2}) \right] \big\}  + \frac{3(1-\alpha)}{1+\alpha}\sigma^2 \nonumber\\
		= &\frac{3}{(m+1)(1-\alpha^2)} \mathbb{E}\left[\left\|\tilde{v}_0-\nabla h\left(\widehat{\mathbf{x}}_0\right)\right\|^2\right]+
		\frac{3l_3\Lambda^2 \alpha^2}{(m+1)(1-\alpha^2)} 
		\big[  
		(2-\alpha^{2m}-\alpha^{2(m-1)})q_{0} +  \nonumber\\
		& (2-\alpha^{2(m-1)}-\alpha^{2(m-2)})q_{1}+ \cdots + (2-\alpha^{4}-\alpha^{2})q_{m-2}+(1-\alpha^2)q_{m-1}
		\big]
		+ \frac{3(1-\alpha)}{1+\alpha}\sigma^2\nonumber\\
		\mathop{\leq}^{(i)} &  \frac{3}{(m+1)(1-\alpha^2)} \mathbb{E}\left[\left\|\tilde{v}_0-\nabla h\left(\widehat{\mathbf{x}}_0\right)\right\|^2\right]+
		\frac{6l_3\Lambda^2 \alpha^2}{(m+1)(1-\alpha^2)} \sum_{i=0}^{m-1}q_i
		+ \frac{3(1-\alpha)}{1+\alpha}\sigma^2 \nonumber\\
		\mathop{\leq}^{(ii)} & \frac{6l_3\Lambda^2 \alpha^2}{(m+1)(1-\alpha^2)} \sum_{i=0}^{m-1}q_i
		+ \frac{3\sigma^2}{1+\alpha}(c_1+\frac{1}{c_1})\frac{1}{\sqrt{M(m+1)}},
		\label{grad inequ 4}
	\end{align}
	where the first inequality is from (\ref{hybrid var up bound}) and the definition of $l_3$, the inequality (i) is due to $(2-\alpha^{2i}-\alpha^{2(i-1)})\leq 2, 1-\alpha^2\leq 2$, and the inequality (ii) is due to $\mathbb{E}\left[\left\|\tilde{v}_0-\nabla h\left(\widehat{\mathbf{x}}_0\right)\right\|^2\right]\leq \frac{\sigma^2}{M}$ and (\ref{1-alpha}). By the choice of $r=\sqrt{\frac{2l_3 \Lambda^2 \alpha^2}{  (1-\alpha^2)}}$,  $\alpha=1-\frac{c_1}{\sqrt{M(m+1)}}, M=\sigma^{k_1}(m+1)^{k_2}, k_2 \in (0,1)$, we have 
	\begin{equation}
		\begin{cases}
			\frac{1-\alpha}{1+\alpha}&\leq 1- \alpha= \frac{c_1}{\sqrt{M(m+1)}}\\
			1-\alpha^2 & \geq  \frac{c_1}{\sqrt{M(m+1)}} \label{equ of alpha}.
		\end{cases}
	\end{equation}


	Using (\ref{final potential H 1}), definitions of $r,\alpha$ and the inequality (\ref{inequ of alpha}), there exists $c_3>0$ to drive 
	\begin{align}
		&\frac{6l_3\Lambda^2 \alpha^2}{(m+1)(1-\alpha^2)} \sum_{i=0}^{m}q_i = \frac{6l_3\Lambda^2 }{m+1}(-1+\frac{1}{1-\alpha^2}) \sum_{i=0}^{m}q_i \nonumber\\
		\leq& 6c_3 l_3\Lambda^2(\frac{1}{1-\alpha^2})  
		\left[  \frac{\mathbf{H}_{0}-\mathbf{H}_{m+1}}{m+1} + \frac{\sigma^2}{2r(1+\alpha)}(c_1+\frac{1}{c_1})\frac{1}{\sqrt{M(m+1)}} \right] \nonumber
		\\
		= &6c_3 l_3\Lambda^2(c_1+\frac{1}{c_1})\frac{1}{1-\alpha^2}\frac{\sigma^2}{2r(1+\alpha)}\frac{1}{\sqrt{M(m+1)}}  +6c_3l_3\Lambda^2 \frac{1}{1-\alpha^2}\frac{\mathbf{H}_0-\mathbf{H}_{m+1}}{m+1}\nonumber\\
		\mathop{\leq}^{(i)} &6c_3 l_3\Lambda^2(c_1+\frac{1}{c_1})\frac{1}{1-\alpha^2}\frac{\sigma^2}{2r(1+\alpha)}\frac{1}{\sqrt{M(m+1)}}   +6c_3l_3\Lambda^2 \frac{\sqrt{\frac{M}{m+1}}}{c_1}\left[\mathbf{H}_0-\mathbf{H}_{m+1}\right],\nonumber\\
		\mathop{\leq}^{(ii)}& 6c_3l_3\Lambda^2 \frac{\sqrt{\frac{M}{m+1}}}{c_1}\left[\mathbf{H}_0-\mathbf{H}_{m+1}\right]+
		\frac{6c_3 l_3\Lambda^2(c_1+\frac{1}{c_1})\sigma^2}{2c_1\Lambda\sqrt{2l_3}} \frac{\sqrt{1-\alpha^2}}{\alpha}\nonumber\\
		\mathop{\leq}^{(iii)} & 6c_3l_3\Lambda^2 \frac{\sqrt{\frac{M}{m+1}}}{c_1}\left[\mathbf{H}_0-\mathbf{H}_{m+1}\right]+\frac{6c_3 l_3\Lambda^2(c_1+\frac{1}{c_1})\sigma^2}{2\Lambda\sqrt{2l_3c_1}}\sqrt{\frac{2\big(\sqrt{M(m+1)}-c_1\big)}{\big(\sqrt{M(m+1)}-c_1\big)^2}}
		\label{middle term}
	\end{align}
	for the middle term $\frac{6l_3\Lambda^2 \alpha^2}{(m+1)(1-\alpha^2)} \sum_{i=0}^{m}q_i$ in (\ref{grad inequ 4}), where (i) is due to $1-\alpha^2 \geq  \frac{c_1}{\sqrt{M(m+1)}}$ in (\ref{equ of alpha}), (ii) is due to  $\frac{1}{1+\alpha}\leq1$, $1-\alpha^2 \geq  \frac{c_1}{\sqrt{M(m+1)}}$ in in (\ref{grad inequ 4}), and the definition of $r :=\sqrt{\frac{2l_3 \Lambda^2 \alpha^2}{  (1-\alpha^2)}}$.  The last inequality (iii) is due to  $\alpha=1-\frac{c_1}{\sqrt{M(m+1)}}$, and $\frac{\sqrt{1-\alpha^2}}{\alpha}=\sqrt{\frac{1}{\alpha^2}-1}=\sqrt{\frac{c_1\big(2\sqrt{M(m+1)}-c_1\big)}{(\big(\sqrt{M(m+1)}-c_1\big))^2}}$.
	
	
	Using (\ref{st 2.47}), (\ref{exp phi}), (\ref{middle term}), and $k_2 \in (0,1)$, we drive
	\begin{align}
		\lim _{m \rightarrow \infty} \left[ \frac{3 }{m+1}\sum_{t=0}^{m} \mathbb{E}\left[ \left\|\  
		\nabla h\left(\widehat{\mathbf{x}}_t\right) -\tilde{v}_t 
		\right\|^2 \right]\right]=0. \label{grad of h}
	\end{align}
	
	Inserting (\ref{st}), (\ref{exp phi}) and (\ref{grad of h}) into (\ref{grad of phi 3}), we finally obtain 
	\begin{align}
		\lim _{m \rightarrow \infty} \frac{1}{m+1}\sum_{t=0}^{m}\mathbb{E}\left[ \left\|\ \nabla \Phi\left(\widehat{\mathbf{x}}_t\right) \right\|^2 \right]=0,	
	\end{align}
	which implies 
	\begin{align}
		\lim _{m \rightarrow \infty}	\mathbb{E} \left\|\ \nabla \Phi\left(\widehat{\mathbf{x}}_{\bar{m}}\right) \right\|^2=0.
	\end{align}
	
	Since $\nabla \Phi(\mathbf{x})=\nabla_x \mathcal{L}_\beta\left(\mathbf{x}, \mathbf{y}^{k+1}, \lambda^k\right)+\beta \mathcal{D}_x^k\left(\mathbf{x}-\mathbf{x}^k\right)$ and $\lim _{t \rightarrow \infty} \mathbb{E}
	\left[\nabla \Phi\left({\mathbf{x}}_{\bar{m}}\right)\right]=0,$ the inexact optimal condition (\ref{x update criteria}) will be satisfied by setting $\widehat{\mathbf{x}}^k=\widehat{\mathbf{x}}_{\bar{m}}$ for all $t$ sufficiently large.
\end{proof}

\subsection{Additional Experiments and Discussions}\label{add for exps}

This section provides specific details and additional information regarding the experiments. 
All experiments utilize a dual variable step size of $s=1.2$, AL function penalty parameter $\beta=1.01$, and the same initial $\mathbf{x}^{0}$ sampled from the standard normal distribution. 
All algorithms are implemented in Python, and the experiments are conducted on a PC equipped with an Intel i7-12700F CPU and 16GB of memory.

For comparison, we implement the following algorithms:
\begin{itemize}
	\item
	SADMM with contant or scheduled decaying step sizes as $\eta_k:=\frac{\eta_0}{1+\eta^{\prime}\lceil k / N\rceil}$ is implemented. 
	$\eta^{\prime}=0$ degenerates into the constant-type step-size. 
	Based on experimental results, we set the parameter to $\eta_0=0.05$, $\eta^{\prime}=1.0$, gradient batch $b=\lceil\sqrt{N}\rceil$.
	
	\item We also implemented SPIDER-ADMM, based on SARAH recursive gradient estimation, and SVRG-ADMM. It is worth noting that \cite{huang2019faster} has already compared SPIDER-ADMM with various other variance reduction algorithms; thus, we skip comparing it with SAGA-type ADMM algorithms here. For SPIDER-ADMM, we use a constant step size $\eta=\frac{1}{2L}$ with a batch size $b=\lceil\sqrt{N}\rceil$. 
	SVRG-ADMM utilizes a step size of $\eta=\frac{1}{3L}$ with a batch size $b=\lceil{N}^{\frac{2}{3}}\rceil$.
	
	\item
	The proposed AH-SADMM simultaneously incorporates momentum acceleration techniques and hybrid gradient techniques. The algorithm uses a step size of $\eta=\frac{1}{2L}$,with a specified batch size $b=\lceil{N}^{\frac{1}{3}}\rceil$. The selection of hybrid gradient parameter $\alpha$ follows \eqref{alpha setting}. The momentum parameter $\tau$ is chosen as $\tau=0.8$.

	\item
	The variant algorithm of AH-SADMM, named H-SADMM, omits the use of momentum techniques and solely employs hybrid gradient techniques. The algorithm uses a step size of 
	$\eta=\frac{1}{2L}$, with a batch size $b=\lceil{N}^{\frac{1}{3}}\rceil$. The selection of hybrid gradient parameter $\alpha$ still follows \eqref{alpha setting}.
	
	\item
	Another variant of AH-SADMM, named ASADMM, exclusively utilizes momentum acceleration techniques without employing hybrid gradient techniques. ASADMM shares the same step size and batch as H-SADMM. The momentum parameters are set identical to AH-SADMM.

\end{itemize}

\subsubsection{Nonconvex Binary Classification problem}

The nonconvex SCAD penalty $p_\kappa(\cdot)$ is defined as
\begin{equation}\label{def of SCAD}
	p_\kappa(\theta):= \begin{cases}\kappa \theta, & \theta \leq \kappa, \\ \frac{-\theta^2+2 c \kappa \theta-\kappa^2}{2(c-1)}, & \kappa<\theta \leq c \kappa, \\ \frac{(c+1) \kappa^2}{2}, & \theta>c \kappa,\end{cases}
\end{equation}
where $c>2$ and $\kappa>0$ are the knots of the the quadratic spline function.

In particular, it is well-known that the minimization problem of the form $\min _{\mathbf{y}} \sum_{i=1}^n p_\kappa\left(\left|\mathbf{y}_i\right|\right)+\frac{1}{2 v}\|\mathbf{y}-\mathbf{q}\|^2$ with $1+v \leq c$ has a closed form solution
\begin{equation}\label{closed form of SCAD}
	y_i:= \begin{cases}\operatorname{sign}\left(q_i\right) \max \left\{\left|q_i\right|-\kappa v, 0\right\}, & \left|q_i\right| \leq(1+v) \kappa \\ \frac{(c-1) q_i-\operatorname{sign}\left(q_i\right) c \kappa v}{c-1-v}, & (1+v) \kappa<\left|q_i\right| \leq c \kappa \\ q_i, & \left|q_i\right|>c \kappa .
	\end{cases}
\end{equation}

In the experiment for the nonconvex binary classification problem, we set the parameters in the SCAD function to $(c, \kappa)=(3.7,0.1)$.

\subsubsection{Graph-Guided Binary Classification}
The famous graph-guided fused Lasso model \cite{kim2009multivariate} can be represented as
\begin{equation}\label{equ fused Lasso}
	\min_{\mathbf{x}}\frac{1}{n}\sum^{n}_{i=1} f_{i}(\mathbf{x})+\lambda_1 \|A\mathbf{x}\|_{1},
\end{equation}
where the nonconvex sigmoid loss function $f_{i}$, the set of training samples $\left\{\left(a_i, b_i\right)\right\}_{i=1}^{N}$, matrix $A$, and the introduction of linear constraints follow the same approach as the first problem \eqref{equ log with SCAD}.
The difference lies in the regularizer for this problem, which is the $l_1$-norm.

\begin{figure}[h] 
	\vskip -0.1in
	\centering
	\subfigure[ijcnn1]{\includegraphics[width=0.232\textwidth]{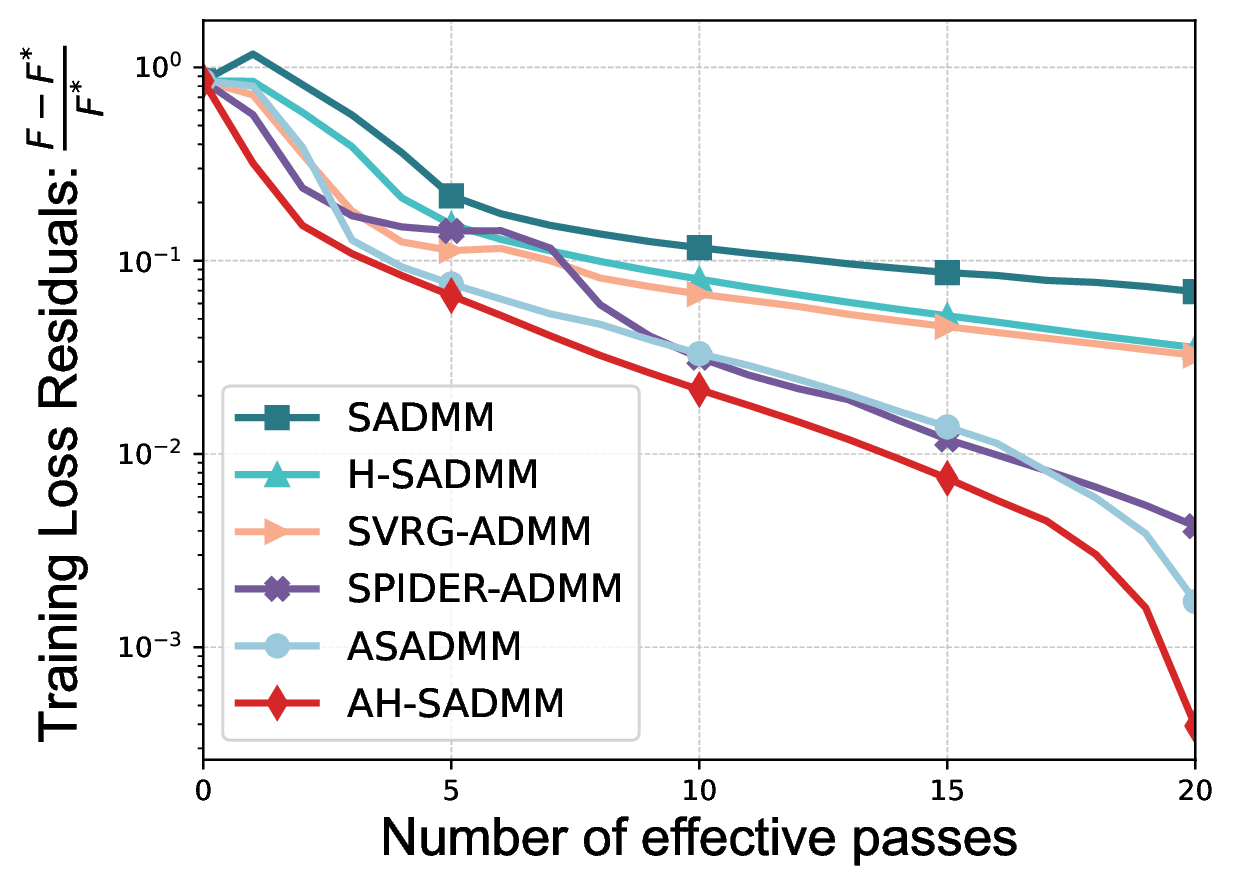}} 
	\subfigure[w8a]{\includegraphics[width=0.232\textwidth]{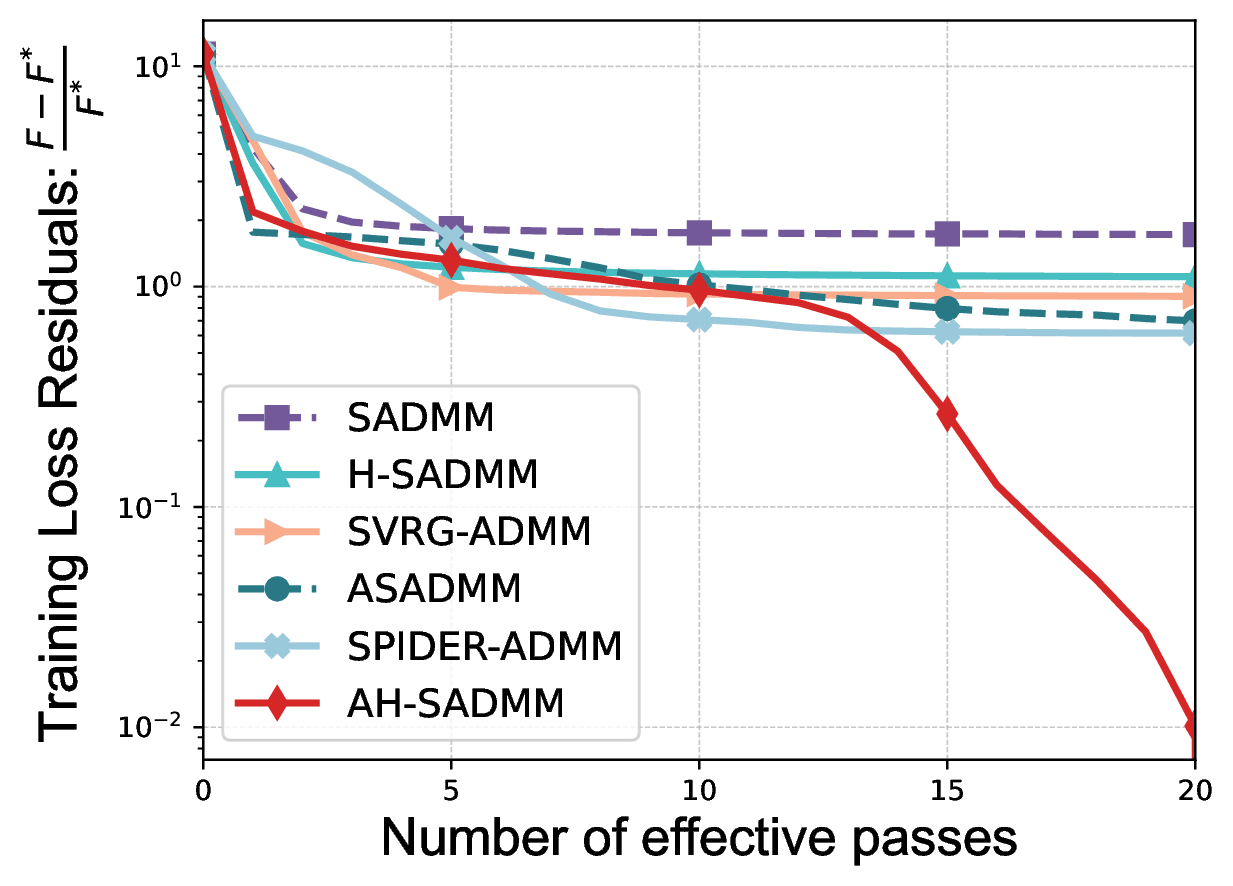}}
	\subfigure[covtype.binary]{\includegraphics[width=0.232\textwidth]{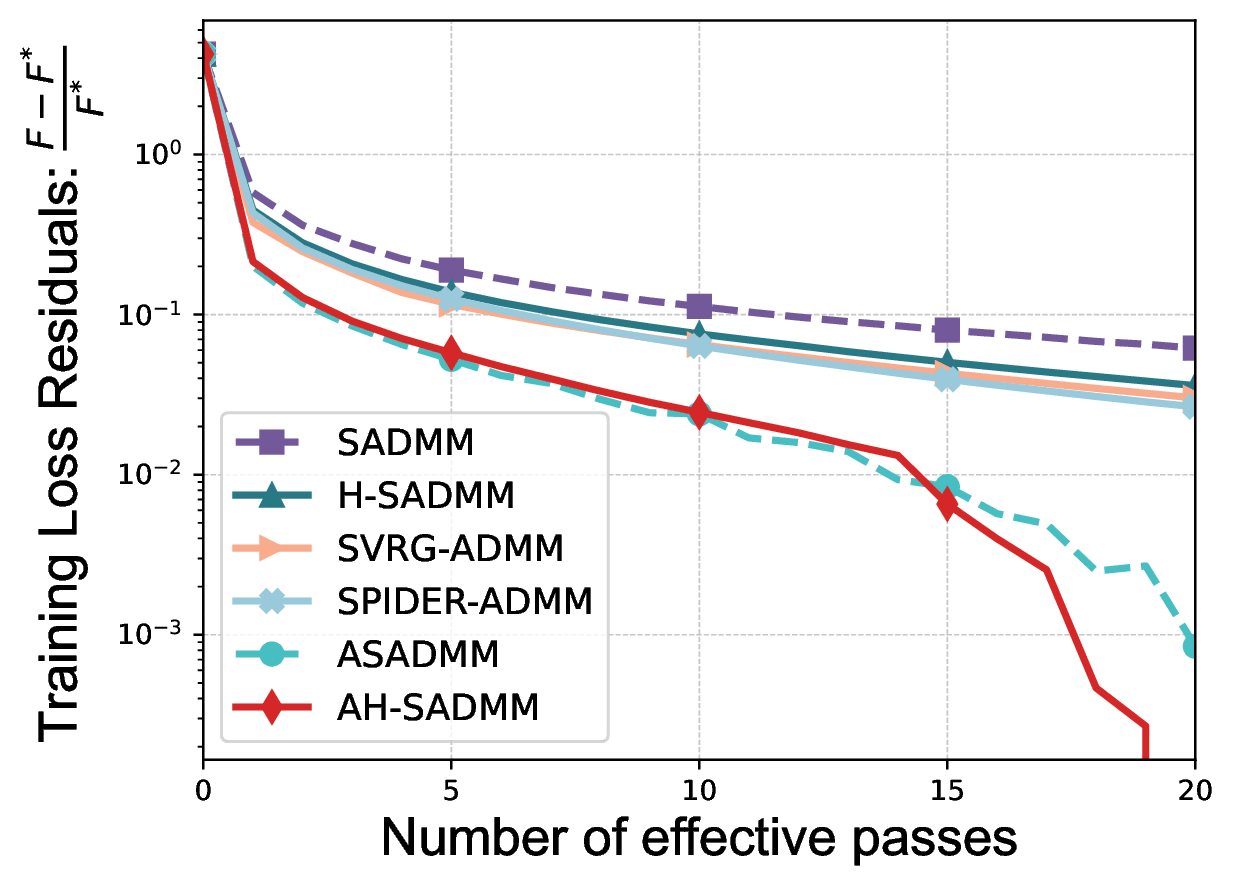}}
	\subfigure[a9a]{\includegraphics[width=0.232\textwidth]{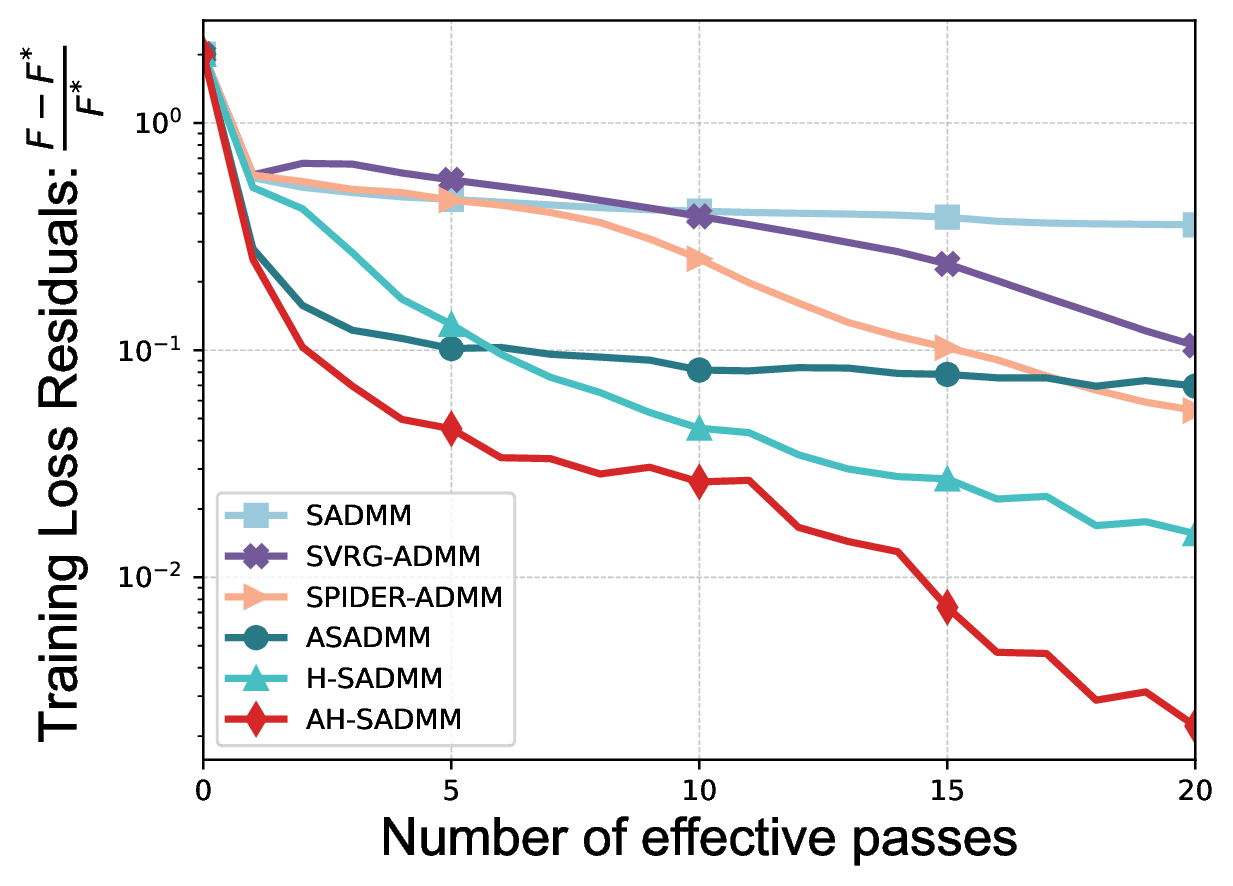}}
	\caption{The training loss of \eqref{equ fused Lasso} on some real datasets.}
	\label{fig_l1:ijcn1}
	\vskip -0.1in
\end{figure}

We continue to validate the effectiveness of our algorithm by addressing the graph-guided fused Lasso problem on several publicly available datasets from LIBSVM. \cref{fig_l1:ijcn1} illustrates that variance reduction algorithms still outperform SADMM without variance reduction techniques. Additionally, algorithms utilizing acceleration techniques, such as ASADMM and AH-SADMM, exhibit improved convergence. Moreover, our algorithm combines the advantages of both variance reduction and acceleration algorithms, achieving faster and better convergence.



\end{document}